%% file: measureNN.tex
\documentclass[11pt]{amsart}
\usepackage[utf8]{inputenc}
\usepackage{fullpage}

\usepackage{latexsym,amsmath,amsfonts,mathtools}
\usepackage[colorlinks,urlcolor=blue,linkcolor=blue,citecolor=green]{hyperref}
\usepackage{xcolor}

\usepackage{tikz,graphicx}
\usetikzlibrary{decorations.pathreplacing}

\usepackage[T1]{fontenc}
\usepackage{geometry}
\usepackage[sc,osf]{mathpazo}

\newtheorem{theorem}{Theorem}
\newtheorem{fact}{Fact}
\newtheorem{proposition}[theorem]{Proposition}
\newtheorem{corollary}[theorem]{Corollary}
\newtheorem{lemma}[theorem]{Lemma}

\newtheorem{assumption}[theorem]{Assumption}
\newtheorem{definition}[theorem]{Definition}

\newtheorem{claim}{Claim}

\newtheorem{remark}[theorem]{Remark}

\newtheorem{problem}{Problem}

\numberwithin{theorem}{section}
\numberwithin{figure}{section}

\numberwithin{equation}{section}

\newcommand{\eqnref}[1]{(\ref{eq:#1})} 

\def\be{\begin{equation} }
\def\ee{ \end{equation}}

\def\ben{\begin{equation*}}
\def\een{\end{equation*}}
\def\bea{\begin{eqnarray}}
\def\eea{\end{eqnarray}}
\def\ee{\end{eqnarray}}
\def\bean{\begin{eqnarray*}}
\def\eean{\end{eqnarray*}}

\newcommand\ignore[1]{}

\def\R{\mathbb{R}} 
\def\Q{\mathbb{Q}} 
\def\N{\mathbb{N}} 

\newcommand{\dotprod}[1]{\left\langle #1\right\rangle}
\newcommand{\abs}[1]{\left|#1\right|}
\newcommand{\norm}[1]{\left\|#1\right\|}
\newcommand{\pnorm}[2]{\norm{#2}_{#1}}
\newcommand{\opnorm}[1]{\pnorm{\text{Op}}{#1}}
\newcommand{\supnorm}[1]{\pnorm{\infty}{#1}}
\newcommand{\lipnorm}[1]{\pnorm{\text{Lip}}{#1}}
\newcommand{\parn}[1]{\left(#1\right)}
\newcommand{\Lnorm}[1]{\norm{#1}_{(L)}}

\renewcommand{\Pr}[1]{\mathbb{P}\left(#1\right)} 
\newcommand{\Ex}[1]{\mathbb{E}\left[#1\right]} 
\newcommand{\Exp}[2]{\mathbb{E}_{#1}\left[#2\right]} 

\def\iid{\stackrel{\makebox[0pt]{\mbox{\normalfont\sffamily\tiny iid}}}{\sim}}

\newcommand{\bigoh}[1]{\sO\left(#1\right)}

\def\sB{\mathcal{B}}
\def\sF{\mathcal{F}}
\def\sG{\mathcal{G}}\def\sH{\mathcal{H}}

\def\sM{\mathcal{M}}\def\sN{\mathcal{N}}\def\sO{\mathcal{O}}
\def\sQ{\mathcal{Q}}
\def\sT{\mathcal{T}}\def\sU{\mathcal{U}}

\newcommand\QED{\ifhmode\allowbreak\else\nobreak\fi
\quad\nobreak$\Box$\medbreak}

\def\eps{\varepsilon}

\def\Coup{\text{Coup}}

\def\Gsoln{G^{\star}}
\def\dist{{\rm dist}}
\def\Grad{\overline{\mbox{Grad}}}
\def\Gradh{\widehat{\mbox{Grad}}}

\def\M{\text{Mrt}}
\def\yhat{\widehat{y}_N}

\def\btheta{{\boldsymbol \theta}}
\def\bthetabar{\overline{\boldsymbol \theta}}
\def\bthetatil{\tilde{\boldsymbol \theta}}
\def\thetabar{\overline{\btheta}}
\def\thetatil{\tilde{\btheta}}
\def\Thetabar{\overline{\Theta}}

\def\Lbar{\overline{\mathcal{L}}}
\def\zbar{\overline{z}}
\def\ybar{\overline{y}}
\def\Mbar{\overline{M}}
\def\gbar{\overline{\Gamma}}

\def\L{\mathcal{L}}

\def\Probspace{\mathcal{P}}
\def\Probspecial{\mathcal{P}^{\star}}

\def\soln{\mu^{\star}}

\newcommand{\fsigma}[2]{\sigma^{(#1)}\parn{#2}}

\newcommand{\boldj}[2]{\boldsymbol{j}_{#1}^{#2}}

\newcommand{\boldN}[2]{\boldsymbol{N}_{#1}^{#2}}

\begin{document}
\title{A mean-field limit for certain deep neural networks}
\author{Dyego Ara\'{u}jo, Roberto I. Oliveira \& Daniel Yukimura }
\thanks{IMPA, Rio de Janeiro, Brazil. \texttt{\{dyego-eu,rimfo,yukimura\}@impa.br}. RIO is supported by CNPq grants 432310/2018-5 (Universal) and 309569/2015-0 (Produtividade em Pesquisa); and FAPERJ grant 224687/2016 (Cientista do Nosso Estado). This research was supported in part by the National Science Foundation under Grant No. NSF PHY-1748958 to the Kavli Institute of Theoretical Physics.}


\begin{abstract}Understanding deep neural networks (DNNs) is a key challenge in the theory of machine learning, with potential applications to the many fields where DNNs have been successfully used.

This article presents a scaling limit for a DNN being trained by stochastic gradient descent. Our networks have a fixed (but arbitrary) number $L\geq 3$ of inner layers; $N\gg 1$ neurons per layer; full connections between layers; and fixed weights (or "random features" that are not trained) near the input and output.  

Our results describe the evolution of the DNN during training in the limit when $N\to +\infty$, which we relate to a mean field model of McKean-Vlasov type. Specifically, we show that network weights are approximated by certain "ideal particles"  whose distribution and dependencies are described by the mean-field model. A key part of the proof is to show existence and uniqueness for our McKean-Vlasov problem, which does not seem to be amenable to existing theory.

Our paper extends previous work on the $L=1$ case by Mei, Montanari and Nguyen; Rotskoff and Vanden-Eijnden; and Sirignano and Spiliopoulos. We also complement recent independent work on $L>1$ by Sirignano and Spiliopoulos (who consider a less natural scaling limit) and Nguyen (who nonrigorously derives similar results).

\smallskip
\noindent {\bf \keywordsname:} deep neural networks; McKean-Vlasov process; scaling limits.
\end{abstract}

\maketitle
\pagebreak

\tableofcontents

\section{Introduction}
\label{sec:intro}
\input{measureNN_aux/1-intro}

\section{Background and related work}\label{sec:background}
\input{measureNN_aux/1.5-background}

\section{Setup for rigorous results}
\label{sec:setup}
\input{measureNN_aux/2-setup_for_results}

\section{Mean-field picture for the weights}
\label{sec:McKeanVlasov}
\input{measureNN_aux/3-mkeanvlasov}

\section{Statement of rigorous results}
\label{sec:rigorous_results}
\input{measureNN_aux/4-rigorous_results}

\section{Proof overview}\label{sec:overview}
\input{measureNN_aux/5-overview}

\section{Simple properties of the McKean-Vlasov problem}
\label{sec:apriori_simple}
\input{measureNN_aux/6-apriori_simple}

\section{Potential solutions of the McKean-Vlasov problem}
\label{sec:apriori_space}
\input{measureNN_aux/7-apriori_space}

\section{Existence and uniqueness for the McKean-Vlasov problem}
\label{sec:existence}
\input{measureNN_aux/8-existence}

\section{From SGD to continuous-time gradient flow}
\label{sec:SGDtoCTGD}
\input{measureNN_aux/9-SGD_to_CTGD}

\section{Ideal particles and the gradient flow}
\label{sec:ideal}
\input{measureNN_aux/10-ideal_particles}

\section{Loss approximation and ideal particles}
\label{sec:coupling}
\input{measureNN_aux/11-coupling}

\section{Conclusion}
\label{sec:conclusion}
\input{measureNN_aux/12-conclusion}

\appendix

\section{Some technical lemmas}
\label{sub:technicalities}
\input{measureNN_aux/A1-technicalities}

\section{Estimates on Lipschitz constants}
\label{sub:lipshitz}
\input{measureNN_aux/A2-lipshitz}

\section{The set of special probability measures is closed}
\label{sec:Probspecialisclosed}
\input{measureNN_aux/A3-specialmeasures}

\bibliography{measureNN}
\bibliographystyle{plain}

\end{document}

%% file: measureNN_aux/1-intro.tex
Deep neural networks (also known as deep nets or DNNs) and the "deep learning" methods they implement have revolutionized Machine Learning \cite{Goodfellow2016,lecun2015}. They are behind recent spectacular successes in machine translation, pattern recognition, and two-person game playing, to name but a few fields. Unfortunately, we are far from understanding theoretically what makes DNNs work. In fact, even describing the {\em evolution of a DNN during training} is quite difficult.

Large neural networks are high-dimensional disordered objects with many interacting components. It is natural to use methods from Statistical Mechanics to describe them \cite{Engel2001} . The main goal of this paper is to give a scaling limit for the evolution of certain deep neural networks during training via a {\em mean field model} inspired by Statistical Mechanics. That is, we show that, for a large network, the parameters of a large DNN behave like random particles interacting with each other and with their own probability densities in the way prescribed by a certain McKean-Vlasov process. This will require the following assumptions: complete connections between layers, plus "random features"~that are not learned on the first and last layers of weights. The latter assumption does not seem to be crucial, as we explain below (see \S \ref{sub:intro:contribution} and \S \ref{sub:comparison}). 

Our result extends work done for simpler "shallow" neural networks by Mei, Montanari and Nguyen \cite{Mei2018,MeiArxiv2018} (see also \cite{Mei2019}); Rotskoff and Vanden-Eijnden \cite{Rotskoff2018}; and Sirignano and Spiliopoulos \cite{Sirignano2018,Sirignano2018.2}. A recent (and independent) preprint by Sirignano and Spiliopoulos \cite{Sirignano2019} also gives a limiting description of deep networks. As we will see, that paper studies a different sense of limit: a DNN has "layers of neurons," and \cite{Sirignano2019} considers what happens when layers grow one at a time. By contrast, our methods allow us to consider networks all of whose layers are large, which seems more natural. Our work is closely related to the nonrigorous (and independent) results of Nguyen \cite{Nguyen2019}.

The paper is organized as follows. Section \ref{sec:background} gives presents the context for our result, including related work. Our own work starts in Section \ref{sec:setup}, we present our problem setup. Section \ref{sec:McKeanVlasov} motivates and describes the McKean-Vlasov process used to find the scaling limit of DNNs. The main theorems are stated in Section \ref{sec:rigorous_results}, where we offer additional comparisons with related work. Proofs are outlined in Section \ref{sec:overview} and presented in Sections \ref{sec:apriori_simple} to \ref{sec:coupling}. Some extensions and open problems are described in the Conclusion (Section \ref{sec:conclusion}).\\

\noindent{\bf Notation:} Vectors in Euclidean space $\R^d$ treated as column vectors. We use $|\cdot|$ to denote th Euclidean norm over any such space. The operator norm on the space $\R^{d\times k}$ of $d\times k$ matrices is denoted by $\|\cdot\|$. The sup norm of a continuous function $f:[0,T]\to\R^d$ is denoted by $\|\cdot\|_{\infty}$.

For integers $\ell\leq r$, $[\ell:r]$ denotes the set of all integers $\ell\leq k\leq r$. 

Given a Polish space $(M,\rho)$, we will use $\Probspace(M)$ to denote the set of probability measures over the Borel $\sigma$-field of $(M,\rho)$. Given $\mu\in \Probspace(M)$, we write $Z\sim \mu$ to say that $Z$ is a random element of $M$ with law $\mu$.

%% file: measureNN_aux/1.5-background.tex
This section provides some context for our results and how they relate to the existing literature. In \S \ref{sub:intro:supervised}, we describe the Statistical Learning formulation of supervised learning with deep neural networks. In \S \ref{sub:intro:previouswork}, we discuss the work on "shallow" networks and how it relates to McKean-Vlasov processes. In \S \ref{sub:intro:contribution}, we outline our main results. In \S \ref{sub:intro:background}, give some additional pointers to related literature.  

\subsection{Supervised learning with deep neural networks.}\label{sub:intro:supervised} At a high level, a DNN is just one special type of "parametric function" \[\yhat:\R^{d_X}\times \R^{p_N}\to \R^{d_Y}.\] The interpretation of such a function is that, for each choice of "parameters" or "weights" $\btheta_N\in\R^{p_N}$, $\yhat(\cdot,\btheta_N)$ is a function mapping inputs $x\in\R^{d_X}$ to outputs $y=\yhat(x,\btheta_N)\in\R^{d_Y}$. The number $N$ may be thought of as a measure of the "size" of the network, which is related to the dimension $p_N$ of parameter space. We are especially interested in $N\gg 1$, $p_N\gg 1$. A full definition of the DNNs we consider is given in Section \ref{sec:setup}.

In this paper, we consider DNNs in the context of supervised learning with mean-square loss. One is given "training data" in the form an i.i.d. sample from a distribution $P$ over $(X,Y)\in \R^{d_X}\times \R^{d_Y}$. The goal is to use the data to choose a weight vector $\btheta_N^\star$ that makes the loss (or energy) function
\[\L_N(\btheta):=\frac{1}{2}\Exp{(X,Y)\sim P}{|Y-\yhat(X,\btheta_N)|^2}\]
 as small as possible. In the jargon of the area, minimizing $\L_N$ means that the DNN {\em generalizes well} to unseen data coming from the same distribution as the training data \cite{Anthony2009}.

The trouble with DNNs comes from their sheer size and complexity. The seminal "AlexNet" DNN  \cite{Krizhevsky2017} had $p\approx 10^8$ parameters; recent ones are much larger. In addition, the functions computed in DNNs are complicated layers of compositions of nonlinear functions. When deep nets are trained by gradient-descent type procedures, as is usually the case, many questions can be raised. Does gradient descent typically converge to a global minimum \cite{Venturi2018}, and, if so, does it converge in reasonable time \cite{Du2018,Du2018.2}? How does the shape of the "energy landscape" relate to generalization versus "overfitting" (or finding spurious patterns in data) of the DNN \cite{Zhang2016}? In spite of many recent results, we still lack satisfying answers to these questions.

\subsection{McKean-Vlasov scaling limits}\label{sub:intro:previouswork} Given the above difficulties, it is reasonable to take a step back and ask the simpler question: what is a deep neural network actually doing when it is trained?

Statistical Mechanics offers a natural way to approach large neural networks, be they deep or shallow. Physicists and mathematicians have been using this perspective since the original boom of neural nets; see the books of Engel and van den Broeck \cite{Engel2001} and M\'{e}zard and Montanari \cite{Mezard2009} for more information.

A more recent development has been to describe neural networks being trained to mean-field models of McKean-Vlasov type \cite{McKean1967,Sznitman1991}. As noted above, the papers \cite{Mei2018,Rotskoff2018,Sirignano2018} obtained results of this kind in the case of "shallow" networks, which are simpler than deep nets. 

Such networks contain $N\gg 1$ weights of dimension $d$, and the total dimension of the space of parameters is $p=p_N=Nd\gg 1$. An element $\btheta_N\in \R^{p_N}$ takes the form \[\btheta_N = (\theta_i)_{i=1}^N\mbox{ with each }\theta_i\in\R^d.\]
Assume the DNN is initialized with values $\theta_i$ that are i.i.d. with law $\mu_0$. Suppose also that we train the network by gradient descent in the following sense:
\[\frac{d}{dt}\btheta(t)= -N\nabla\,\L_N(\btheta(t)).\]
(In truth, the network is trained by {\em stochastic} gradient descent, but we ignore the distinction for the time being.)

The authors of \cite{Mei2018,Rotskoff2018,Sirignano2018}  noted that, for shallow networks, one can rewrite this evolution as
\begin{equation}\label{eq:towardsMcKV}\frac{d}{dt}\theta_i(t) = -b(\theta_i(t),\widehat{\mu}_t), 1\leq i\leq N,\end{equation}
where $b$ is a drift term depending on the value of the weight and on the empirical measure of the weights at time $t$:
\[\widehat{\mu}_t:=\frac{1}{N}\sum_{i=1}^N{\boldsymbol\delta}_{\theta_i(t)}.\]
A natural {\em ansatz} is that for large $N$, $\widehat{\mu}_t$ approaches a deterministic measure $\soln_t$ for each $t\geq 0$. If we had $\soln_t=\widehat{\mu}_t$, then the weight trajectories $\theta_{i}$ would be decoupled and thus i.i.d. (as the $\theta_i(0)$ are i.i.d.). So we expect that for large $N$ the weights are nearly i.i.d. Under this ansatz, we also have that for a random index $i$, $\theta_i(t)$ follows $\widehat{\mu}_t\approx \soln_t$. Therefore, one may conjecture that, in the thermodynamic limit, the $\theta_i$ approach i.i.d. particles satisfying
\[(\star)\;\left\{\begin{array}{ll}\thetabar(0)\sim  \mu_0,\\ \dfrac{d\thetabar}{dt}(t) = -b(\thetabar(t),\soln_t), & t\geq 0,\\ \thetabar(t) \sim \soln_t, & t\geq 0.\end{array}\right.\]

Problem $(\star)$ describes a particle $\thetabar(t)$ starting from a random initial state which interacts with its own distribution $\soln_t$. Under fairly reasonable conditions on the function $b$, there exists a unique probability measure $\soln$ over continuous trajectories such that, if $\thetabar(\cdot)\sim \soln$, then $(\star)$ is satisfied, and the $\soln_t$ are the time-$t$ marginals of $\soln_t$. In this paper, measures $\soln$ obtained as solutions to problems of the form $(\star)$ are called McKean-Vlasov measures, although this name is reserved for the case where the evolution of $\thetabar$ includes a diffusion term. More information about McKean-Vlasov problems and their history can be found in the references given in \S \ref{sub:intro:background:McKV}.

If there is a unique McKean-Vlasov measure satisfying $(\star)$, one may then give a coupling argument to show that the trajectories $\theta_i$ are close to an i.i.d. sample of $\soln$. This property is called "propagation of chaos" \cite{Sznitman1991,Kac1959} and is used directly in \cite{MeiArxiv2018}. The densities $p(t,x)$ of the measures $\soln_t$ (if they exist) are weak solutions of the PDE:
\[\partial_t p(t,x) = \nabla_x\cdot (b(x,\soln_t)\,p(t,x)).\]
Montanari, Mei and Nguyen \cite{Mei2018,MeiArxiv2018} have used properties of PDE in the case of shallow learning to study the generalization abilities of shallow networks.

\subsection{Deep neural networks and our contribution}\label{sub:intro:contribution} The main result of this paper, described in sections \ref{sec:setup} to \ref{sec:rigorous_results}, extends the analysis of  \cite{Mei2018,Rotskoff2018,Sirignano2018} to deep neural networks. More precisely, we consider a specific class of DNNs with complete connections and with first and last layers as "frozen random features" (the role of these is elucidated in Remark \ref{rem:timescales} below). These networks consist of units organized in layers. We consider what happens when the DNN is trained by a version of stochastic gradient descent.

As in the papers discussed above, we find that, when $N$ is large, the weights of the DNN are close to certain "ideal particles" derived from a McKean-Vlasov process. However, the picture for DNNs is much more complicated than that of shallow networks. The asymptotic weight distributions are layer-dependent, and weights do {\em not} become independent in the large-$N$ limit. In fact, the fundamental units of our analysis are not individual weights, but rather {\em paths of weights} in the network that go from the input to the output. The specific dependency structure of weights along a path is essential in the analysis. We will also see that the drift term for our McKean-Vlasov process will depend on the solution $\soln$ in a discontinuous fashion, which means that proving existence and uniqueness is a nontrivial task.

Two recent independent papers of Sirignano and Spiliopoulos \cite{Sirignano2019} and Nguyen \cite{Nguyen2019} obtain results related to our work. Reference \cite{Nguyen2019} uses a nonrigorous ansatz to obtain essentially the same process that we do. One difference comes from the fact that \cite{Nguyen2019} uses different scalings for the updates in different layers, in order to avoid the issues discussed in Remark \ref{rem:timescales}.

By contrast, \cite{Sirignano2019} considers a different scaling limit: the number of units in different layers do not all diverge simultaneously, but rather one at a time. The authors of \cite{Sirignano2019} note that letting all layers diverge at the same time brings a host of technicalities. Our main contribution is to tackle these technicalities head on, which will be important for future progress in the field. In fairness, we note that our assumptions ("random features", differentiable activation functions) leave something to be desired, though some of these restrictions may be lifted. A more thorough comparison between our results and those of \cite{Sirignano2019,Nguyen2019} is given in \S \ref{sub:comparison}.

The existence of a limiting process may also contribute to the analysis of learning in the class of networks we study. This program has been partially carried out by Mei et al. \cite{Mei2018} for shallow networks, where the limiting PDEs describe so-called Wasserstein gradient flows in the space of probability measures \cite{Ambrosio2005}. Our limiting process leads to a PDE family that we believe is new and deserving of further study.

\subsection{Additional background}\label{sub:intro:background} We present here a concise and partial overview of the work related to our paper.

\subsubsection{Neural networks, deep nets and generalization} The field of neural network applications is simply too large to survey adequately. Literally dozens of arXiv preprints on the subject appear every week. The landmark "AlexNet"~paper on image recognition \cite{Krizhevsky2017}, the survey \cite{lecun2015} and the book \cite{Goodfellow2016} serve as reasonable introductions to this field. Generalized Adversarial Networks or GANs \cite{goodfellow2014} are one example of DNNs that fall outside the scope of supervised learning.

Statistical Learning Theory was developed to explain generalization in much more general contexts than deep learning. Much of that theory is predicated on assuming that the algorithm under consideration cannot interpolate the training data \cite{Anthony2009}. However, practical applications feature "over parameterized"~deep nets, with $p_N$ much larger than the sample size. In this case, interpolation may well happen \cite{Zhang2016}. This has generated a flurry of related activity on statistical methods that do interpolate: see \cite{belkin2018} for a recent example.

\subsubsection{Other scaling limits for neural networks} We described above McKean-Vlasov limits for deep and shallow networks. Natural Tangent Kernels have been put forward by Jacot, Gabriel and Hongler \cite{Jacot2018} as scaling limits of deep networks when the number of neurons diverges, albeit under a different time scale. Other papers have followed up on this idea \cite{Geiger2019}. However, Chizat and Bach \cite{Chizat2018} argue that this stems from small step sizes and an effective linearization of the loss function around the starting point of the optimization. This idea is also explored by Mei, Misiakiewicz and Montanari \cite{Mei2019}.

\subsubsection{Mean-field models, propagation of chaos and related topics}\label{sub:intro:background:McKV} The kind of models discussed in \S \ref{sub:intro:previouswork} come from a body of work on nonlinear Markov processes, mean-field particle models, and kinetic theory. Here we give only a few references to the area.

McKean's seminal paper \cite{McKean1967} introduces McKean-Vlasov processes. More general nonlinear Markov processes are surveyed in Kolokoltsov's book \cite{Kolokoltsov2010}.  Spohn's book \cite{Spohn2012} covers kinetic theory, Vlasov's equation, and related topics. The term "propagation of chaos" was a concept introduced by Kac \cite{Kac1959} in the context of kinetic theory and made central in Sznitman's St Flour lectures \cite{Sznitman1991}. Rachev and R\"{u}schendorf  \cite{Rachev2006} present methods for proving existence and uniqueness for McKean-Vlasov problems. We will adapt their method to our setting, but this will require circumventing certain discontinuities.

%% file: measureNN_aux/2-setup_for_results.tex
This section describes the sort of deep neural networks and regression task considered in this work.

\subsection{Functions and losses}\label{sub:loss}
Our problem centers around understanding the following regression task:

\begin{problem}[Regression Task]
\label{prob:regression-task}
Let $(X,Y)\in\R^{d_X}\times \R^{d_Y}$ be a random pair endowed with an unknown distribution $P$. The goal is to find a function $f:\R^{d_X}\to \R^{d_Y}$ such that the mean-squared loss,
\[
	\L(f):=\frac{1}{2}\Exp{(X,Y)\sim P}{\abs{Y-f(X)}^2},
\]
is as small as possible.
\end{problem}
In our case, the possible functions $f$ are constrained by the architecture of a Deep Neural Network. This consists of a family of parametric functions indexed by $N$,
\[
	\yhat:\R^{d_X}\times \R^{p_N}\to \R^{d_Y},
\]
where $\R^{p_N}$ is a space of parameters ~$\btheta_N$. In this setting, the regression problem becomes an optimization task over $\R^{p_N}$, with cost function given by
\begin{equation}
\label{eq:defLN}
	\L_N(\btheta_N):=\frac{1}{2}\Exp{(X,Y)\sim P}{|Y-\yhat(X,\btheta_N)|^2}.
\end{equation}

\subsection{The neural network.}\label{sub:theNN} Our parametric functions will be defined as follows.

We will consider networks with $L\geq 3$ hidden layers, where each hidden layer $\ell$ contains $N_\ell$ neurons. To set up the hidden layer functions, we choose dimension parameters:

\noindent{
\begin{tabular}{p{3cm}p{1.8cm}p{1.8cm}p{1.6cm}p{1.8cm}p{3cm}}
$d_0:=d_X, D_0,$&$d_1, D_1,$&$d_2, D_2,$&$\cdots$&$d_L, D_L,$&$d_{L+1}:=d_Y,$
\end{tabular}}
where $d_\ell$ represents the dimension of the input and $D_\ell$ represents the dimension of the parameters for the hidden layer function in the $\ell^\text{th}$ layer. In our setting, the function in the $(L+1)^\text{th}$ (last) layer do not have parameters, which is why we will not need $D_{L+1}$.

To make our analysis simpler, we choose a number $N\in \N\backslash\{0\}$ and set $N_\ell:=N$ for $1\leq \ell\leq L$ (that is, all hidden layers have the same number of neurons). For the innermost and outermost layers, we set $N_0=N_{L+1}:=1$. We also select activation functions:
\[
\left\{
\begin{array}{rcll}
	\sigma^{(\ell)}:&\R^{d_{\ell}}\times \R^{D_\ell}&\to \R^{d_{\ell+1}}&\text{ if }\ell\in[0:L],\\
	\sigma^{(L+1)}:&\R^{d_{L+1}}&\to\R^{d_{L+1}}.&
\end{array}
\right.
\]

Let us describe the vector of parameters or weights. For each $\ell\in[0:L]$ and each pair $(i_\ell,i_{\ell+1})\in [1:N_\ell]\times [1:N_{\ell+1}]$, a parameter (or weight)
\[\theta^{(\ell)}_{i_\ell,i_{\ell+1}}\in \R^{D_\ell},\]
 will go into the function $\sigma^{(\ell)}$. The interpretation is that $\theta^{(\ell)}_{i_\ell,i_{\ell+1}}$ parametrizes the connection between the $(i_\ell)^\text{th}$ neuron in the $\ell^\text{th}$ layer and the $(i_{\ell+1})^\text{th}$ neuron in the $(\ell+1)^\text{th}$ layer.

We can collect all the parameters $\theta^{(\ell)}_{i_\ell,i_{\ell+1}}$ into a single parameter vector:
\[
	\btheta_N:=\Big(\theta^{(\ell)}_{i_\ell,i_{\ell+1}}\,:\,\ell\in[0:L],\;(i_\ell,i_{\ell+1})\in [1:N_\ell]\times [1:N_{\ell+1}]\Big).
\]
This parameter is an element of $\R^{p_N}$, where
\[
	p_N:= \sum_{\ell=0}^{L}\,N_{\ell}N_{\ell+1}D_\ell.
\]

We now explain how function $\yhat:\R^{d_X}\times \R^{p_N}\to \R^{d_Y}$ is parametrized by $\btheta_N$.
\begin{definition}[Parametric Deep Neural Network]\label{def:DNN}
Given $(x,\btheta_N)\in \R^{d_X}\times \R^{p_N}$, we define $\yhat(x,\btheta_N)$ as follows.

First, for $i_1\in [1:N_1]$, we define
\begin{equation}
\label{eq:defz1}
	z^{(1)}_{i_1}(x,\btheta_N) :=
	\fsigma{0}{x,\theta^{(0)}_{1,i_1}}.
\end{equation}
For $\ell\in[1:L]$ and $i_{\ell+1}\in[1:N_{\ell+1}]$, we proceed inductively to define
\begin{equation}
\label{eq:defzell}
	z^{(\ell+1)}_{i_{\ell+1}}(x,\btheta_{N}):=
	\frac{1}{N_\ell}	\sum_{i_\ell=1}^{N_\ell}	\fsigma{\ell}{z^{(\ell)}_{i_\ell}(x,\btheta_{N}), \theta^{(\ell)}_{i_\ell,i_{\ell+1}}}.
\end{equation}
Finally, the output of the network is given by
\begin{equation}
\label{eq:outputNN}
	\yhat(x,\btheta_{N})   :=   \fsigma{L+1}{z_1^{(L+1)}(x,\btheta_{N})}.
\end{equation}
\end{definition}

The interpretation is that $z^{(\ell)}_{i_\ell}$ is the function computed at the $i_\ell^{th}$ neuron of layer $\ell$. Figure \ref{fig:parameters_drawing} gives a representation of the DNN: edges between units in layers $\ell$ and $\ell+1$ represent that the function computed by the network at a neuron in layer $\ell+1$ involves all values at layer $\ell$.

\noindent
\begin{figure}[h]
\input{measureNN_aux/fig1-net_parameters}
\label{fig:parameters_drawing}
\end{figure}

Our model defines a very general version of a neural network with fully connected layers, where the internal units may have dimension greater than $1$.

\begin{remark}[Averages] Notice that, for $\ell>0$, the definition of $z^{(\ell+1)}_{i_{\ell+1}}$ involves averages over $i_\ell$. These averages will effectively lead to a Law of Large Numbers as the number of neurons grows
\end{remark}

\subsection{The backpropagation equations.}\label{sub:backprop} As is standard, the gradients of $\yhat$ with respect to the weights may be computed via backpropagation \cite{Goodfellow2016}. The key idea is to write the derivatives of $\yhat$ with respect to the activations recursively, starting from the output. Since
\[
	\yhat(x,\btheta_N) = \fsigma{L+1}{z_1^{(L+1)}(x,\btheta_N)},
\]
we write
\[
	\frac{\partial \yhat}{\partial z_1^{(L+1)}}(x,\btheta_N) = D\fsigma{L+1}{z_1^{(L+1)}(x,\btheta_N)}.
\]
For $0\leq \ell\leq L$, we have
\[
	\frac{\partial z_{i_{\ell+1}}^{(\ell+1)}}{\partial z_{i_\ell}^{(\ell)}}(x,\btheta_N) =
	\frac{1}{N_\ell} D_z\fsigma{\ell}{z_{i_\ell}^{(\ell)}(x,\btheta_N),\theta^{(\ell)}_{i_\ell,i_{\ell+1}}},
\]
where $D_z\sigma^{(\ell)}$ denotes the partial Fr\'{e}chet derivative of $\sigma^{(\ell)}$ with respect to the first variable.

To compute the derivative of the function with respect to the $z^{(\ell)}_{i_\ell}$, we will establish the following notation:
We denote $\boldN{\ell}{L+1} = (N_\ell,N_{\ell+1},\cdots,N_L,N_{L+1})$. Then we use the notation
\[
	[1:\boldN{\ell}{L+1}] := [1:N_{\ell}]\times[1:N_{\ell+1}]\times\cdots\times[1:N_{L+1}]
\]
to denote a set of multi-indices.
We write the elements of $[1:\boldN{\ell}{L+1}]$ as
$$\boldj{\ell}{L+1} = (j_\ell, j_{\ell+1},\cdots,j_{L+1}),$$
and when joining an index $i_\ell$ with a list $\boldj{\ell}{L+1} $ we write
\[
	\left(i_\ell, \boldj{\ell+1}{L+1}\right) = (i_\ell, j_{\ell+1}, \cdots, j_{L+1})\in [1:\boldN{\ell}{L+1}].
\]

Using this notation, we obtain
\begin{equation}
	\frac{\partial \yhat}{\partial z_{i_\ell}^{(\ell)}}(x,\btheta_N) =
	\left(\frac{1}{\prod_{k=\ell}^{L+1}N_k}\right)   \sum_{\boldj{\ell+1}{L+1}\in [1:\boldN{\ell+1}{L+1}]}
	M^{(\ell)}_{\boldj{\ell+1}{L+1}}(x,\btheta_N),
\end{equation}
where
\begin{align}
\label{eq:defM}
	\nonumber M^{(L+1)}_1(x,\btheta_N) :=& D\fsigma{L+1}{z^{(L+1)}(x,\btheta_N)}\quad\quad\text{ and}\\
	          M^{(\ell)}_{\boldj{\ell}{L+1}}(x,\btheta_N) :=&
	          M^{(\ell+1)}_{\boldj{\ell+1}{L+1}}(x,\btheta_N)\,\cdot
	          D_z\fsigma{\ell}{z_{j_\ell}^{(\ell)}(x,\btheta_N),\theta^{(\ell)}_{j_\ell,j_{\ell+1}}}\\
	\nonumber & \text{for all } \ell\in[0:L], \quad \boldj{\ell}{L+1}\in [1:\boldN{\ell}{L+1}].
\end{align}

Then, for $\ell\in [0:L]$ and $(i_\ell,i_{\ell+1})\in [1:N_\ell]\times [1:N_{\ell+1}]$, we have
\begin{align*}
	\frac{\partial \yhat}{\partial \theta_{i_\ell,i_{\ell+1}}^{(\ell)}}(x,\btheta_N) =&
	\frac{\partial\yhat}{\partial z_{i_{\ell+1}}^{(\ell+1)}}(x,\btheta_N)\,.\,
	\frac{\partial z_{i_{\ell+1}}^{(\ell+1)}}{\partial \theta_{i_\ell,i_{\ell+1}}^{(\ell)}}(x,\btheta_N)\\ =&  
	\frac{\partial\yhat}{\partial z_{i_{\ell+1}}^{(\ell+1)}}(x,\btheta_N)\,\cdot\,
	\parn{\frac{1}{N_\ell}\cdot D_\theta\fsigma{\ell}{z_{i_\ell}^{(\ell)}(x,\btheta_N),\theta^{(\ell)}_{i_\ell,i_{\ell+1}}}},
\end{align*}
where $D_\theta\sigma^{(\ell)}$ denotes the Fr\'{e}chet derivative with respect to the second variable. We may deduce via recursion that
\begin{align}
\label{eq:formulagrad}
	\frac{\partial \yhat}{\partial \theta_{i_\ell,i_{\ell+1}}^{(\ell)}}(x,\btheta_N) =&
	\left(
		\frac{1}{\prod_{k={\ell+2}}^{L+1}N_k}\cdot
		\sum\limits_{\boldj{\ell+2}{L+1}\in[1:\boldN{\ell+2}{L+1}]}M^{(\ell+1)}_{(i_{\ell+1},\boldj{\ell+2}{L+1})}(x,\btheta_N)
	\right)
	\\\nonumber
	&\times
	\left(
		\frac{1}{N_\ell N_{\ell+1}}\cdot
		D_\theta\fsigma{\ell}{z_{i_\ell}^{(\ell)}(x,\btheta_N),\theta^{(\ell)}_{i_\ell,i_{\ell+1}}}
	\right).
\end{align}

\begin{remark}[Expected behavior] Notice that \eqnref{formulagrad} involves an average over the operators $M^{(\ell)}_{[\boldsymbol \cdot]}$. Since $N_{L+1}=1$, this ``averaging" consists of a single term when $\ell=L+1$. For smaller $\ell$, we will replace this average by an integral in the mean-field limit.\end{remark}

\begin{remark}[Time scales]\label{rem:timescales} Formula \eqnref{formulagrad} also suggests that weights $\theta^{(\ell)}_{i_\ell,i_{\ell+1}}$ evolve at a time scale of $1/N_{\ell}N_{\ell+1}$. This is is $1/N$ for $\ell=0,L$ and $1/N^2$ for other $\ell$. This strongly suggests a separation of time scales. Unfortunately, we currently lack the techniques to address this. This is the reason why we keep the weights with superscripts $\ell=0,L$ fixed throughout our training. By contrast, Nguyen \cite{Nguyen2019}, Sirignano, and Spiliopoulos \cite{Sirignano2019} rescale the evolution of the weights of layers $\ell=0,L$ to make all time scales match. Although we have not investigated this carefully, we believe our results can be directly extended to that setting.  \end{remark}

\subsection{Training by stochastic gradient descent}
\label{sub:DNN_definition}
Lastly, we describe the training procedure for our network by {\em Stochastic Gradient Descent} (or SGD). Specifically, weights are updated at discrete time steps according to a point estimate of the gradient of the loss function $\L_N$ obtained from a fresh sample point. This matches the assumption of previous papers such as \cite{Mei2018,MeiArxiv2018,Sirignano2018,Rotskoff2018,Sirignano2019,Nguyen2019}.

We start by defining the function $\Gradh^{(\ell)}_{i_\ell,i_{\ell+1}}:\R^{d_{X}}\times\R^{d_{Y}}\times\R^{p_N}\to \R^{D_\ell}$ as follows:
\begin{equation}
\label{eq:grad-ell}
{\Gradh}^{(\ell)}_{i_\ell,i_{\ell+1}}(X,Y,\btheta_N) =
\begin{cases}
\bold{0}_{\R^{d_\ell}} &\text{if } \ell\in\{0,L\},\\
\big(Y - \yhat(X,\btheta_N)\big)^\dag N_\ell N_{\ell+1}\frac{\partial \yhat}{\partial \theta_{i_\ell,i_{\ell+1}}^{(\ell)}}(X,\btheta_N)
& \ell\in[1:L-1].
\end{cases}
\end{equation}
The full vector of scaled gradients is denoted by
$$\Gradh_N(X,Y,\btheta_N) = \parn{\Gradh^{(\ell)}_{i_\ell,i_{\ell+1}}(X,Y,\btheta_N),\,\substack{\ell\in[0:L],\\
(i_\ell,i_{\ell+1})\in[1:N_\ell]\times[1:N_{\ell+1}]}}.$$

Fix $\eps>0$ and let $\alpha:\R_+\to \R_+$ be a given function. Let $\{X_k,Y_k, k\in\N\}\iid P$ be the dataset. We choose an initial vector $\btheta_N(0)\in \R^{p_N}$ independently of the sample. Then we proceed recursively.

Assume $\btheta_N(s)\in \R^{p_N}$ is defined for $0\leq s\leq k$. For $\ell\in[0:L]$ and $(i_\ell,i_{\ell+1})\in [N_\ell]\times [N_{\ell+1}]$, we define the processes
\begin{equation}
\label{def:sgd-process}
\theta^{(\ell)}_{i_\ell,i_{\ell+1}}(k+1) =  \theta^{(\ell)}_{i_\ell,i_{\ell+1}}(k) - \eps\cdot\alpha(k\eps)\cdot \cdot\Gradh^{(\ell)}_{i_\ell,i_{\ell+1}}\left(X_{k+1},Y_{k+1},\btheta_N(k)\right).
\end{equation}

\begin{remark}[Averaging the increment]
\label{rem:average-grad}
Let $\{\sF_{k},\,\,{k\in\N}\}$ be the filtration generated by $\btheta_N(0)$ and $\{(X_s,Y_s),\, 1\leq s\leq k\}$.
Observe that for $\ell\in[1:L-1]$ and $(i_\ell,i_{\ell+1})\in [N_\ell]\times [N_{\ell+1}]$
\begin{equation}
\Exp{(X_{k+1},Y_{k+1})\sim P}{{\Gradh}^{(\ell)}_{i_\ell,i_{\ell+1}}(X_{k+1},Y_{k+1},\btheta_N)} = N^2\,\frac{\partial \L_N(\btheta_N)}{\partial \theta^{(\ell)}_{i_\ell,i_{\ell+1}}}.
\end{equation}
This means that for $\ell\in[1:L-1]$
\[
	\Ex{\left.\theta^{(\ell)}_{i_\ell,i_{\ell+1}}(k+1) -  \theta^{(\ell)}_{i_\ell,i_{\ell+1}}(k)\right|\sF_{k-1}} =
	-\eps\cdot\alpha(k\eps)\cdot N^2\cdot\frac{\partial \L_N(\btheta_N(k))}{\partial \theta_{i_\ell,i_{\ell+1}}^{(\ell)}}.
\]
In other words, the network weights of the hidden layers follow, on average, the negative gradient of the population loss.
\end{remark}

\begin{remark}[Random features]\label{rem:randomfeatures}As noted in Remark \ref{rem:timescales}, we avoid the problem of time scales in weight updates by keeping fixed $\theta^{(\ell)}$ with $\ell=0,L$. One may interpret that by saying that we require "random features" \cite{Rahmi2008} close to the input and output. Therefore, this is not a completely unnatural assumption. Still, we believe it is inessential, as we could have opted to tune the learning rates at different layers (again as dicussed in Remark \ref{rem:timescales}).\end{remark}

%% file: measureNN_aux/fig1-net_parameters.tex
\makebox[\textwidth]{
\resizebox{\textwidth}{!}{
\begin{tikzpicture}

\node [above] at (1,4) {$X$};

\draw [rounded corners, cyan, fill=cyan] (4,0) rectangle (6,7.4);
\draw [rounded corners, fill=white] (4.2,.2) rectangle (5.8,1.2);
\node at (5, .7) {$z^1_{N_1}$};
\draw [rounded corners, fill=white] (4.2,2.6) rectangle(5.8,3.6 );
\node at (5, 3.1) {$z^1_{i_1}$};
\draw [rounded corners, fill=white] (4.2,5) rectangle(5.8,6);
\node at (5, 5.5) {$z^1_2$};
\draw [rounded corners, fill=white] (4.2,6.2) rectangle(5.8,7.2);
\node at (5, 6.7) {$z^1_1$};

\node [font=\Huge] at (5,4.4) {$\vdots$};
\node [font=\Huge] at (5,2) {$\vdots$};

\draw  (1,3.7) to [out = 60,in=180]
node[near end, fill=white] {$\theta^{(0)}_{1, 1}$} (4.2,6.7);
\draw  (1,3.7) to [out = 30,in=180]
node[near end, fill=white] {$\theta^{(0)}_{1, 2}$} (4.2,5.5);
\draw  (1,3.7) to [out = -10,in=180]
node[near end, fill=white] {$\theta^{(0)}_{1, i_1}$} (4.2,3.1);
\draw  (1,3.7) to [out = -60,in=180]
node[near end, fill=white] {$\theta^{(0)}_{1, N_1}$} (4.2,0.7);
\draw [black, fill=black] (1,3.7) circle (.2);
\draw [rounded corners, lightgray, fill=lightgray] (2.5,8.5) rectangle (3.5,7.5);
\node at (3,8) {$\sigma^0$};

\draw [rounded corners, cyan, fill=cyan] (9,0) rectangle (11,7.4);

\draw [rounded corners, fill=white] (9.2,.2) rectangle (10.8,1.2);
\node at (10, .7) {$z^2_{N_2}$};
\draw [rounded corners, fill=white] (9.2,2.6) rectangle(10.8,3.6 );
\node at (10, 3.1) {$z^2_{i_2}$};
\draw [rounded corners, fill=white] (9.2,5) rectangle(10.8,6);
\node at (10, 5.5) {$z^2_2$};
\draw [rounded corners, fill=white] (9.2,6.2) rectangle(10.8,7.2);
\node at (10, 6.7) {$z^2_1$};

\node [font=\Huge] at (10,4.4) {$\vdots$};
\node [font=\Huge] at (10,2) {$\vdots$};

\begin{scope}[lightgray, very thin]
\draw [very thick] (5.8,6.7) -- (6.3,6.7);
\draw  (6.3,6.7) to [out = 10,in=170] (9.2,6.7);
\draw  (6.3,6.7) to [out = -20,in=180] (9.2,5.5);
\draw  (6.3,6.7) to [out = -45,in=180] (9.2,3.1);
\draw  (6.3,6.7) to [out = -75,in=180] (9.2,0.7);
\draw  [fill=lightgray](6.3,6.7) circle (.1);

\draw [thick] (5.8,5.5) -- (6.3,5.5);
\draw  (6.3,5.5) to [out = 20,in=170] (9.2,6.7);
\draw  (6.3,5.5) to [out = 0,in=180] (9.2,5.5);
\draw  (6.3,5.5) to [out = -20,in=180] (9.2,3.1);
\draw  (6.3,5.5) to [out = -45,in=180] (9.2,0.7);
\draw  [fill=lightgray](6.3,5.5) circle (.1);

\begin{scope}[black, thick]
\draw [thick] (5.8,3.1) -- (6.3,3.1);
\draw  (6.3,3.1) to [out = 45,in=170]
node[midway, fill=white] {$\theta^{(1)}_{i_1, 1}$} (9.2,6.7);
\draw  (6.3,3.1) to [out = 30,in=180]
node[midway, fill=white] {$\theta^{(1)}_{i_1, 2}$} (9.2,5.5);
\draw  (6.3,3.1) to [out = 0,in=180]
node[midway, fill=white] {$\theta^{(1)}_{i_1, i_2}$} (9.2,3.1);
\draw  (6.3,3.1) to [out = -20,in=180]
node[midway, fill=white] {$\theta^{(1)}_{i_1, N_2}$} (9.2,0.7);
\draw  [fill=black](6.3,3.1) circle (.1);
\end{scope}

\draw [thick] (5.8,.7) -- (6.3,.7);
\draw  (6.3,.7) to [out = 75,in=170] (9.2,6.7);
\draw  (6.3,.7) to [out = 45,in=180] (9.2,5.5);
\draw  (6.3,.7) to [out = 20,in=180] (9.2,3.1);
\draw  (6.3,.7) to [out = -10,in=190] (9.2,0.7);
\draw  [fill=lightgray](6.3,0.7) circle (.1);
\end{scope}

\draw[black,fill=white] (9,6.7) circle (.25) node [font=\large] {$+$};
\draw[black,fill=white] (9,5.5) circle (.25) node [font=\large] {$+$};
\draw[black,fill=white] (9,3.1) circle (.25) node [font=\large] {$+$};
\draw[black,fill=white] (9,0.7) circle (.25) node [font=\large] {$+$};

\draw [rounded corners, lightgray, fill=lightgray] (7.3,8.5) rectangle (8.3,7.5);
\node at (7.8,8) {$\sigma^1$};

\begin{scope}[lightgray, very thin]
\draw [very thick] (10.8,6.7) -- (11.3,6.7);
\draw  (11.3,6.7) to [out = 10,in=170]  (14.2,6.7);
\draw  (11.3,6.7) to [out = -20,in=180] (14.2,5.5);
\draw  (11.3,6.7) to [out = -45,in=180] (14.2,3.1);
\draw  (11.3,6.7) to [out = -75,in=180] (14.2,0.7);
\draw  [fill=lightgray](11.3,6.7) circle (.1);

\draw [thick] (10.8,5.5) -- (11.3,5.5);
\draw  (11.3,5.5) to [out = 20,in=170] (14.2,6.7);
\draw  (11.3,5.5) to [out = 0,in=180] (14.2,5.5);
\draw  (11.3,5.5) to [out = -20,in=180] (14.2,3.1);
\draw  (11.3,5.5) to [out = -45,in=180] (14.2,0.7);
\draw  [fill=lightgray](11.3,5.5) circle (.1);

\begin{scope}[black,thick]
\draw [thick] (10.8,3.1) -- (11.3,3.1);
\draw  (11.3,3.1) to [out = 45,in=170]
node[midway, fill=white] {$\theta^{(2)}_{i_2, 1}$} (14.2,6.7);
\draw  (11.3,3.1) to [out = 30,in=180]
node[midway, fill=white] {$\theta^{(2)}_{i_2, 2}$} (14.2,5.5);
\draw  (11.3,3.1) to [out = 0,in=180]
node[midway, fill=white] {$\theta^{(2)}_{i_2, i_3}$} (14.2,3.1);
\draw  (11.3,3.1) to [out = -20,in=180]
node[midway, fill=white] {$\theta^{(2)}_{i_2, N_3}$} (14.2,0.7);
\draw  [fill=black](11.3,3.1) circle (.1);
\end{scope}

\draw [thick] (10.8,.7) -- (11.3,.7);
\draw  (11.3,.7) to [out = 75,in=170] (14.2,6.7);
\draw  (11.3,.7) to [out = 45,in=180] (14.2,5.5);
\draw  (11.3,.7) to [out = 20,in=180] (14.2,3.1);
\draw  (11.3,.7) to [out = -10,in=190] (14.2,0.7);
\draw  [fill=lightgray](11.3,0.7) circle (.1);

\draw [rounded corners, lightgray, fill=lightgray] (12,8.5) rectangle (13,7.5);
\node [black] at (12.5,8) {$\sigma^2$};

\node [font=\huge] at (17.2,6.7) {$\cdots$};
\node [font=\huge] at (17.2,5.5) {$\cdots$};
\node [font=\huge] at (17.2,3.1) {$\cdots$};
\node [font=\huge] at (17.2,0.7) {$\cdots$};

\end{scope}

\draw [rounded corners, cyan, fill=cyan] (14,0) rectangle (16,7.4);
\draw [rounded corners, fill=white] (14.2,.2) rectangle (15.8,1.2);
\node at (15, .7) {$z^3_{N_3}$};
\draw [rounded corners, fill=white] (14.2,2.6) rectangle(15.8,3.6 );
\node at (15, 3.1) {$z^3_{i_3}$};
\draw [rounded corners, fill=white] (14.2,5) rectangle(15.8,6);
\node at (15, 5.5) {$z^3_2$};
\draw [rounded corners, fill=white] (14.2,6.2) rectangle(15.8,7.2);
\node at (15, 6.7) {$z^3_1$};

\begin{scope}[lightgray]
\draw [thick] (15.8,6.7) -- (16.3,6.7);
\draw  [lightgray,fill=lightgray](16.3,6.7) circle (.1);
\draw [thick] (15.8,5.5) -- (16.3,5.5);
\draw  [lightgray,fill=lightgray](16.3,5.5) circle (.1);
\draw [thick] (15.8,3.1) -- (16.3,3.1);
\draw  [lightgray,fill=lightgray](16.3,3.1) circle (.1);
\draw [thick] (15.8,0.7) -- (16.3,0.7);
\draw  [lightgray,fill=lightgray](16.3,0.7) circle (.1);
\end{scope}

\draw[black,fill=white] (14,6.7) circle (.25) node [font=\large] {$+$};
\draw[black,fill=white] (14,5.5) circle (.25) node [font=\large] {$+$};
\draw[black,fill=white] (14,3.1) circle (.25) node [font=\large] {$+$};
\draw[black,fill=white] (14,0.7) circle (.25) node [font=\large] {$+$};

\draw [rounded corners, cyan, fill=cyan] (21,0) rectangle (23,7.4);
\draw [rounded corners, fill=white] (21.2,.2) rectangle (22.8,1.2);
\node at (22, .7) {$z^{\ell+1}_{N_\ell}$};
\draw [rounded corners, fill=white] (21.2,2.6) rectangle(22.8,3.6 );
\node at (22, 3.1) {$z^{\ell+1}_{i_\ell}$};
\draw [rounded corners, fill=white] (21.2,5) rectangle(22.8,6);
\node at (22, 5.5) {$z^{\ell+1}_2$};
\draw [rounded corners, fill=white] (21.2,6.2) rectangle(22.8,7.2);
\node at (22, 6.7) {$z^{\ell+1}_1$};

\begin{scope}[lightgray]
\draw [thick] (22.8,3.1) -- (23.3,3.1);
\draw [thick] (22.8,5.5) -- (23.3,5.5);
\draw [thick] (22.8,6.7) -- (23.3,6.7);
\draw [thick] (22.8,0.7) -- (23.3,0.7);
\draw  [lightgray,fill=lightgray](23.3,6.7) circle (.1);
\draw  [lightgray,fill=lightgray](23.3,5.5) circle (.1);
\draw  [lightgray,fill=lightgray](23.3,3.1) circle (.1);
\draw  [lightgray,fill=lightgray](23.3,0.7) circle (.1);
\end{scope}

\node [font=\Huge] at (22,4.4) {$\vdots$};
\node [font=\Huge] at (22,2) {$\vdots$};

\begin{scope}[lightgray, very thin]
\draw [very thick] (17.8,6.7) -- (18.3,6.7);
\draw  (18.3,6.7) to [out = 10,in=170]  (21.2,6.7);
\draw  (18.3,6.7) to [out = -20,in=180] (21.2,5.5);
\draw  (18.3,6.7) to [out = -45,in=180] (21.2,3.1);
\draw  (18.3,6.7) to [out = -75,in=180] (21.2,0.7);
\draw  [fill=lightgray](18.3,6.7) circle (.1);

\draw [thick] (17.8,5.5) -- (18.3,5.5);
\draw  (18.3,5.5) to [out = 20,in=170] (21.2,6.7);
\draw  (18.3,5.5) to [out = 0,in=180] (21.2,5.5);
\draw  (18.3,5.5) to [out = -20,in=180] (21.2,3.1);
\draw  (18.3,5.5) to [out = -45,in=180] (21.2,0.7);
\draw  [fill=lightgray](18.3,5.5) circle (.1);

\begin{scope}[black,thick]
\draw [thick] (17.8,3.1) -- (18.3,3.1);
\draw  (18.3,3.1) to [out = 45,in=170]
node[midway, fill=white] {$\theta^{(\ell)}_{i_\ell, 1}$} (21.2,6.7);
\draw  (18.3,3.1) to [out = 30,in=180]
node[midway, fill=white] {$\theta^{(\ell)}_{i_\ell, 2}$} (21.2,5.5);
\draw  (18.3,3.1) to [out = 0,in=180]
node[midway, fill=white] {$\theta^{(\ell)}_{i_\ell, i_{\ell+1}}$} (21.2,3.1);
\draw  (18.3,3.1) to [out = -20,in=180]
node[midway, fill=white] {$\theta^{(\ell)}_{i_\ell, N_\ell}$} (21.2,0.7);
\draw  [fill=black](18.3,3.1) circle (.1);
\end{scope}

\begin{scope}[lightgray]
\draw [thick] (17.8,.7) -- (18.3,.7);
\draw  (18.3,.7) to [out = 75,in=170] (21.2,6.7);
\draw  (18.3,.7) to [out = 45,in=180] (21.2,5.5);
\draw  (18.3,.7) to [out = 20,in=180] (21.2,3.1);
\draw  (18.3,.7) to [out = -10,in=190] (21.2,0.7);
\draw  [fill=lightgray](18.3,0.7) circle (.1);
\end{scope}

\draw[black,fill=white] (21,6.7) circle (.25) node [font=\large] {$+$};
\draw[black,fill=white] (21,5.5) circle (.25) node [font=\large] {$+$};
\draw[black,fill=white] (21,3.1) circle (.25) node [font=\large] {$+$};
\draw[black,fill=white] (21,0.7) circle (.25) node [font=\large] {$+$};

\draw [rounded corners, lightgray, fill=lightgray] (19,8.5) rectangle (20,7.5);
\node [black] at (19.5,8) {$\sigma^\ell$};

\node [font=\huge] at (24.2,6.7) {$\cdots$};
\node [font=\huge] at (24.2,5.5) {$\cdots$};
\node [font=\huge] at (24.2,3.1) {$\cdots$};
\node [font=\huge] at (24.2,0.7) {$\cdots$};
\end{scope}

\draw [rounded corners, cyan, fill=cyan] (28,2.2) rectangle (30,4.9);
\node [above] at (32,3.9) {$\hat{y}_N$};
\draw [black] (32,3.7) -- (29.8, 3.7);
\draw [black,fill=black] (32,3.7) circle (.2);

\begin{scope}[thick]

\draw  (28,3.7) to [out = 120,in=0]
node[near end, fill=white] {$\theta^{(L)}_{1, 1}$} (25.2,6.7);
\draw  (28,3.7) to [out = 150,in=0]
node[near end, fill=white] {$\theta^{(L)}_{2,1}$} (25.2,5.5);
\draw  (28,3.7) to [out = 190,in=0]
node[near end, fill=white] {$\theta^{(L)}_{i_L,1}$} (25.2,3.1);
\draw  (28,3.7) to [out = -120,in=0]
node[near end, fill=white] {$\theta^{(L)}_{N_L,1}$} (25.2,0.7);
\end{scope}

\draw [rounded corners, black, fill=white] (28.2,4.2) rectangle (29.8,3.2);
\node [black] at (29,3.7) {$z^{L+1}_1$};

\draw [rounded corners, lightgray, fill=lightgray] (26,8.5) rectangle (27,7.5);
\node [black] at (26.5,8) {$\sigma^L$};

\draw [rounded corners, lightgray, fill=lightgray] (30.4,5.5) rectangle (31.6,4.5);
\node [black] at (31,5) {$\sigma^{L+1}$};

\draw[black,fill=white] (28,3.7) circle (.25) node [font=\large] {$+$};

\draw [decorate,decoration={brace,amplitude=10pt},yshift=-5pt]
(6,0) -- (0,0) node [black,midway,yshift=-20pt] {Layer 0};
\draw [decorate,decoration={brace,amplitude=10pt},yshift=-5pt]
(11.3,0) -- (6,0) node [black,midway,yshift=-20pt] {Layer 1};
\draw [decorate,decoration={brace,amplitude=10pt},yshift=-5pt]
(16,0) -- (11.3,0) node [black,midway,yshift=-20pt] {Layer 2};
\draw [decorate,decoration={brace,amplitude=10pt},yshift=-5pt]
(23,0) -- (18,0) node [black,midway,yshift=-20pt] {Layer $\ell$};
\draw [decorate,decoration={brace,amplitude=10pt},yshift=-5pt]
(30,0) -- (25,0) node [black,midway,yshift=-20pt] {Layer L};
\draw [decorate,decoration={brace,amplitude=10pt},yshift=-5pt]
(32,0) -- (30,0) node [black,midway,yshift=-20pt] {Layer L+1};

\end{tikzpicture}
}}
\caption{\footnotesize{Representation of the Parameters over the Deep Neural Network with L+1 Layers.}}

%% file: measureNN_aux/3-mkeanvlasov.tex
We now explain the mean-field behavior we expect from our system in the large-$N$ limit.

\subsection{An ansatz} The basic idea behind our model is that the summations appearing in the gradients of $L_N$ are to be replaced by integrals over conditional distributions. This idea was applied nonrigorously (and independently) by Nguyen \cite{Nguyen2019}. The {\em ansatz} can be summarized as follows:
\begin{itemize}
\item Weights on the same layer are statistically indistinguishable.
\item Gradients of weights in layers $\ell=0,1,2,3,\dots,L-1,L$ depend on averages taken over network paths. Some paths go "backwards" to the input and some go "forward" to the output.
\item Averages over many paths are expected to obey a law of large numbers and become deterministic in the limit. 
\item Weights $\theta^{(\ell)}_{i_\ell,i_{\ell+1}}$ in layers $2\leq \ell\leq L-2$ have lots of paths both ways. Since averages become deterministic integrals, we expect these weights to decouple from the rest of the network. The same holds for weights in layer $\ell=0$ (only forward paths) and $\ell=L$ (only backward paths).  
\item By contrast, a weight $\theta^{(1)}_{i_1,i_{2}}$ has one backwards path through $\theta^{(0)}_{1,i_1}$. So we expect that $\theta^{(1)}_{i_1,i_{2}}$ will depend on $\theta^{(0)}_{1,i_{1}}$ (but will decouple from other weights) even in the thermodynamic limit. A similar reasoning shows that $\theta^{(L-1)}_{i_{L-1},i_{L}}$ should depend on $\theta^{(L)}_{i_L,1}$, but on no other weights.
\end{itemize}

This leads us to consider {\em input-to-output paths} in the network as our basic units in the analysis:
\begin{equation}\label{eq:defpath}
	\theta = \parn{\theta^{(0)}_{1,i_1},\theta^{(1)}_{i_1,i_2},\theta^{(2)}_{i_2,i_3},\dots,\theta^{(L)}_{i_L,i_1}}.
\end{equation}
Figure \ref{fig:unit} illustrates one such path.
\input{measureNN_aux/fig2-path}
If our ansatz is correct, understanding the distribution of weights on a path suffices to describe the law of all weights in the network. We clarify this in the next subsection. 

\subsection{Paths and measures} Note that the vector $\theta$ in (\ref{eq:defpath}) lives in
\begin{equation}
	\R^{D_0}\times \R^{D_1} \times \dots \times \R^{D_L}\cong \R^D\mbox{ ,where }D:= D_0 + D_1 + D_2 + \dots + D_L.
\end{equation}

Given $x\in \R^D$, we write $x= (x^{(\ell)})_{\ell=0}^L$ with each $x^{(\ell)}\in \R^{D_\ell}$. We will need some notation for the high-dimensional measures.

\begin{definition}[Marginals]
Let $\mu$ be a probability measure over $\R^{D}$. Given $i\in [0:L]$, we write $\mu^{(i)}$ to denote the marginal over the $i$-th factor $\R^{D_i}$ of the Cartesian product. That is, if $\Theta\sim \mu$, then $\Theta^{(i)}\sim \mu^{(i)}$. If $i\leq L-1$, we write $\mu^{(i,i+1)}$ for the marginal over the $i$-th and $(i+1)$-th factors. We also use the symbol $\mu^{(i\mid i+1)}(\cdot\mid a^{(i+1)})$ to denote the conditional distribution of $\Theta^{(i)}$ given $\Theta^{(i+1)}=a^{(i+1)}$.
\end{definition}

Note that the heuristic analysis suggests a factorization for the distribution over a path: the measure over a path should factor as
\begin{equation}\label{eq:mufactorizes}\mu = \mu^{(0,1)}\times \mu^{(2)}\times \dots \times \mu^{(L-2)}\times \mu^{(L-1,L)}.\end{equation}

Moreover, if we can compute the law of the weights on a path, the weights over the entire network should factor (approximately) as follows, at any fixed training step.
\begin{enumerate}
\item We expect $\theta^{(\ell)}_{i_\ell,i_{\ell+1}}\sim \mu^{(\ell)}$ independently for all $\ell\neq 1,L-1$ and $(i_\ell,i_{\ell+1})\in [1:N_\ell]\times [1:N_{\ell+1}]$.
\item Conditionally on the above choices, we expect that the weights on layers $\ell=1,L-1$ are conditionally independent, with $\theta^{(1)}_{i_1,i_2}\sim \mu^{(1\mid 0)}(\cdot\mid \theta^{(0)}_{1,i_1})$ and $\theta^{(L-1)}_{i_{L-1},i_L}\sim \mu^{(L-1\mid L)}(\cdot\mid \theta^{(L)}_{i_L,1})$.\end{enumerate}

We will eventually make this precise via the introduction of "ideal particles". It will be crucial that the above independence structure also works at the level of {\em trajectories}. For this we will need the following notation.

\begin{definition}[Measures on trajectories]
\label{def:measuresontraj}
We use symbols such as $\mu_{[0,T]}$ (greek letter with the subscript $[0,T]$) to denote a probability measure over the product space
\[
	C([0,T],\R^{D}) \cong \bigotimes_{i=0}^{L} C([0,T],\R^{D_i}).
\]
As in the previous definition, if $i\in[0:L]$, we write $\mu_{[0,T]}^{(i)}$ to denote the marginal over the $i^{\text{th}}$ factor in the above product (which is a probability measure over $C([0,T],\R^{D_i})$). Moreover, if $i\in[0:L-1]$, we write $\mu_{[0,T]}^{(i,i+1)}$ for the marginal over both $i^{\text{th}}$ and $(i+1)^{\text{th}}$ factors (which is a probability measure over $C([0,T],\R^{D_i}\times\R^{D_{i+1}})$). We employ $\mu_{[0,T]}^{(i\mid j)}(\cdot\mid x_{[0,T]})$ to denote the conditional law of the $i^{\text{th}}$ trajectory given that the $j^{\text{th}}$ trajectory equals $x_{[0,T]}$. Also, we use a subscript $0\leq t\leq T$ to denote the marginal at time $t$, which is a probability measure over $\R^{D}$.

This means, for instance, that if $\Theta \sim \mu_{[0,T]}$, then $\mu_{t}$ is the law of $\Theta(t)$, $\mu^{(i)}_t$ is the law of its $i^{\text{th}}$ coordinate $\Theta^{(i)}(t)$, $\mu^{(i)}_{[0,T]}$ is the law of the $i^{\text{th}}$ coordinate's trajectory $\Theta^{(i)}\in C([0,T], \R^{D_i})$, and $\mu_{[0,T]}^{(i\mid j)}(\cdot\mid x_{[0,T]})$ is the law of $\Theta^{(i)}$ given $\Theta^{(j)}=x_{[0,T]}$.\end{definition}

\subsection{Mean-field representation of network evolution}
\label{sub:McKeanVlasov}
With our heuristic in mind, we define the "mean-field substitutes" for the function $\yhat$ and the gradients that will appear in the analysis. We first describe what happens when we apply our ansatz to approximate the terms $z^{(\ell)}_{i_\ell}$ (\ref{eq:defzell}) and $M^{(\ell)}_{\boldj{\ell}{L+1}}$ (\ref{eq:defM}) in the definition of gradients. The idea is that the summations will converge to integrals over factors of the measure $\mu$.
\subsubsection{The mean-field versions of neuron values.}
We consider a series of mappings taking as input a pair $(x,\mu)\in \R^{d_X}\times \Probspace(\R^D)$. We first define
\begin{equation}
\label{eq:defzbar2}
	\zbar^{(2)}(x,\mu):=
	\int\limits_{\R^{D_0}\times \R^{D_1}}\sigma^{(1)}\parn{\sigma^{(0)}(x,a^{(0)}),a^{(1)}}\,d\mu^{(0,1)}(a^{(0)},a^{(1)}).
\end{equation}

For $2<\ell\leq L-1$, we define:
\begin{equation}
\label{eq:defzbarell}
	\zbar^{(\ell)}(x,\mu):=
	\int\limits _{\R^{D_{\ell-1}}}\sigma^{(\ell-1)}\parn{\zbar^{(\ell-1)}(x,\mu),a^{(\ell-1)}}\,d\mu^{(\ell-1)}(a^{(\ell-1)}).
\end{equation}

The case of $\zbar^{(L)}$ is different: we define it in terms of $x$ and the {\em conditional distribution} $\mu^{(L-1\mid L)}$. Letting $a^{(L)}\in R^{D_L}$, we define:
\begin{equation}
\label{eq:defzbarL}
	\zbar^{(L)}(x,\mu,a^{(L)}):=
	\int\limits_{\R^{D_{L-1}}}\sigma^{(L-1)}\parn{\zbar^{(L-1)}(x,\mu),a^{(L-1)}}\,d\mu^{(L-1\mid L)}(a^{(L-1)}\mid a^{(L)}).
\end{equation}
This definition depends on the choice of regular conditional probability, but this issue will not concern us in what follows.

To finish, we define:
\begin{align}
\label{eq:defzbarL+1}
	&&\zbar^{(L+1)}(x,\mu)&:=
	\int\limits_{\R^{D_{L}}}\sigma^{(L)}\parn{\zbar^{(L)}(x,\mu,a^{(L)}),a^{(L)}}\,d\mu^{(L)}(a^{(L)})\\
	\mbox{ and }&& \ybar(x,\mu)&:=\sigma^{(L+1)}\parn{\zbar^{(L+1)}(x,\mu)}\nonumber.
\end{align}

\begin{remark}
\label{def:lbar}
Once we have the function $\ybar$, we can consider the mean-squared loss associated with a measure $\mu\in\Probspace(\R^D)$:
\[\Lbar(\mu) = \frac{1}{2}\Exp{(X,Y)\sim P}{|Y-\ybar(X,\mu)|^2}.\]
\end{remark}

\subsubsection{Mean-field versions of gradients.}
The next objects to be defined are the mean-field limits of the $M$ operators and the gradients presented in Section \ref{sub:backprop}. The idea here is that dependencies on $z^{(\ell)}_{i_\ell}$ are to be replaced by the corresponding $\zbar^{(\ell)}$.

Given $(x,\mu)\in \R^{d_X}\times \Probspace(\R^D)$, we first define

\begin{equation}
\label{eq:defMbarLplus1}
	\Mbar^{(L+1)}(x,\mu) := D\sigma^{(L+1)}\parn{\zbar^{(L+1)}(x,\mu)}.
\end{equation}

Then, we proceed as follows:
\begin{align}
\Mbar^{(L)}(x,\mu,a^{(L)}) &:= \Mbar^{(L+1)}(x,\mu) \cdot D_z\sigma^{(L)}\parn{\zbar^{(L)}(x,\mu,a^{(L)}), a^{(L)}},
\label{eq:defMbarL-Lminus1}\\
\Mbar^{(L-1)}(x,\mu) &:= {\displaystyle \int\limits_{\R^{D_{L}}}}\Mbar^{(L)}\parn{x,\mu,a^{(L)}}d\mu^{(L)}(a^{(L)})
\label{eq:defMbarLminus1}\\
 &\times{\displaystyle \int\limits_{\R^{D_{L-1}}}} D_z\sigma^{(L-1)}\parn{\zbar^{(L-1)}(x,\mu),a^{(L-1)}}\,d\mu^{(L-1,)}(a^{(L-1)})\nonumber.
\end{align}

Finally, we proceed recursively defining for $\ell\in [2:L-2]$:
\begin{equation}\label{eq:defMbarell}
\Mbar^{(\ell)}(x,\mu) =  \Mbar^{(\ell+1)}(x,\mu)\,\int\limits_{\R^{D_\ell}}\,D_z\sigma^{(\ell)}\parn{\zbar^{(\ell)}(x,\mu),a^{(\ell)}}\,d\mu^{(\ell)}(a^{(\ell)}).
\end{equation}

We may now define the mean-field equivalent of $\frac{d\yhat}{d\theta^{(\ell)}_{i_\ell,i_{\ell+1}}}$ that will dictate the evolution of the weights in the limit. For weights in the layer $\ell=1$, these gradients take the form
\[
	\gbar^{(1)}:\R^{D}\times \R^{d_X}\times \Probspace(\R^D)\to \R^{D_1},
\]
\begin{equation}
\label{eq:defgbar1}
	\gbar^{(1)}(\theta,x,\mu) = \Mbar^{(2)}(x,\mu)\,D_\theta\sigma^{(1)}\parn{\sigma^{(0)}(x,\theta^{(0)}),\theta^{(1)}}.
\end{equation}

For $\ell\in [2:L-2]$, we take
\[
	\gbar^{(\ell)}: \R^{D}\times \R^{d_X}\times \sM\to \R^{D_\ell},
\]
\begin{equation}
\label{eq:defgbarell}
	\gbar^{(\ell)}(\theta,x,\mu) := \Mbar^{(\ell+1)}(x,\mu)\,D_\theta\sigma^{(\ell)}\parn{\zbar^{(\ell)}(x,\mu),\theta^{(\ell)}}.
\end{equation}
Finally, for $\ell=L-1$ we take
\[
	\gbar^{(L-1)}:\R^{D}\times \R^{d_X}\times \sM\to \R^{D_{L-1}},
\]
\begin{equation}
\label{eq:defgbarL-1}
	\gbar^{(L-1)}(\theta,x,\mu) := \Mbar^{(L)}(x,\mu,\theta^{(L)})\,D_\theta\sigma^{(L-1)}\parn{\zbar^{(L-1)}(x,\mu),\theta^{(L-1)}}.
\end{equation}

We can define, analogously to (\ref{eq:grad-ell}), a function as follows:
\begin{equation}
\label{def:Grad-ell}
\Grad^{(\ell)}(X,Y,\mu,\theta) =
\begin{cases}
\bold{0}_{\R^{D_\ell}} &\text{if } \ell\in\{0,L\},\\
\big(Y - \ybar(X,\mu)\big)^\dag\,\gbar^{(\ell)}(\theta,X,\soln_s)
& \ell\in[1:L-1].
\end{cases}
\end{equation}

Now that we have the mean-field version of the gradients, we can define the following McKean-Vlasov problem.

\begin{definition}[McKean-Vlasov problem for DNN]
\label{def:McKeanVlasov}
Let $\mu_0$ be a probability measure over $\R^{D}$. We say that a probability measure $\soln_{[0,T]}$ over $C([0,T],\R^D)$ is a solution to the McKean-Vlasov DNN problem with initial distribution $\mu_0$ if a random trajectory \[\Thetabar\sim \soln_{[0,T]}\] satisfies:
\begin{enumerate}
\item $\Thetabar(0)\sim \mu_0$;
\item for $\ell \in [0:L]$, $\Thetabar^{(\ell)}$ satisfies almost surely
\[
	\Thetabar^{(\ell)}(t) = \Thetabar^{(\ell)}(0) - \displaystyle{\int_0^t}\alpha(s)\,\Exp{(X,Y)\sim P}{\Grad^{(\ell)}(X,Y,\soln_s,\Thetabar_N(s))}ds,
\]
for all $t\in[0,T]$.
\end{enumerate}
\end{definition}

\begin{remark}[Vector notation]
\label{rem:vectorMcKeanVlasov}
With the above notation, define
\[
	\Grad(X,Y,\mu,\theta) = \parn{\Grad^{(\ell)}(X,Y,\mu,\theta),\ell\in[0:L]}\in \R^{D}.
\]
Then we may rewrite conditions in the above definition as saying that, if $\Thetabar\sim \soln_{[0,T]}$, then
\[
\begin{cases}
	\Thetabar(0)\sim \mu_0,\\
	\Thetabar(t) = \Thetabar(0) -
	\displaystyle{\int_0^t}\alpha(s)\cdot\Exp{(X,Y)\sim P}{\Grad(X,Y,\soln_t,\Thetabar(s))} ds &\mbox{a.s.}
\end{cases}
\]
\end{remark}

\begin{remark}
\label{rem:discontinuity}
The presence of a conditional distribution in the definition of $\zbar^{(L)}$ makes the drift term of Definition \ref{def:McKeanVlasov} potentially discontinuous as a function of the measure $\soln$, both in the weak and Wasserstein topologies. To the best of our knowledge, this kind of discontinuity has not been considered before in the literature on McKean-Vlasov problems. The proof of existence for the McKean-Vlasov problem for DNN will require that we identify a closed set $\Probspecial\subset \Probspace(C([0,T],\R^D))$ where $\zbar^{(L)}$ behaves well. This will be discussed in section \ref{sec:apriori_space}.\end{remark}

%% file: measureNN_aux/fig2-path.tex
\noindent
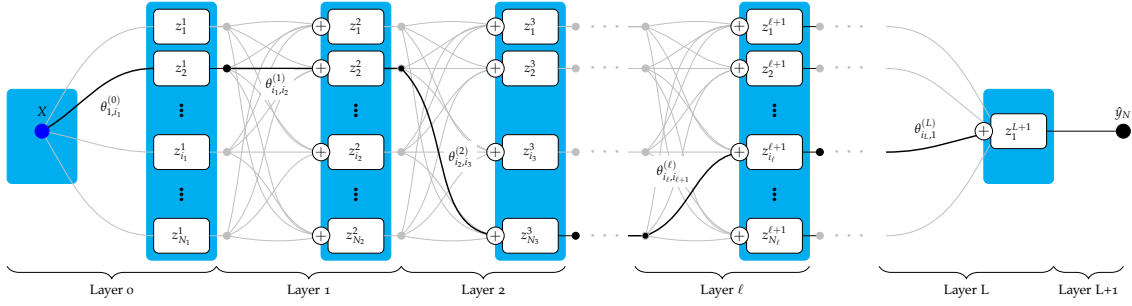
\begin{figure}
\label{fig:unit}
\makebox[\textwidth]{
\resizebox{\textwidth}{!}{
\begin{tikzpicture}
\draw [rounded corners, cyan, fill=cyan] (0,2.2) rectangle (2,4.9);

\node [above] at (1,4) {$X$};

\draw [rounded corners, cyan, fill=cyan] (4,0) rectangle (6,7.4);
\draw [rounded corners, fill=white] (4.2,.2) rectangle (5.8,1.2);
\node at (5, .7) {$z^1_{N_1}$};
\draw [rounded corners, fill=white] (4.2,2.6) rectangle(5.8,3.6 );
\node at (5, 3.1) {$z^1_{i_1}$};
\draw [rounded corners, fill=white] (4.2,5) rectangle(5.8,6);
\node at (5, 5.5) {$z^1_2$};
\draw [rounded corners, fill=white] (4.2,6.2) rectangle(5.8,7.2);
\node at (5, 6.7) {$z^1_1$};

\node [font=\Huge] at (5,4.4) {$\vdots$};
\node [font=\Huge] at (5,2) {$\vdots$};

\begin{scope}[lightgray]
\draw  (1,3.7) to [out = 60,in=180] (4.2,6.7);
\begin{scope}[black,very thick]
\draw  (1,3.7) to [out = 30,in=180] node [midway, below right, fill=white] {$\theta^{(0)}_{1,i_1}$} (4.2,5.5);
\end{scope}
\draw  (1,3.7) to [out = -10,in=180] (4.2,3.1);
\draw  (1,3.7) to [out = -60,in=180] (4.2,0.7);
\end{scope}
\draw [blue, fill=blue] (1,3.7) circle (.2);

\draw [rounded corners, cyan, fill=cyan] (9,0) rectangle (11,7.4);

\draw [rounded corners, fill=white] (9.2,.2) rectangle (10.8,1.2);
\node at (10, .7) {$z^2_{N_2}$};
\draw [rounded corners, fill=white] (9.2,2.6) rectangle(10.8,3.6 );
\node at (10, 3.1) {$z^2_{i_2}$};
\draw [rounded corners, fill=white] (9.2,5) rectangle(10.8,6);
\node at (10, 5.5) {$z^2_2$};
\draw [rounded corners, fill=white] (9.2,6.2) rectangle(10.8,7.2);
\node at (10, 6.7) {$z^2_1$};

\node [font=\Huge] at (10,4.4) {$\vdots$};
\node [font=\Huge] at (10,2) {$\vdots$};

\begin{scope}[lightgray, very thin]
\draw [very thick] (5.8,6.7) -- (6.3,6.7);
\draw  (6.3,6.7) to [out = 10,in=170] (9.2,6.7);
\draw  (6.3,6.7) to [out = -20,in=180] (9.2,5.5);
\draw  (6.3,6.7) to [out = -45,in=180] (9.2,3.1);
\draw  (6.3,6.7) to [out = -75,in=180] (9.2,0.7);
\draw  [fill=lightgray](6.3,6.7) circle (.1);

\draw [thick,black] (5.8,5.5) -- (6.3,5.5);
\draw  (6.3,5.5) to [out = 20,in=170] (9.2,6.7);

\draw  (6.3,5.5) to [out = -20,in=180] (9.2,3.1);
\draw  (6.3,5.5) to [out = -45,in=180] (9.2,0.7);
\draw  [black,fill=black](6.3,5.5) circle (.1);

\draw [thick] (5.8,3.1) -- (6.3,3.1);
\draw  (6.3,3.1) to [out = 45,in=170] (9.2,6.7);
\draw  (6.3,3.1) to [out = 30,in=180] (9.2,5.5);
\draw  (6.3,3.1) to [out = 0,in=180] (9.2,3.1);
\draw  (6.3,3.1) to [out = -20,in=180] (9.2,0.7);
\draw  [fill=lightgray](6.3,3.1) circle (.1);

\draw [thick] (5.8,.7) -- (6.3,.7);
\draw  (6.3,.7) to [out = 75,in=170] (9.2,6.7);
\draw  (6.3,.7) to [out = 45,in=180] (9.2,5.5);
\draw  (6.3,.7) to [out = 20,in=180] (9.2,3.1);
\draw  (6.3,.7) to [out = -10,in=190] (9.2,0.7);
\draw  [fill=lightgray](6.3,0.7) circle (.1);
\begin{scope}[black,very thick]
\draw  (6.3,5.5) to [out = 0,in=180] node [midway, below, fill=white] {$\theta^{(1)}_{i_1,i_2}$} (9.2,5.5);
\end{scope}
\end{scope}

\draw[black,fill=white] (9,6.7) circle (.25) node [font=\large] {$+$};
\draw[black,fill=white] (9,5.5) circle (.25) node [font=\large] {$+$};
\draw[black,fill=white] (9,3.1) circle (.25) node [font=\large] {$+$};
\draw[black,fill=white] (9,0.7) circle (.25) node [font=\large] {$+$};

\begin{scope}[lightgray, very thin]
\draw [thick] (10.8,6.7) -- (11.3,6.7);
\draw  (11.3,6.7) to [out = 10,in=170]  (14.2,6.7);
\draw  (11.3,6.7) to [out = -20,in=180] (14.2,5.5);
\draw  (11.3,6.7) to [out = -45,in=180] (14.2,3.1);
\draw  (11.3,6.7) to [out = -75,in=180] (14.2,0.7);

\draw  [fill=lightgray](11.3,6.7) circle (.1);

\draw [very thick, black] (10.8,5.5) -- (11.3,5.5);
\draw  (11.3,5.5) to [out = 20,in=170] (14.2,6.7);
\draw  (11.3,5.5) to [out = 0,in=180] (14.2,5.5);
\draw  (11.3,5.5) to [out = -20,in=180] (14.2,3.1);

\draw  [fill=black](11.3,5.5) circle (.1);

\draw [thick] (10.8,3.1) -- (11.3,3.1);
\draw  (11.3,3.1) to [out = 45,in=170] (14.2,6.7);
\draw  (11.3,3.1) to [out = 30,in=180] (14.2,5.5);
\draw  (11.3,3.1) to [out = 0,in=180] (14.2,3.1);
\draw  (11.3,3.1) to [out = -20,in=180] (14.2,0.7);
\draw  [fill=lightgray](11.3,3.1) circle (.1);

\begin{scope}[lightgray]
\draw [thick] (10.8,.7) -- (11.3,.7);
\draw  (11.3,.7) to [out = 75,in=170] (14.2,6.7);
\draw  (11.3,.7) to [out = 45,in=180] (14.2,5.5);
\draw  (11.3,.7) to [out = 20,in=180] (14.2,3.1);
\draw  (11.3,.7) to [out = -10,in=190] (14.2,0.7);
\draw  [fill=lightgray](11.3,0.7) circle (.1);
\end{scope}

\begin{scope}[black,very thick]
\draw  (11.3,5.5) to [out = -45,in=180] node [midway, above right, fill=white] {$\theta^{(2)}_{i_2,i_3}$}(14.2,0.7);
\end{scope}

\node [font=\huge] at (17.2,6.7) {$\cdots$};
\node [font=\huge] at (17.2,5.5) {$\cdots$};
\node [font=\huge] at (17.2,3.1) {$\cdots$};
\node [font=\huge] at (17.2,0.7) {$\cdots$};

\end{scope}

\draw [rounded corners, cyan, fill=cyan] (14,0) rectangle (16,7.4);
\draw [rounded corners, fill=white] (14.2,.2) rectangle (15.8,1.2);
\node at (15, .7) {$z^3_{N_3}$};
\draw [rounded corners, fill=white] (14.2,2.6) rectangle(15.8,3.6 );
\node at (15, 3.1) {$z^3_{i_3}$};
\draw [rounded corners, fill=white] (14.2,5) rectangle(15.8,6);
\node at (15, 5.5) {$z^3_2$};
\draw [rounded corners, fill=white] (14.2,6.2) rectangle(15.8,7.2);
\node at (15, 6.7) {$z^3_1$};

\begin{scope}[lightgray]
\draw [thick] (15.8,6.7) -- (16.3,6.7);
\draw  [lightgray,fill=lightgray](16.3,6.7) circle (.1);
\draw [thick] (15.8,5.5) -- (16.3,5.5);
\draw  [lightgray,fill=lightgray](16.3,5.5) circle (.1);
\draw [thick] (15.8,3.1) -- (16.3,3.1);
\draw  [lightgray,fill=lightgray](16.3,3.1) circle (.1);
\draw [thick,black] (15.8,0.7) -- (16.3,0.7);
\draw  [black,fill=black](16.3,0.7) circle (.1);
\end{scope}

\draw[black,fill=white] (14,6.7) circle (.25) node [font=\large] {$+$};
\draw[black,fill=white] (14,5.5) circle (.25) node [font=\large] {$+$};
\draw[black,fill=white] (14,3.1) circle (.25) node [font=\large] {$+$};
\draw[black,fill=white] (14,0.7) circle (.25) node [font=\large] {$+$};

\draw [rounded corners, cyan, fill=cyan] (21,0) rectangle (23,7.4);
\draw [rounded corners, fill=white] (21.2,.2) rectangle (22.8,1.2);
\node at (22, .7) {$z^{\ell+1}_{N_\ell}$};
\draw [rounded corners, fill=white] (21.2,2.6) rectangle(22.8,3.6 );
\node at (22, 3.1) {$z^{\ell+1}_{i_\ell}$};
\draw [rounded corners, fill=white] (21.2,5) rectangle(22.8,6);
\node at (22, 5.5) {$z^{\ell+1}_2$};
\draw [rounded corners, fill=white] (21.2,6.2) rectangle(22.8,7.2);
\node at (22, 6.7) {$z^{\ell+1}_1$};

\draw [thick,black] (22.8,3.1) -- (23.3,3.1);
\draw [thick] (22.8,5.5) -- (23.3,5.5);
\draw [thick] (22.8,6.7) -- (23.3,6.7);
\draw [thick] (22.8,0.7) -- (23.3,0.7);
\draw  [lightgray,fill=lightgray](23.3,6.7) circle (.1);
\draw  [lightgray,fill=lightgray](23.3,5.5) circle (.1);
\draw  [black,fill=black](23.3,3.1) circle (.1);
\draw  [lightgray,fill=lightgray](23.3,0.7) circle (.1);

\node [font=\Huge] at (22,4.4) {$\vdots$};
\node [font=\Huge] at (22,2) {$\vdots$};

\begin{scope}[lightgray, very thin]
\draw [very thick] (17.8,6.7) -- (18.3,6.7);
\draw  (18.3,6.7) to [out = 10,in=170]  (21.2,6.7);
\draw  (18.3,6.7) to [out = -20,in=180] (21.2,5.5);
\draw  (18.3,6.7) to [out = -45,in=180] (21.2,3.1);
\draw  (18.3,6.7) to [out = -75,in=180] (21.2,0.7);
\draw  [fill=lightgray](18.3,6.7) circle (.1);

\draw [thick] (17.8,5.5) -- (18.3,5.5);
\draw  (18.3,5.5) to [out = 20,in=170] (21.2,6.7);
\draw  (18.3,5.5) to [out = 0,in=180] (21.2,5.5);
\draw  (18.3,5.5) to [out = -20,in=180] (21.2,3.1);
\draw  (18.3,5.5) to [out = -45,in=180] (21.2,0.7);
\draw  [fill=lightgray](18.3,5.5) circle (.1);

\draw [thick] (17.8,3.1) -- (18.3,3.1);
\draw  (18.3,3.1) to [out = 45,in=170] (21.2,6.7);
\draw  (18.3,3.1) to [out = 30,in=180] (21.2,5.5);
\draw  (18.3,3.1) to [out = 0,in=180] (21.2,3.1);
\draw  (18.3,3.1) to [out = -20,in=180] (21.2,0.7);
\draw  [fill=lightgray](18.3,3.1) circle (.1);

\draw [very thick,black] (17.8,.7) -- (18.3,.7);
\draw  (18.3,.7) to [out = 75,in=170] (21.2,6.7);
\draw  (18.3,.7) to [out = 45,in=180] (21.2,5.5);
\draw [very thick, black] (18.3,.7) to [out = 20,in=180] node [above left,midway, fill=white] {$\theta^{(\ell)}_{i_\ell,i_{\ell+1}}$} (21.2,3.1);
\draw  (18.3,.7) to [out = -10,in=190] (21.2,0.7);
\draw  [fill=black](18.3,0.7) circle (.1);

\draw[black,fill=white] (21,6.7) circle (.25) node [font=\large] {$+$};
\draw[black,fill=white] (21,5.5) circle (.25) node [font=\large] {$+$};
\draw[black,fill=white] (21,3.1) circle (.25) node [font=\large] {$+$};
\draw[black,fill=white] (21,0.7) circle (.25) node [font=\large] {$+$};

\node [font=\huge] at (24.2,6.7) {$\cdots$};
\node [font=\huge] at (24.2,5.5) {$\cdots$};
\node [font=\huge] at (24.2,3.1) {$\cdots$};
\node [font=\huge] at (24.2,0.7) {$\cdots$};
\end{scope}

\draw [rounded corners, cyan, fill=cyan] (28,2.2) rectangle (30,4.9);
\node [above] at (32,3.9) {$\hat{y}_N$};
\draw [black] (32,3.7) -- (29.8, 3.7);
\draw [black,fill=black] (32,3.7) circle (.2);

\begin{scope}[lightgray, thick]
\draw  (28.4,3.7) to [out = 120,in=0] (25.2,6.7);
\draw  (28.4,3.7) to [out = 150,in=0] (25.2,5.5);
\draw [very thick,black] (28.4,3.7) to [out = 190,in=0] node [midway, above left, fill=white] {$\theta^{(L)}_{i_L,1}$} (25.2,3.1);
\draw  (28.4,3.7) to [out = -120,in=0] (25.2,0.7);
\end{scope}

\draw [rounded corners, black, fill=white] (28.2,4.2) rectangle (29.8,3.2);
\node [black] at (29,3.7) {$z^{L+1}_1$};

\draw[black,fill=white] (28,3.7) circle (.25) node [font=\large] {$+$};

\draw [decorate,decoration={brace,amplitude=10pt},yshift=-5pt]
(6,0) -- (0,0) node [black,midway,yshift=-20pt] {Layer 0};
\draw [decorate,decoration={brace,amplitude=10pt},yshift=-5pt]
(11.3,0) -- (6,0) node [black,midway,yshift=-20pt] {Layer 1};
\draw [decorate,decoration={brace,amplitude=10pt},yshift=-5pt]
(16,0) -- (11.3,0) node [black,midway,yshift=-20pt] {Layer 2};
\draw [decorate,decoration={brace,amplitude=10pt},yshift=-5pt]
(23,0) -- (18,0) node [black,midway,yshift=-20pt] {Layer $\ell$};
\draw [decorate,decoration={brace,amplitude=10pt},yshift=-5pt]
(30,0) -- (25,0) node [black,midway,yshift=-20pt] {Layer L};
\draw [decorate,decoration={brace,amplitude=10pt},yshift=-5pt]
(32,0) -- (30,0) node [black,midway,yshift=-20pt] {Layer L+1};
\end{tikzpicture}
}}
\caption{Visualization of a path}
\end{figure}

%% file: measureNN_aux/4-rigorous_results.tex
\subsection{Technical assumptions.}\label{sub:assumptions}

In order to properly analyse our model, we start by delineating a series of assumptions. These are meant to be as broad as possible, while still allowing
the use of known analytical tools. We start by making a hypothesis on the initialization of the Network.

\begin{assumption}[Initialization]
\label{assump:initialization}
At time $k=0$, the weights of the DNN
\[
	\theta^{(\ell)}_{i_\ell,i_{\ell+1}}(0)\in\R^{D_\ell}\,:\, 0\leq \ell\leq L,\;(i_\ell,i_{\ell+1}) \in [1:\boldN{\ell}{\ell+1}]
\]
are random and independent both from each other and from the i.i.d. sample $(X_k,Y_k)\sim P$ ($k\in\N$). Network weights with the same superscript $\ell$ have a common law given by a probability measure $\mu^{(\ell)}_0\in \Probspace(\R^{D_\ell})$ with finite first moment.
\end{assumption}

Basically, this assumption means that initialization is independent and weights in the same layer share a common law. The next assumption deals with the nature of the activation functions.

\begin{assumption}[Activation functions are bounded with bounded derivatives]
\label{assump:activations} There exist a constant $C>1$ with the following properties.

The data distribution $P$ of $(X,Y)\in\R^{d_X}\times \R^{d_Y}$ is such that $|Y|\leq C$ almost surely. Moreover, the functions $\sigma^{(\ell)}$,  $\alpha$, and their Fr\'{e}chet derivatives are all bounded by $C$ in the appropriate norms. In addition, each Fr\'{e}chet derivative of the $\sigma^{(\ell)}$ is also $C$-Lipschitz. \ignore{More precisely, the following inequalities hold for all $\ell \in [0:L+1]$ and all $(z,\theta),(z',\theta')\in\R^{d_{\ell}}\times\R^{D_\ell}$:
\begin{enumerate}
\item $|\sigma^{(\ell)}(z,\theta)|\leq C$ (where $|\cdot|$ is the Euclidean norm);
\item $\opnorm{D\sigma^{(\ell)}(z,\theta)}\leq C$ (where $\pnorm{\text{Op}}{\cdot}$ is the operator norm);
\item $\opnorm{D\sigma^{(\ell)}(z,\theta) - D\sigma^{(\ell)}(z',\theta')} \leq C\,(|z-z'| + |\theta - \theta'|)$.
\end{enumerate}
Finally, we assume that $D\sigma^{(L+1)}$ is bounded by $C$ with with Lipschitz constant also bounded by $C$ (we do {\em not} assume $\sigma^{(L+1)}$ itselt is bounded).}
\end{assumption}
Assumption \ref{assump:activations} is made for technical convenience, and excludes the popular ReLU nonlinearity. We expect that our results will extend with minor changes to more general activation functions.

\subsection{Main theorems.}

We now state the main results of the paper. The first one states that our McKean-Vlasov problem is well-posed.
\begin{theorem}[Existence and uniqueness of McKean-Vlasov solution; proof in Section \ref{sec:existence}]
\label{thm:independencestructure}
Under the assumptions in Section \ref{sub:assumptions}, the McKean-Vlasov problem in Section \ref{sub:McKeanVlasov} has a unique solution $\soln_{[0,T]}$. This solution has the following structure:
\begin{enumerate}
\item There exist deterministic functions \[\begin{array}{llllll}F^{(1)}&:&\R^{D_0}\times \R^{D_1}&\to& C([0,T],\R^{D_1});\\
F^{(\ell)}&:&\R^{D_\ell}&\to& C([0,T],\R^{D_\ell}), &\text{for } \ell\in [2:L-2]; \\
F^{(L-1)}&:&\R^{D_{L-1}}\times \R^{D_L}&\to&  C([0,T],\R^{D_{L-1}}),\end{array}\] such that, if $\Thetabar\sim \soln_{[0,T]}$, then almost surely, for all $t\in[0,T]$
\begin{align*}
	\Thetabar^{(1)}(t) =& F^{(1)}\parn{\Thetabar^{(0)}(0),\Thetabar^{(1)}(0)}(t),&\\
	\Thetabar^{(\ell)}(t) =& F^{(\ell)}\parn{\Thetabar^{(\ell)}(0)}(t),\quad\quad\quad\text{for }\ell\in [2:L-2],\\
	\Thetabar^{(L-1)}(t) =& F^{(L-1)}\parn{\Thetabar^{(L-1)}(0),\Thetabar^{(L)}(0)}(t).&
\end{align*}
\item As a consequence, $\soln_{[0,T]}$ factorizes into independent components, ie,
\[\soln_{[0,T]} = {\soln}^{(0,1)}_{[0,T]}\times\left(\prod_{\ell=2}^{L-2} {\soln}^{(\ell)}_{[0,T]} \right)\times {\soln}^{(L-1,L)}_{[0,T]}.\]
\end{enumerate}
\end{theorem}

Our second main results shows that we can describe the loss of our DNN with $N\gg 1$ particles in terms of the McKean-Vlasov limiting process.

\begin{theorem}[Approximation of errors; proof in Section \ref{sec:coupling}]\label{thm:error-approx}
Make the assumptions in the previous Section. Given $T>0$, let $\soln_{[0,T]}$ denote the unique solution of the McKean-Vlasov problem described in Subsection \ref{sub:McKeanVlasov}. Also, let $\left\{\btheta_N(k),\, k\in \left[0:\left\lceil\frac{T}{\eps}\right\rceil\right]\right\}$ be steps in the SGD process for the DNN described in Subsection \ref{sub:DNN_definition}.

Then, for all $k\in \left[0:\left\lceil\frac{T}{\eps}\right\rceil\right]$, it holds that:
\begin{equation}
\label{eq:error-general}
\Exp{\stackrel{(X_k,Y_k) \stackrel{\tiny{iid}}{\sim} P}{\btheta_N(0)\sim \mu_0}}{\abs{\L_N(\btheta_N(k)) - \Lbar(\soln_{k\eps})}}\leq C_{\ref{thm:error-approx}}\,\left(\sqrt{\frac{d}{{N}}}+\eps+\sqrt{\eps\,d}\right)
\end{equation}
where $d:=\max\{d_\ell,\, \ell\in[0:L+1]\}$ and $C_{\ref{thm:error-approx}}$ depends only on $C$, $T$, and the number of hidden layers $L$.
\end{theorem}

In fact, our proofs also give a "microscopic" approximation of SGD, whereby we approximate individual weights of the network. 
\begin{theorem}[Approximation by ideal particles; proof in Section \ref{sec:coupling}]\label{thm:microscopic} Under the assumptions in Section \ref{sub:assumptions}, consider $\soln_{[0,T]}$ as in Theorem \ref{thm:independencestructure}, and let \[\left\{\btheta_N(k),\, k\in \left[0:\left\lceil\frac{T}{\eps}\right\rceil\right]\right\}\]
be as in Theorem \ref{thm:error-approx}. Then one can couple the above evolution with certain "ideal particles"
\[\left\{\bthetabar_N(t)\in\R^{p_N},\, t\in [0,T]\right\}\]
so that:
\begin{enumerate}
\item the trajectories $\thetabar^{(\ell)}_{i_\ell,i_{\ell+1},[0,T]}$ for $\ell\in [0:L]\backslash\{1,L-1\}$ and $(i_\ell,i_{\ell+1})\in [1:N_\ell]\times [1:N_{\ell+1}]$ are all independent, and each $\thetabar^{(\ell)}_{i_\ell,i_{\ell+1},[0,T]}\sim {\soln}^{(\ell)}_{[0,T]}$;
\item conditionally on the above, $\thetabar^{(\ell)}_{i_\ell,i_{\ell+1},[0,T]}$ for $\ell\in \{1,L-1\}$ and $(i_\ell,i_{\ell+1})\in [1:N_\ell]\times [1:N_{\ell+1}]$ are independent, with each \[\thetabar^{(1)}_{i_1,i_{2},[0,T]}\sim {\soln}_{[0,T]}^{(1\mid 0)}(\cdot\mid  \thetabar^{(0)}_{1,i_{1},[0,T]})\] and each \[\thetabar^{(L-1)}_{i_{L-1},i_{L},[0,T]}\sim {\soln}_{[0,T]}^{(L-1\mid L)}(\cdot\mid  \thetabar^{(L)}_{i_L,1,[0,T]});\]
\item finally, for each choice of input-to-output path $i_0,i_1,\dots,i_L,i_{L+1}$ with $i_{0}=i_{L+1}=0$ and $i_\ell\in [1:N_\ell]$ for $\ell\in [1:L]$, we have:
\[(\thetabar^{(0)}_{i_\ell,i_{\ell+1},[0,T]}\,:\, \ell\in [0:L])\sim \soln_{[0,T]},\]
and for all $k\in [0: \lceil T/\eps\rceil ]$:
\[\Exp{\stackrel{(X_k,Y_k) \stackrel{\tiny{iid}}{\sim} P}{\btheta_N(0)\sim \mu_0}}{\sum_{\ell=0}^L|\theta^{(\ell)}_{i_\ell,i_{\ell+1}}(k) - \thetabar^{(\ell)}_{i_\ell,i_{\ell+1}}(k\eps)|}\leq C_{\ref{thm:microscopic}}\,\left(\sqrt{\frac{d}{{N}}}+\eps+\sqrt{\eps\,d}\right),\]
where $d:=\max\{d_\ell,\, \ell\in[0:L+1]\}$ and $C_{\ref{thm:error-approx}}$ depends only on $C$, $T$, and the number of hidden layers $L$.
\end{enumerate}\end{theorem}

\begin{remark}[Propagation of chaos: weights versus paths]\label{rem:propagationofchaos} One key difference between our system and and other interacting particle systems of mean-field type (eg. as popularized by \cite{Sznitman1991}) is that network weights do {\em not} become independent in the limit, as even ideal particles have nontrivial dependencies. However, disjoint {\em paths of weights} will become independent. This implies a propagation of chaos for weight vectors over randomly chosen paths:a constant number of them will become asymptotically independent. This is in keeping with the point made in Section \ref{sec:setup} where we noted that paths are the fundamental units of our system.\end{remark}

\subsection{More comments on related results} \label{sub:comparison} As noted, there are two recent preprints that tackle similar problems to our main results. Nguyen \cite{Nguyen2019} has essentially the same ansatz and the same end result as we do, but he uses nonrigorous methods. 

By contrast, Sirignano and Spiliopoulos \cite{Sirignano2019} consider networks with $L\geq 3$ and $N_\ell$ units in each layer, but then make the $N_\ell$ diverge one at a time. As they observe in \cite[\S 2.2]{Sirignano2019}, to consider all $N_\ell$ diverging simultaneously, one alternative would be to write an empirical measure for the weights, and then write the dynamics of SGD as a closed dynamics over empirical measures. This would be the natural extension of work on the shallow case \cite{MeiArxiv2018,Rotskoff2018,Sirignano2018}. However, Sirignano and Spiliopoulos note that there are difficulties with this and related ideas. 

Our own work seems to identify one important source of the above difficulties. The most natural measure to consider is over paths of weights. Then, however the closed evolution obtained turns out to involve conditional measures, and conditioning is a discontinuous operation. In our approach, we have {\em not} tried to write a single empirical measure for all weights. Rather, our approach follows the ideas outlined in the next section.

%% file: measureNN_aux/5-overview.tex
In this section we present an outline for the proofs of our main results, which are further described in the subsequent sections. We start by summarizing the proof for existence and uniqueness of our limiting McKean-Vlasov problem (Theorem \ref{thm:independencestructure}). Next, we overview the results from sections \ref{sec:SGDtoCTGD} to \ref{sec:coupling}, showing that our strategy for training DNNs approaches the trajectories obtained from our McKean-Vlasov problem (Theorems \ref{thm:error-approx} and \ref{thm:microscopic}). 

\subsection{Existence and uniqueness of the McKean-Vlasov problem for DNNs}\label{sub:overview-existence}

Our main strategy for proving existence and uniqueness for the problem described in Definition \ref{def:McKeanVlasov} follows a standard path through a fixed point problem. First, we define an operator $\psi$ on the space of probability measures over $C([0,T],\R^D)$, such that fixed points correspond to a solution of the McKean-Vlasov problem. We then show that iterates of this operator contract the Wasserstein metric, proving existence and uniqueness of a fixed point. 

The details are more complicated: technical reasons will force us to consider a "special" subset of probability measures over $C([0,T],\R^D)$. From Definition \ref{def:McKeanVlasov}, we consider the natural candidate for the map $\psi$. For $\nu_{[0,T]}$, a measure over $C([0,T],\R^D)$, we define $\psi(\nu_{[0,T]})$ as the law of the process $\Thetabar\in C([0,T],\R^D)$  which solves:
\begin{equation}\label{eq:psi-iteration}
  \begin{cases}
  \Thetabar(0) \sim \mu_0;\\
  \Thetabar(t) = \Thetabar(0) - \displaystyle{\int_0^t}\alpha(s)\Exp{(X,Y)\sim P}{\Grad^{(\ell)}(X,Y,\Thetabar(t),\nu_t)} ds,& t\in[0,T].
  \end{cases}
\end{equation}

Section \ref{sec:apriori_simple} is where is where we make our first restriction to the set of measures. Solutions for our McKean-Vlasov problem follow Lipschitz trajectories through time (see Lemma \ref{lem:trivialsoln}). This property is also enough to guarantee that $\psi$ is well-defined, i.e. that the system (\ref{eq:psi-iteration}) has a unique solution (Corollary \ref{cor:mappingwelldefined}). Lemma \ref{lem:constrainstructure}, where the above statement is in fact proved, is a consequence of the Picard-Lindel\"{o}f Theorem. 

Even though our map is well defined, standard ODE results are not enough to ensure that we can iterate our map and find a fixed point. This is mostly due to the particularity of the dependencies we have for the outermost layers. Therefore, we need to make sure that the output law from the map $\psi$ also behaves well, i.e. that it still maintains Lipschitz continuity in the initial conditions.

With this in mind, we further restrict our set of measure to ensure good behavior. In Section \ref{sec:apriori_space}, we consider measures $\mu_{[0,T]}\in \Probspace(C([0,T],\R^D))$ that are "special" (cf. Definition \ref{def:Rspecialmeasure}) A special measure has a "special" function $F_\mu$ (Definition \ref{def:Rspecialfunction}) with certain quantitative Lipschitz properties rsuch that, if $\Theta\sim \mu_{[0,T]}$, for all $t\in[0,T]$
\[\Theta^{(L-1)}(t) = F_\mu\parn{\Theta^{(L-1)}(0),\Theta^{(L)}(0)}(t)\mbox{ almost surely.}\]
This set of measures is topologically closed (Theorem \ref{thm:Probspecialisclosed}).

A series of results then show that this set of special measures is closed under $\psi$ (Lemma \ref{lem:specialclosedpsi}), and closed under the Wasserstein $L^1$ metric (Theorem \ref{thm:Probspecialisclosed}). One key dea is to show that, for special measures, $\zbar^{(L)}$ is also Lipschitz continuous with respect to the Wasserstein metric (Lemma \ref{lem:estimatezbarL}). So the discontinuity issues in our problem disappar. We will also need a series of technical estimates on the Cauchy problem for the ODEs considered above.

In Section \ref{sec:existence} we finally show that $\psi$ is a contraction on the set of special measures defined in Section \ref{sec:apriori_space}. Since this set is invariant for $\psi$, contains all solutions and is topologically closed, existence and uniqueness follow. We emphasize that finding this "special set of measures" is the key idea that makes the fixed point method work in our setting.

\subsection{Approximation of loss function and ideal particles} The McKean-Vlasov problem is our guess for the process that describes solutions. To show that it is the right guess, we must couple discrete-time SGD to the continuous-time "ideal particles" described in Theorem \ref{thm:microscopic}, and compare the corresponding values of the loss function.

The first step, done in Section \ref{sec:SGDtoCTGD}, is to compare SGD to a continuous-time gradient flow for network weights. This same idea appears in \cite{MeiArxiv2018}, but the details here are more complicated.

The second step is to introduce the "ideal particles" from Theorem \ref{thm:microscopic} to which we will compare the continuous-time network weights. Section \ref{sec:ideal} carries our the first steps of this task. A crucial point here is that this: the specific dependency structure of the McKean-Vlasov solution allows us to "glue" the ideal particles to the edges of the DNN so that one "sees" the McKean-Vlasov solution along each path. In this way, we can reduce the problem of finding the mean-field approximation of the DNN to a coupling of real and ideal weights.

Finally, Section \ref{sec:coupling} shows that ideal particles give good approximations to network weights and allow one to compare the loss of the true DNN to a loss associated to the McKean-Vlasov problem.

%% file: measureNN_aux/6-apriori_simple.tex
In this section, we prove some preliminary facts and estimates that any solution to our McKean-Vlasov problem must satisfy. In Section \ref{sub:trivialsoln}, we start with certain trivial properties that restrict which measures could potentially be solutions. In Section \ref{sub:towardsfixedpoint}, we show that, if we apply the "natural" iteration to a potential solution, then the corresponding map is well-defined and satisfies certain properties.

\subsection{A preliminary result.}\label{sub:trivialsoln}
Our first lemma is fairly straightforward. We start with the following definition.

\begin{definition}[Preliminary set of measures]\label{def:prelimset}Let $\mu_0 = \prod_{\ell=0}^{L}\mu^{(\ell)}_0\in\Probspace(\R^D)$ be defined as the product of the measures in Assumption \ref{assump:initialization}. Given $R>0$, we define\[\Probspace_R\subset \Probspace(C([0,T],\R^D))\] as the set of all measures $\nu_{[0,T]}$ with the following properties: if $\Thetabar_{[0,T]}\sim \nu_{[0,T]}$, then:
\begin{enumerate}
\item $\Thetabar(0)\sim \mu_0$;
\item $\Thetabar^{(0)}_{[0,T]}$ and $\Thetabar^{(L)}_{[0,T]}$ are a.s. constant;
\item each $\Thetabar^{(\ell)}_{[0,T]}$ is an $R$-Lipschitz as a function of time.
\end{enumerate}
\end{definition}

\begin{lemma}
\label{lem:trivialsoln}
Under assumptions in section \ref{sub:assumptions}, there exists a constant $R_{\ref{lem:trivialsoln}} = R_{\ref{lem:trivialsoln}}(C,L) > 0$ such that if  $\soln_{[0,T]}$ is a solution of the McKean-Vlasov problem in Definition \ref{def:McKeanVlasov}, then it must belong to $\Probspace_R$ for any $R > R_{\ref{lem:trivialsoln}}$.
\end{lemma}
\begin{proof}
This is an easy consequence of the definition of the problem. Indeed, the initial distribution of $\Thetabar(0)$ and the fact that $\Thetabar^{(0)}_{[0,T]}$ and $\Thetabar^{(L)}_{[0,T]}$ are constant are part of the definition. The Lipschitz property comes from the fact that the functions $\alpha$ and $\Grad^{(\ell)}$ are bounded by a constant depending on $C$ and $L$.
\end{proof}

\subsection{A mapping and its structure}\label{sub:towardsfixedpoint}

The next step is to define a map $\psi:\Probspace_R\to \Probspace_R$ with the following property:
\[\soln_{[0,T]}\in \Probspace_R \mbox{ solves the McKean-Vlasov problem} \Leftrightarrow \psi(\soln_{[0,T]})=\soln_{[0,T]}.\]

A natural candidate to be the map $\psi$ is as follows: Given $\nu_{[0,T]}$, define $\psi(\nu_{[0,T]})$ as the law of the process $\Thetabar\in C([0,T],\R^D)$  which satisfies:
\[\begin{cases}
\Thetabar(0) \sim \mu_0;\\
\Thetabar(t) = \Thetabar(0) - \displaystyle{\int_0^t}\alpha(s)\Exp{(X,Y)\sim P}{\Grad^{(\ell)}(X,Y,\Thetabar(t),\nu_t)} ds,& t\in[0,T].
\end{cases}\]
The following lemma shows that this operation is well-defined and has some additional structure.

\begin{lemma}
\label{lem:constrainstructure}
Make the assumptions in Section \ref{sub:assumptions}. Let $\nu_{[0,T]}\in\Probspace_R,$ where $R\geq R_{\ref{lem:trivialsoln}}$. Then, one can find measurable mappings:
\[
\begin{array}{rll}
 G_\nu^{(1)}:&\R^{D_{0}}\times \R^{D_1}\to C([0,T],\R^{D_{1}})\\
  \mbox{ with }&G_\nu^{(1)}(\theta^{(0)},\theta^{(1)})(0)=\theta^{(1)},&\forall (\theta^{(0)},\theta^{(1)})\in \R^{D_{0}}\times \R^{D_1};\\
 G^{(\ell)}_\nu :&\R^{D_{\ell}}\to C([0,T],\R^{D_{\ell}}),  & \forall\ell\in [2:L-2],\\
 \mbox{ with }&G_\nu^{(\ell)}(\theta^{(\ell)})(0)=\theta^{(\ell)}, & \forall \theta^{(\ell)}\in \R^{D_\ell};\\
 G^{(L-1)}_\nu :&\R^{D_{L-1}}\times \R^{D_L}\to C([0,T],\R^{D_{L-1}})\\
 \mbox{ with }&G_\nu^{(L-1)}(\theta^{(L-1)},\theta^{(L)})(0)=\theta^{(L-1)}, & \forall (\theta^{(L-1)},\theta^{(L)})\in \R^{D_{L-1}}\times \R^{D_L}
 \end{array}
 \]
 with the following properties:
\begin{enumerate}
 \item The range of each $G^{(\ell)}_{\nu}$ consists of $R$-Lipschitz trajectories.
 \item For $\mu_0$-almost every choice of $\theta=(\theta^{(\ell)}:\,\ell\in[0:L])$, the
unique solution to the initial value problem
\begin{equation}\label{eq:existenceproblem}
\begin{cases}
\Thetabar(0) = \theta,\\
\dfrac{d\Thetabar}{dt}(t) =  -\alpha(t)\cdot\Ex{\Grad^{(\ell)}(X,Y,\Thetabar(t),\nu_t)}, & t\in[0,T]
\end{cases}
\end{equation}
is given by
\begin{equation}
\begin{array}{rclrcl}
\Thetabar^{(0)}(t) &=& \theta^{(0)} \quad\mbox{ (constant)}& \Thetabar^{(L-1)}(t) &=& G_\nu^{(L-1)}(\theta^{(L-1)},\theta^{(L)})(t)\\
\Thetabar^{(1)}(t) &=& G_\nu^{(1)}(\theta^{(0)},\theta^{(1)})(t)&\Thetabar^{(L)}(t) &= & \theta^{(L)}\quad\mbox{ (constant)}\\
\Thetabar^{(\ell)}(t)&=& G^{(\ell)}_\nu(\theta^{(\ell)})(t) \quad\quad\text{ for }&\ell\in [2:L-2].
\end{array}\\
\end{equation}
\end{enumerate}
\end{lemma}

An immediate consequence of the Lemma \ref{lem:constrainstructure} is that the map $\psi$ discussed above is well-defined.

\begin{corollary}[$\psi$ is well defined; proof omitted]
\label{cor:mappingwelldefined}
In the setting of Lemma \ref{lem:constrainstructure}, there exists  $R_{\ref{lem:trivialsoln}}>0$ such that, if $R\geq R_{\ref{lem:trivialsoln}}$, then there exists a map $\psi:\Probspace_{R}\to \Probspace_R$ defined as follows: Given $\nu_{[0,T]}\in\Probspace_R$, then $\psi(\nu_{[0,T]})$ is the unique measure with $\mu_0$ as its time $0$ marginal, and therefore if $\Thetabar \sim \psi(\nu_{[0,T]})$, then almost surely $\Thetabar$ solves the initial value problem \ref{eq:existenceproblem}.
\end{corollary}

\begin{proof}[Proof of Lemma \ref{lem:constrainstructure}] The result stated in the lemma follows from the fact that the ODE (\ref{eq:existenceproblem}) has a unique solution for any starting state. To clarify this statement, a few steps are required.

First, we observe that $\Thetabar^{(\ell)}(t)=\Thetabar^{(\ell)}(0)$ holds for $\ell\in \{0,L\}$ as a consequence of Definition \ref{def:Grad-ell}.

Next, we remark that, for any $\ell\in [0:L]$ and {\em any} measure $\eta_{[0,T]}\in\Probspace(C([0,T],\R^D))$, the time-$t$ marginals $\eta^{(\ell)}_t$ depend continously on $t$ (in the weak topology). To see this, it suffices to note that, for any bounded continuous $f:\R^{D_{\ell}}\to \R$,
\[
\int\limits_{\R^{D_\ell}}\,f\parn{a^{(\ell)}}\,d\eta_t^{(\ell)}(a^{(\ell)}) = \Exp{\Theta\sim \eta_{[0,T]}}{f\parn{\Theta^{(\ell)}(t)}},
\]
and the integrand in the RHS is a bounded continuous function of $t$. Therefore, the expectation changes continuously with $t$ by Bounded Convergence.

One can use this to show that the quantities $\zbar^{(\ell)}(x,\nu_t)$ for $\ell\in [2:L-2]$ defined in (\ref{eq:defzbarell}) all vary continuously with $t$. The situation for $\ell=L$ is different: we have the conditional measure $\nu^{(L-1\mid L)}_t$ showing up. However, since $\Theta^{(L)}(t)=\Theta^{(L)}(0)$ a.s., we have that
\begin{equation}
\nu_t^{(L-1\mid L)}(\cdot\mid \Theta^{(L)}(t)) = \Pr{\Theta^{(L-1)}(t)\in \cdot\mid \Theta^{(L)}(0)}\,\mbox{a.s., }\,\,\forall t\in [0,T],
\end{equation}
and we obtain
\[\zbar^{(L)}(x,\mu_t,\Thetabar^{(L)}(t)) =\Ex{\fsigma{L-1}{\zbar^{(L-1)}(x,\mu_t),\Thetabar^{(L-1)}(t)}\mid \Thetabar^{(L)}(0)}.\]
Therefore, on almost every choice $\theta^{(L)}$ of $\Thetabar^{(L)}(0)$, the function $\zbar^{(L)}(x,\mu_t,\Thetabar^{(L)}(t))$ is continuous in $t$. We may also show that $\zbar^{(L+1)}(x,\nu_t)$ is continuous in $t$. Using the same reasoning, we may see that the $\Mbar^{(\ell)}$ all depend continuously on $t$.

Let us now write our system of ODEs in (\ref{eq:existenceproblem}) for each $\ell\in [1:L-1]$. Let us consider first the case where $\ell=1$. Using that $\Thetabar^{(0)}(0)=\Thetabar^{(0)}(t)$ holds a.s. for $0\leq t\leq T$, the ODE for $\Thetabar^{(1)}$ takes the following form in a measure-one set\footnote{It is useful to revisit the Definitions \ref{def:Grad-ell}, \ref{eq:defgbar1},\ref{eq:defgbarell},\ref{eq:defgbarL-1} to understand equations below}:
\[
\frac{d}{dt}\Thetabar^{(1)}(t) =- \alpha(t)\,\Ex{(Y-\ybar(X,\nu_t))^\dag\, \Mbar^{(2)}(x,\nu_t)\,D_\theta\fsigma{1}{\sigma^{(0)}(X,\Thetabar^{(0)}(0)),\Thetabar^{(1)}(t)}}.
\]
Since the function inside the expected value in the RHS is bounded and continuous in $t$, the RHS is itself continuous in $t$. Moreover, it is also a globally Lipschitz function of $\Thetabar^{(0)}(0)$ and $\Thetabar^{(1)}(t)$, with a constant that comes from the boundedness of $\Mbar^{(2)}$ and the assumption that $\sigma^{(1)}$ has Lipschitz derivative. The Picard-Lindel\"{o}f Theorem implies that the trajectory $\Thetabar^{(1)}$ is a deterministic function $G_\nu^{(1)}$ of $\Thetabar^{(0)}(0)$ (which appears as a parameter in the RHS) and of the initial condition $\Thetabar^{(1)}(0)$. Under our conditions, the solution to the ODE is the uniform limit of Euler-Maruyama discretizations. The recursive construction of these discretizations means that each of them is measurable with respect to $(\Thetabar^{(0)}(0),\Thetabar^{(1)}(0))$, and the same is true of the limit $G_\nu^{(1)}$. Finally, the Lipschitz property of trajectories comes from the fact that the RHS of the ODE is bounded.

A similar reasoning  applies to each $\ell\in [2:L-2]$. In this case the differential equation is
\[\frac{d}{dt}\Thetabar^{(\ell)}(t) = -\alpha(t)\,\Ex{(Y-\ybar(X,\nu_t))^\dag\, \Mbar^{(\ell+1)}(x,\nu_t)\,D_\theta\sigma^{(\ell)}(\zbar^{(\ell)}(X,\nu_t),\Thetabar^{(\ell)}(t))}.\]
Again the RHS is continuous in $t$ and Lipschitz in $\Thetabar^{(\ell)}(t)$, which means that the solution is a deterministic function $G_\nu^{(\ell)}$ of the initial condition $\Thetabar^{(\ell)}(0)$. The function is measurable for being a limit of discretizations, as we argued above. The Lipschitz property of trajectories also follows naturally.

For $\ell=L-1$, using that $\Thetabar^{(L)}(t)=\Thetabar^{(L)}(0)$ a.s. for all $0\leq t\leq T$, we have the ODE:
\begin{align}
\label{eq:ODETheta-L}
\nonumber\frac{d}{dt}\Thetabar^{(L-1)}(t) =& -\alpha(t)\,\mathbb{E}\left[(Y-\ybar(X,\nu_t))^\dag\cdot\Mbar^{(L)}\parn{x,\nu_t,\Thetabar^{(L)}(0)}\right.\\
&\quad\quad\quad\quad\quad\quad\cdot\,\left.D_\theta\fsigma{L-1}{\zbar^{(L-1)}(X,\nu_t),\Thetabar^{(L-1)}(t)}\right].
\end{align}
The RHS is continuous in $t$ and Lipschitz in $\Thetabar^{(L-1)}(t)$, so the same reasoning we used for $\ell=1$ still works and we can find a $G_\nu^{(L-1)}$ with the desired properties.

Finally, the fact that \ref{eq:existenceproblem} has a unique solution for almost every initial position clearly implies that $\psi(\nu_{[0,T]})$ is uniquely specified.\end{proof}

%% file: measureNN_aux/7-apriori_space.tex
In this section we continue the work of Section \ref{sec:apriori_simple}. Corollary \ref{cor:mappingwelldefined} shows that there is a well-defined map $\psi$ whose fixed points are solutions of the McKean Vlasov problem. It is then natural to iterate that map to obtain existence and uniqueness of solutions.

To make this work, we will need some analytic properties of the functions $G_{\nu}^{(\ell)}$ appearing in Lemma \ref{lem:constrainstructure}. For $\ell\neq L-1$ one could show  that these functions are Lipschitz, which is enough for our purposes. However, the function $G_{\nu}^{(L-1)}$ is a bit more delicate because, in principle, it could be that the dependence on $\Thetabar^{(L)}(0)$ is not "nice". This is related to Remark \ref{rem:discontinuity}.

To circumvent such problems, we will show that we can further constrain the set of potential solutions of our problem. We will define the set of $R$-special measures (Definition \ref{def:Rspecialmeasure}) so that a mesure $\nu_{[0,T]}$ in this set will have special Lipschitz properties in its $(L-1)^\text{th}$ coordinate. The results in this section will imply that this set is well-behaved with respect to $\psi$ and that the search for solutions can be narrowed down to this set.

To make all of this precise, we need two new definitions. The first defines the property we want the $(L-1)^\text{th}$ coordinate process to obey.

\begin{definition}[$R$-special function]\label{def:Rspecialfunction}Given $R>0$, we say that \[F:\R^{D_{L-1}}\times \R^{D_L}\to C([0,T],\R^{D_{L-1}})\] is {\em $R$-special} if:
\begin{enumerate}
\item For any $a=(a^{(L-1)},a^{(L)})\in \R^{D_{L-1}}\times \R^{D_L}$, the trajectory
\[t\in [0,T]\mapsto F(a)(t)\in \R^{D_{L-1}}\]
is $R$-Lipschitz in $t$, and $F(a)(0)=a^{(L-1)}$.
\item For any $t\in [0,T]$, the mapping
\[a=(a^{(L-1)},a^{(L)})\in \R^{D_{L-1}}\times \R^{D_L}\mapsto F(a)(t)\in \R^{D_{L-1}}\]
is $e^{Rt}$-Lipschitz as a function of $a$.\end{enumerate}
\end{definition}

The second definition defines the set of probability measures we will work with.

\begin{definition}[$R$-special measure]\label{def:Rspecialmeasure} A measure \[\mu_{[0,T]}\in \Probspace(C([0,T],\R^D))\] is said to be $R$-special if $\mu_{[0,T]}\in \Probspace_R$ and there exists an $R$-special function $F_\mu$ such that, if $\Theta\sim \mu_{[0,T]}$, then almost surely, for all $t\in[0,T]$:
\[\Theta^{(L-1)}(t) = F_\mu\parn{\Theta^{(L-1)}(0),\Theta^{(L)}(0)}(t).\]
The set of all $R$-special measures is denoted by $\Probspecial_R$.\end{definition}

Our first result on $R$-special measures is that the discontinuity issues raised in Remark \ref{rem:discontinuity} will not matter when we consider these measures. In fact, we canestimate the difference between the values of $\zbar^{(L)}$ for different measures via their Wasserstein distance. That will be accomplished by the next Lemmas.

\begin{lemma}[Proof in \S \ref{sub:proofredefiningzL}]
\label{lem:redefiningzL}
When $\nu_{[0,T]}\in \Probspecial_R$, we have that, for $\mu_0^{(L)}$-almost every $a^{(L)}\in \R^{D_L}$ and $t\in [0,T]$:
\[
\zbar^{(L)}(x,\nu_t,a^{(L)}) =\int\limits_{\R^{D_{L-1}}}\fsigma{L-1}{\zbar^{(L-1)}(x,\nu_t),F_\nu\parn{\tilde{a}^{(L-1)},a^{(L)}}(t)}\,d\mu_0^{(L-1)}\parn{\tilde{a}^{(L-1)}}.
\]
\end{lemma}

\begin{lemma}[Proof in \S \ref{sub:proofestimatezbarL}]\label{lem:estimatezbarL}When $\tilde{\nu}_{[0,T]},\nu_{[0,T]}\in \Probspecial_R$, we have that, for all $t\in [0,T]$ and $x\in\R^{d_X}$:
\[\int\limits_{\R^{D_L}}\abs{\zbar^{(L)}(x,\tilde{\nu}_t,a^{(L)}) - \zbar^{(L)}(x,\nu_t,a^{(L)})}\,d\mu^{(L)}_0\parn{a^{(L)}}\leq C\cdot(e^{Rt}+1)\,W(\tilde{\nu}_{[0,t]},\nu_{[0,t]}).\]\end{lemma}

The next two Lemmas are two of our main results about $R$-special measures. Lemma \ref{lem:specialclosedpsi} shows that this set is closed under the mapping $\psi$ while Lemma \ref{lem:specialcontainssoln} that all solutions must be $R$-special.

\begin{lemma}[Special measures are closed under $\psi$; Proof in \S \ref{sub:proofpsispecial}]
\label{lem:specialclosedpsi} There exists a value $R_{\ref{lem:specialclosedpsi}}>0$ depending only on $C$ and $L$ such that, for all $R\geq R_{\ref{lem:specialclosedpsi}}$, the mapping $\psi$ introduced in Corollary \ref{cor:mappingwelldefined} satisfies $\psi(\Probspecial_R)\subset \Probspecial_R$.
\end{lemma}

\begin{lemma}[All solutions are special; Proof in \S \ref{sub:proofspecialcontainssoln}]
\label{lem:specialcontainssoln}
There exists a value $R_{\ref{lem:specialcontainssoln}}$ depending only on $C$ and $L$ such that, for all $R\geq R_{\ref{lem:specialcontainssoln}}$, any solution to the McKean-Vlasov problem belongs to $\Probspecial_{R}$.
\end{lemma}

Finally, we mention an important technical result, to be proved in Appendix \ref{sec:Probspecialisclosed}.

\begin{theorem}[Special measures form a closed set]
\label{thm:Probspecialisclosed}
The set $\Probspecial_R$ is closed under weak convergence (and thus also under Wasserstein $L^1$ metric).
\end{theorem}

The remainder of the section contains the proofs of the above Lemmas.

\subsection{Proof of a formula for $\zbar^{(L)}$ with a $R$-special measure}
\label{sub:proofredefiningzL}
\begin{proof}[Proof of Lemma \ref{lem:redefiningzL}] 
Let $\Theta \sim \nu_{[0,T]}\in\Probspecial_R$. Then $\Theta^{(L)}$ is a.s. constant and
\[\Pr{\forall t\in [0,T]\,:\,
	\Theta^{(L-1)}(t) = F_\nu\parn{\Theta^{(L-1)}(0),\Theta^{(L)}(0)}(t)}=1.
\]
We also have
\[
	(\Theta^{(L-1)}(0),\Theta^{(L)}(0)) = (\Theta^{(L-1)}(0),\Theta^{(L)}(t))\sim \mu_0^{(L-1,L)} = \mu_0^{(L-1)}\times \mu_0^{(L)}.
\]
Therefore,
\[
	\mu_t^{(L-1\mid L)}(\cdot\mid a^{(L)}) =
	\int\limits_{\R^{D_{L-1}}}\,{\boldsymbol \delta}_{F_\nu(a^{(L-1)},a^{(L)})(t)}\,d\mu_0^{(L-1)}(a^{(L-1)})\mbox{ for all $t\in[0,T]$, a.s., }
\]
where ${\boldsymbol \delta_\cdot}$ is a Dirac delta. Plugging this into Equation \eqnref{defzbarL} which defines $\zbar^{(L)}$ finishes the proof.
\end{proof}

\subsection{Differences of $\zbar^{(L)}$ for $R$-special measures}
\label{sub:proofestimatezbarL}
We now show how to estimate differences in $\zbar^{(L)}$ for $R$-special measures.
\begin{proof}[Proof of Lemma \ref{lem:estimatezbarL}]
Recall from Lemma \ref{lem:redefiningzL} that:
\[\zbar^{(L)}(x,\nu_t,a^{(L)}) =\int\limits_{\R^{D_{L-1}}}\fsigma{L-1}{\zbar^{(L-1)}(x,\mu),F_\nu\parn{\tilde{a}^{(L-1)},a^{(L)}}(t)}\,d\mu_0^{(L-1)}(\tilde{a}^{(L-1)})\]
and similarly for $\zbar^{(L)}(x,\tilde{\nu}_t,a^{(L)})$. Since $\sigma^{(L-1)}$ is $C$-Lipschitz, then
\begin{align*}
&\abs{\zbar^{(L)}(x,\nu_t,a^{(L)}) - \zbar^{(L)}(x,\tilde{\nu}_t,a^{(L)})} \leq \\
&\quad\quad\quad C\cdot\int\limits_{R^{D_{L-1}}}\abs{F_\nu\parn{\tilde{a}^{(L-1)},a^{(L)}}(t) - F_{\tilde{\nu}}\parn{\tilde{a}^{(L-1)},a^{(L)}}(t)}\,d\mu^{(L-1)}_0(\tilde{a}^{(L-1)}).
\end{align*}
It follows that, in order to prove the Lemma, it suffices to show that:
\begin{equation}
\label{eq:goalestimatezbarL}
\mbox{\bf Goal: } \int\limits_{\R^{D_{L-1}}\times \R^{D_{L}}}|F_\nu(a)(t) - F_{\tilde{\nu}}(a)(t))|\,d\mu^{(L-1,L)}_0(a)\leq (e^{Rt}+1)\cdot W(\nu_{[0,t]},\tilde{\nu}_{[0,t]}),
\end{equation}
where as usual we write $a=(a^{(L-1)},a^{(L)})\in\R^{D_{L-1}}\times \R^{D_L}$.

To achieve this goal, let $\Theta\sim \nu_{[0,t]}$, $\tilde{\Theta}\sim \tilde{\nu}_{[0,t]}$ be a coupling of $ \nu_{[0,t]}$ and $\tilde{\nu}_{[0,t]}$ that achieves Wasserstein distance:
\[
\Ex{\sup_{0\leq r\leq t}|\Theta(r) - \tilde{\Theta}(r)|}=W(\nu_{[0,t]},\tilde{\nu}_{[0,t]}).
\]
The quantity in the LHS of (\ref{eq:goalestimatezbarL}) may be rewritten as:
\[
(I):=\Ex{\abs{F_\nu\parn{\Theta^{(L-1)}(0),\Theta^{(L)}(0)}(t) - F_{\tilde{\nu}}\parn{\Theta^{(L-1)}(0),\Theta^{(L)}(0)}(t)}}.
\]
Moreover, because the two measures are special, we have the almost-sure identities:
\begin{align*}
F_\nu\parn{\Theta^{(L-1)}(0),\Theta^{(L)}(0)}(t) =& \Theta^{(L-1)}(t);\\
F_{\tilde{\nu}}\parn{\tilde{\Theta}^{(L-1)}(0),\tilde{\Theta}^{(L)}(0)}(t) =& \tilde{\Theta}^{(L-1)}(t).
\end{align*}
We also know that $F_{\tilde{\nu}}$ is an $R$-special function, which implies that $F_{\tilde{\nu}}(\cdot)(t)$ is $e^{Rt}$-Lipschitz (as per Definition \ref{def:Rspecialfunction}). In particular, we have that:
\[
\abs{F_{\tilde{\nu}}\parn{\tilde{\Theta}^{(L-1)}(0),\tilde{\Theta}^{(L)}(0)}(t) - F_{\tilde{\nu}}\parn{\Theta^{(L-1)}(0),\Theta^{(L)}(0)}(t)}\leq e^{Rt}\cdot\abs{\Theta(0) - \tilde{\Theta}(0)}.
\]
We conclude that:
\begin{align*}
(I) \leq & \Ex{\abs{F_\nu\parn{\Theta^{(L-1)}(0),\Theta^{(L)}(0)}(t)-F_{\tilde{\nu}}\parn{\tilde{\Theta}^{(L-1)}(0),\tilde{\Theta}^{(L)}(0)}(t)}} \\
   &\quad +  \Ex{\abs{F_{\tilde{\nu}}\parn{\Theta^{(L-1)}(0),\Theta^{(L)}(0)}(t)-F_{\tilde{\nu}}\parn{\tilde{\Theta}^{(L-1)}(0),\tilde{\Theta}^{(L)}(0)}(t)}}\\
  \leq & \Ex{\abs{\Theta^{(L-1)}(t)-\tilde{\Theta}^{(L-1)}(t)}} + e^{Rt}\cdot\Ex{\abs{\Theta(0) - \tilde{\Theta}(0)}}\\
  \leq & (e^{Rt}+1)\cdot\Ex{\sup_{0\leq r\leq t}\abs{\Theta(r) - \tilde{\Theta}(r)}}= (e^{Rt}+1)\cdot W(\nu_{[0,t]},\tilde{\nu}_{[0,t]}),
\end{align*}
which achieves Goal \eqref{eq:goalestimatezbarL}.
\end{proof}

\subsection{Proof that special measures are closed under $\psi$.}\label{sub:proofpsispecial}
Now we move on to prove Lemma \ref{lem:specialclosedpsi}, which shows that the map $\psi$ maps $\Probspecial_R$ onto itself, assuming $R$ is large enough.

\begin{proof}[Proof of Lemma \ref{lem:specialclosedpsi}] Let $\nu_{[0,T]}\in\Probspecial_{R}$ for some $R$ to be specified later. Lemma \ref{lem:constrainstructure} indicates that there exists a function $G_{\nu}^{(L-1)}:\R^{D_{L-1}}\times \R^{D_L}\to C([0,T],\R^{D_{L-1}})$ such that, if $\Thetabar\sim \psi(\nu_{[0,T]})$, then almost surely $\Thetabar^{(L-1)}(t)=G_{\nu}^{(L-1)}\parn{\Thetabar^{(L-1)}(0),\Thetabar^{(L)}(0)}(t)$ for all $t\in[0,T]$. Our goal is to show that:
\begin{equation}
\label{eq:goalspecialclosedpsi}
\mbox{{\bf Goal:} } G_{\nu}^{(L-1)}(\cdot)= G(\cdot)\qquad\mbox{ $\mu_0^{(L-1,L)}$-a.e., for some $R$-special function $G$,}
\end{equation}
which clearly suffices to prove the Lemma.

To do this, we investigate the ODE (\ref{eq:ODETheta-L}) satisfied by $\Thetabar^{(L-1)}$ under the assumption that $\nu_{[0,T]}$ is $R$-special. We will need the following result, which is a direct consequence of Picard iteration.

\begin{proposition}[Proof in \S  \ref{sub:proofPicardsimple}]
\label{prop:Picardsimple}
Fix constants $R\geq 2V_{\ref{prop:Picardsimple}} > 0$.
Given $(a^{(L-1)},a^{(L)})$ consider the initial value problem
\begin{equation}
\label{eq:H-ODE}
\left\{\begin{array}{rll}
x(0) =&a^{(L-1)}\\
\dfrac{dx}{dt}(t) =& H\parn{t,x(t),a^{(L)}} & \text{for }t\in [0,T],
\end{array}\right.
\end{equation}
where $H:[0,T]\times \R^{D_{L-1}}\times \R^{D_{L}}\to \R^{D_{L-1}}$ is a continuous map satisfying the following properties:
\begin{enumerate}
\item $H$ is uniformly bounded by $V_{\ref{prop:Picardsimple}}$;
\item for each fixed $t\in [0,T]$ and $a^{(L-1)}\in \R^{D_{L-1}}$, the map
\[a^{(L)}\in \R^{D_{L}}\mapsto H(t,a^{(L-1)},a^{(L)})\] is $V_{\ref{prop:Picardsimple}}\cdot e^{Rt}$-Lipschitz;
\item for each fixed $t\in [0,T]$ and $a^{(L)}\in \R^{D_{L}}$,
\[a^{(L-1)}\in \R^{D_{L}}\mapsto H(t,a^{(L-1)},a^{(L)})\] is $V_{\ref{prop:Picardsimple}}$-Lipschitz.
\end{enumerate}
Then the map
\[G:\R^{D_{L-1}}\times \R^{D_{L}}\to C([0,T],\R^{D_{L-1}})\]
that associates to each $(a^{(L-1)},a^{(L)})$, the solution of \ref{eq:H-ODE} is an $R$-Special function.
\end{proposition}

Given this Proposition, Lemma \ref{lem:specialclosedpsi} follows from the following result.

\begin{claim}
\label{claim:Hexists}
Let $\nu_{[0,T]}\in\Probspecial_R$. Then there exists a constant $V_{\ref{claim:Hexists}}=V_{\ref{claim:Hexists}}(C,L)$ and a continuous function
\[
H:[0,T]\times \R^{D_{L-1}}\times \R^{D_{L}}\to \R^{D_{L-1}}\mbox{ (possibly depending on $\nu_{[0,T]}$)}
\]
satisfying properties (1) - (3) in the proposition above with $V_{\ref{prop:Picardsimple}} = V_{\ref{claim:Hexists}}$, and so that, if $\Thetabar\sim \psi(\nu_{[0,T]})$, then almost surely, for all $t\in[0,T]$
\begin{equation}
\Thetabar^{(L-1)}(t) = \Thetabar^{(L-1)}(0) + \int_0^{t}\,H\parn{s,\Thetabar^{(L-1)}(s),\Thetabar^{(L)}(0)}\,ds.
\end{equation}
\end{claim}
Indeed, we may apply Proposition \ref{prop:Picardsimple} to the $H$ coming from this Claim to deduce that, if $R\geq R_{\ref{lem:specialclosedpsi}}(C,L):=2V_{\ref{claim:Hexists}}(C,L)$ and $\nu_{[0,T]}\in\Probspecial_R$, then for $\Thetabar\sim\psi(\nu_{[0,T]})$ and for all $t\in[0,T]$
\[
\Thetabar^{(L-1)}(t)=G_\nu^{(L-1)}\parn{\Thetabar^{(L-1)}(0),\Thetabar^{(L)}(0)}(t) =  G\parn{\Thetabar^{(L-1)}(0),\Thetabar^{(L)}(0)}(t) \mbox{ a.s.,}
\]
where $G$ is the $R$-special function from Proposition \ref{prop:Picardsimple}. In particular, this implies Goal (\ref{eq:goalspecialclosedpsi}).

We now prove the Claim. In what follows, we will use symbols $V_0,V_1,\dots$ to denote positive quantities that depend only on $C$ and $L$.

Consider $\nu_{[0,T]}\in \Probspecial_R$. From Lemma \ref{lem:constrainstructure} we know that there exists an $R$-special function $F_\nu$ such that the identity
\begin{equation}
\label{eq:zbarLredefined}
\zbar^{(L)}(x,\mu_t,a^{(L)}) =\int\limits_{\R^{D_{L-1}}}\fsigma{L-1}{\zbar^{(L-1)}(x,\mu),F_\nu\parn{\tilde{a}^{(L-1)},a^{(L)}}(t)}\,d\mu_0^{(L-1)}(\tilde{a}^{(L-1)})
\end{equation}
holds simultaneously for all $t\in [0,T]$, $P$-a.e $x\in\R^{d_X}$ for $\mu_0^{(L)}$-a.e. $a^{(L)}\in\R^{D_L}$.

In a slight abuse of notation, we will assume that $\zbar^{(L)}$ has been redefined so that (\ref{eq:zbarLredefined}) holds for {\em all} $a^{(L)}\in \R^{D_{L}}$ and $x\in \R^{d_X}$. Recalling that $\sigma^{(L-1)}$ is $C$-Lipschitz and $F_\nu(\cdot)(t)$ is $e^{Rt}$-Lipschitz (because $F_\nu$ is $R$-special), we easily obtain that:
\begin{equation}
\label{eq:zbarLLipschitz}
\forall t\in[0,T]\,:\,\zbar^{(L)}(x,\nu_t,a^{(L)}) \mbox{ is a $C\cdot e^{Rt}$-Lipschitz function of $a^{(L)}$}.
\end{equation}
Now consider the drift term for the $(L-1)^\text{th}$ particle. Recalling Assumption \ref{assump:activations}, Lemma \ref{lem:redefiningzL}, and the definitions in \ref{sub:McKeanVlasov} we see that
\[\Mbar^{(L)}(x,\nu_t,a^{(L)}) = \Mbar^{(L+1)}(x,\nu_t)\,D_z\fsigma{L}{\zbar^{(L)}(x,\nu_t,a^{(L)})}\]
is continuous in $(t,a^{(L)})$, $V_0\cdot e^{Rt}$-Lipschitz in $a^{(L)}$ for $t\geq 0$ fixed, and bounded by $V_0$. Also, for ${\bf a}=(a^{(\ell)})_{\ell=0}^{(L)}\in \R^D$
\[
\gbar^{(L-1)}(x,\nu_t,{\bf a}) = \Mbar^{(L)}(x,\nu_t,a^{(L)})\,D_\theta\fsigma{L-1}{\zbar^{(L)}(x,\nu_t,a^{(L)}),a^{(L-1)}}
\]
is continuous in $(t,a^{(L-1)},a^{(L)})$, $C$-Lipschitz in $a^{(L-1)}$ and $V_1\cdot e^{Rt}$-Lipschitz in $a^{(L)}$ for $t\geq 0$ fixed. $\gbar^{(L-1)}$ is also uniformly bounded by $V_1$. Once we notice additionally that $\gbar^{(L-1)}(x,\nu_t,{\bf a})$ depends on ${\bf a}$ only through $a=(a^{(L-1)},a^{(L)})$, we may define
\[H(t,a) = H(t,a^{(L-1)},a^{(L)}):= -\alpha(t)\,\Ex{(Y-\ybar(X,\nu_t))^\dag\,\gbar^{(L-1)}(x,\nu_t,{\bf a})},\]
for $(t,a)\in [0,T]\times \R^{D_{\ell-1}}\times \R^{D_\ell}$ (with ${\bf a} = (a^{(\ell)})_{\ell=0}^L$ arbitrarily defined for $\ell\neq L,L-1$). Clearly,  $H$ has the desired properties, as per Assumption \ref{assump:activations} and the previous bounds.
\end{proof}

\subsection{Proof that all solutions are special.}
\label{sub:proofspecialcontainssoln}
We now argue that all fixed points $\soln$ of $\psi$ are $R$-special.
\begin{proof}[Proof of Lemma \ref{lem:specialcontainssoln}]
The proof of this result is similar to that of Lemma \ref{lem:specialclosedpsi} in Section \ref{sub:proofpsispecial}. The key difference is that we will analyse a slightly more complicated Cauchy problem than the one in Proposition \ref{prop:Picardsimple}.

Let $\soln_{[0,T]}\in \Probspace(C([0,T],\R^D))$ be a solution to the McKean-Vlasov Problem in Definition \ref{def:McKeanVlasov}. Lemma \ref{lem:trivialsoln} implies that $\soln_{[0,T]}\in\Probspace_R$ if $R\geq R_{\ref{lem:trivialsoln}}$. Corollary \ref{cor:mappingwelldefined} applies to $\soln_{[0,T]}$, and we have $\psi(\soln_{[0,T]})=\soln_{[0,T]}$ because $\soln_{[0,T]}$ is a solution. Lemma \ref{lem:constrainstructure} implies there exist a measurable mapping $\Gsoln:=G_{\soln}^{(L-1)}$ such that if $\Thetabar\sim \soln_{[0,T]}$, then
\begin{equation}
\label{eq:Gisusefulforsoln}
\Thetabar^{(L-1)}(t) = \Gsoln\parn{\Thetabar^{(L-1)}(0),\Thetabar^{(L)}(0)}(t) \quad\forall t\in[0,T] \mbox{ a.s..}
\end{equation}

We also know from Lemma \ref{lem:constrainstructure} that $\Gsoln$ belongs to the following set of functions if $R\geq  R_{\ref{lem:trivialsoln}}$
\begin{equation}
\label{eq:defsU}
\sU:=\left\{g:\R^{D_{L-1}}\times \R^{D_L} \to C([0,T],\R)\,:\,
\begin{array}{l}
\mbox{$g$ is measurable;}\\
t\mapsto g(a)(t) \mbox{ is $R$-Lipschitz}\\
\forall a\in \R^{D_{L-1}}\times \R^{D_L};\\
g(a)(0)=a^{(L-1)}
\end{array}\right\}.
\end{equation}

Our goal will be to show:
\begin{equation}
\label{eq:goalspecialcontainssoln}
\mbox{{\bf Goal:} $\Gsoln$ is a.s. equal to an $R$-special function $G_o$, if $R\geq R_1(C,L)$.}
\end{equation}
Indeed, this will imply that $\soln_{[0,T]}$ is $R$-special, as per Definition \ref{def:Rspecialmeasure}.

To do this, we will need to "do analysis" on the set $\sU$. We need to establish some additional notation. For $g_1,g_2\in\sU$ and $t\in [0,T]$, we define.
\begin{equation}
\label{eq:distU} 
\dist(g_1,g_2;t):=\sup\limits_{a\in \R^{D_{L-1}}\times \R^{D_L}}\abs{g_1(a)(t) - g_2(a)(t)}
\end{equation}
and note that $\dist(g_1,g_2;t)\leq 2Rt$ because each $g_i$ is $R$-Lipschitz and $g_1(a)(0)=g_2(a)(0)=a^{(L-1)}$ for each $a\in \R^{D_{L-1}}\times \R^{D_L}$.

The next Proposition gives properties of a nonstandard Cauchy-type problem for functions in $\sU$. It should be contrasted with Proposition \ref{prop:Picardsimple} in the proof of Lemma \ref{lem:specialclosedpsi}.

\begin{proposition}[Proof in Appendix \ref{sub:proofPicardcomplicated}]
\label{prop:Picardcomplicated} 
Fix $R\geq 2V_{\ref{prop:Picardcomplicated}}>0$ constants and let
\[
\Xi:[0,T]\times \R^{D_{L-1}}\times \R^{D_L}\times \sU\to \R^{D_{L-1}}
\]
be a measurable function that satisfies the following properties:
\begin{enumerate}
\item $\Xi$ is uniformly bounded by $V_{\ref{prop:Picardcomplicated}}$.
\item If two functions $g_1,g_2\in\sU$ are such that $g_1(a^{(L-1)},a^{(L)})=g_2(a^{(L-1)},a^{(L)})$ $\mu_0^{(L-1)}\times \mu_0^{(L)}$-almost surely, then:
\[
\Xi\parn{t,a^{(L-1)},a^{(L)},g_1}=\Xi\parn{t,a^{(L-1)},a^{(L)},g_2}
\]
$\mu_0^{(L-1)}\times \mu_0^{(L)}$-almost-surely;
\item for any fixed $g\in \sU$, the map
\[
(t,a^{(L-1)},a^{(L)})\mapsto \Xi\parn{t,a^{(L-1)},a^{(L)},g}
\]
is measurable. Also, for $a^{(L)}$ fixed, then $(t,a^{(L-1)})\mapsto \Xi\parn{t,a^{(L-1)},a^{(L)},g}$ is continuous;
\item for any fixed $t\in [0,T]$, $g\in \sU$ and $a,b\in \R^{D_{L-1}}\times \R^{D_L}$ with $|a-b|\leq \epsilon$
\begin{align*}
&\abs{\Xi\parn{t,a^{(L-1)},a^{(L)},g} - \Xi\parn{t,b^{(L-1)},b^{(L)},g}}\\ 
&\qquad\qquad\qquad\leq V_{\ref{prop:Picardcomplicated}}\,\sup\limits_{\stackrel{a',b'\in\R^{D_{L-1}}\times \R^{D_L}}{|a'-b'|\leq \epsilon}}\abs{g(a')(t)-g(b')(t)};
\end{align*}
\item for any fixed $t\in [0,T]$, $g_1,g_2\in \sU$ and $a\in \R^{D_{L-1}}\times \R^{D_L}$:
\[
\abs{\Xi(t,a^{(L-1)},a^{(L)},g_1) - \Xi(t,a^{(L-1)},a^{(L)},g_2)}\leq V_{\ref{prop:Picardcomplicated}}\cdot\dist(g_1,g_2;t).
\]
\end{enumerate}
Then there exists a mapping 
\[\sT:\sU\to \sU\]
such that, for $g\in \sU$, $a = (a^{(L-1)},a^{(L)})\in \R^{D_{L-1}}\times \R^{D_L}$, and $t\in[0,T]$,
\begin{equation}
\label{eq:defTsoln}
\sT(g)(a)(t) = a^{(L-1)} + \int_0^t\,\Xi\parn{s,g(a)(t),a^{(L)},g}\,ds.
\end{equation}
Moreover, $\sT$ has an unique fixed point $G_o$ which is an $R$-special function. If $\tilde{G}\in\sU$ satisfies $\sT(\tilde{G})=\tilde{G}$ $\mu_0^{(L-1)}\times \mu_0^{(L)}$-a.e., then $\tilde{G}=G_o$ $\mu_0^{(L-1)}\times \mu_0^{(L)}$-a.e..
\end{proposition}

We now make the following Claim.

\begin{claim}[Proven below]
\label{claim:Xiexists}
There exist a constant $V_{\ref{claim:Xiexists}} = V_{\ref{claim:Xiexists}}(C,L)>0$ and a function:
\[
\Xi:[0,T]\times \R^{D_{L-1}}\times \R^{D_L}\times \sU\to \R^{D_{L-1}}
\]
satisfying assumptions (1) - (5) in Proposition \ref{prop:Picardcomplicated}, and so that, for $\mu_0$-a.e initial condition and for all $t\in [0,T]$:
\[
\Thetabar^{(L-1)} (t) = \Thetabar^{(L-1)}(0) + \int_0^t\,\Xi\parn{s,\Thetabar^{(L-1)}(t),\Thetabar^{(L)}(0),\Gsoln}\,ds.
\]
\end{claim}

Once we prove this claim, the theorem is essentially proved. We just need to apply Proposition \ref{prop:Picardcomplicated} with this function $\Xi$ together with Lemma \ref{lem:constrainstructure}. The upshot is that, if $R\geq R_{\ref{lem:specialcontainssoln}}:=2V_{\ref{claim:Xiexists}}$, then the following identity holds almost surely, simultaneously for all $t\in [0,T]$:
\begin{align*}
&\Gsoln(\Thetabar^{(L-1)}(0),\Thetabar^{(L)}(0))(t) =\, \Thetabar^{(L-1)}(0) \\  
&\qquad\qquad\qquad + \int_0^t\,\Xi\parn{s,\Gsoln\parn{\Thetabar^{(L-1)}(0),\Thetabar^{(L)}(0)}(s),\Thetabar^{(L)}(0),\Gsoln}\,ds.
\end{align*}
In other words, $\sT(\Gsoln)=\Gsoln$ a.s.. The Proposition implies that there exists a $R$-special function $\Gsoln=G_o$ a.e. , and this finishes the proof because it achieves Goal \ref{eq:goalspecialcontainssoln}.

We now prove the Claim. In what follows, we will use symbols $V_0,V_1,\dots$ to denote positive quantities that only depend on $C$ and $L$.

To start, we observe that, although we do not know if $\soln_{[0,T]}$ is $R$-special, we have (\ref{eq:Gisusefulforsoln}). As a consequence, the reasoning in the proof of Lemma \ref{lem:redefiningzL} applies, with $\soln_{[0,T]}$ replacing $\nu_{[0,T]}$ and $\Gsoln$ replacing $F_\nu$. The upshot is that, for $P\times \mu_0^{(L)}$-a.e. $(x,a^{(L)})\in \R^{d_X}\times \R^{D_L}$, we have the following identity for all $t\in [0,T]$:
\[
\zbar^{(L)}\parn{x,\soln_t,a^{(L)}} =  \int\limits_{\R^{D_{L-1}}}\fsigma{L}{\zbar^{(L-1)}\parn{x,\soln_t)},\Gsoln\parn{\tilde{a}^{(L-1)},a^{(L)}}(t)}\,d(\soln_0)^{(L-1)}(\tilde{a}^{(L-1)}).
\]

Let us rewrite this as follows. Define $\Xi_0:[0,T]\times \R^{d_X}\times \R^{D_{L}}\times \sU$
via the following recipe: for $(t,x,a^{(L)},g)\in [0,T]\times \R^{d_X}\times \R^{D_{L}}\times \sU$,
\begin{equation}
\label{eq:defXi0}
\Xi_0\parn{t,x,a^{(L)},g}:=\int\limits_{\R^{D_{L-1}}}\fsigma{L}{\zbar^{(L-1)}\parn{x,\soln_t},g\parn{\tilde{a}^{(L-1)},a^{(L)}}(t)}\,d(\soln_0)^{(L-1)}(\tilde{a}^{(L-1)}).
\end{equation}

Then $\zbar^{(L)}\parn{x,\soln_t,a^{(L)}}=\Xi_0\parn{t,x,a^{(L)},\Gsoln}$ for every $(t,x)$ and $\mu_0^{(L-1)}$-almost-every $a^{(L)}$. Moreover, the map $\Xi_0$ satisfies analogues of the properties (1) - (5) in Proposition \ref{prop:Picardcomplicated}.

\begin{enumerate}
\item[(1')] $\Xi_0$ is bounded by $V_0$ because $\sigma^{(L-1)}$ is bounded by $C$
\item[(2')] $\Xi_0\parn{t,x,a^{(L)},g_1}=\Xi_0\parn{t,x,a^{(L)},g_2}$ for $\mu_0^{(L)}$-almost everywhere  $a^{(L)}$ if $g_1=g_2$ $\quad\mu_0^{(L-1,L)}$-almost everywhere;
\item[(3')] for all $g\in\sU$, the map $(t,x,a^{(L)})\mapsto \Xi_0\parn{t,x,a^{(L)},g}$ is measurable for being a composition of measurable functions, and it is continuous in $t$ for similar reasons;
\item[(4')] Since $\supnorm{D\sigma^{(L)}}\leq C$, then $\sigma^{(L)}$ is $C$-Lipschitz. So
\begin{align*}
& \abs{\Xi_0\parn{t,x,a^{(L)},g} - \Xi_0\parn{t,x,b^{(L)},g}}\\  
&\quad\quad\leq C\,\int\limits_{\R^{D_{L-1}}}\abs{g\parn{\tilde{a}^{(L-1)},a^{(L)}}(t) - g\parn{\tilde{a}^{(L-1)},b^{(L)}}(t)}\,d(\soln)_0^{(L-1)}(\tilde{a}^{(L-1)})\\ 
&\quad\quad\leq  C	\,\sup\limits_{\stackrel{a',b'\in\R^{D_{L-1}}\times \R^{D_L}}{\abs{a'-b'}\leq \epsilon}} \abs{g(a')(t)-g(b')(t)};
\end{align*}
\item[(5')] finally, for $g_1,g_2\in\sU$ and $(t,x,a^{(L)})$ as above, using again the Lipschitz property:
\[
\abs{\Xi_0\parn{t,a^{(L-1)},a^{(L)},g_1} - \Xi_0\parn{t,a^{(L-1)},a^{(L)},g_2}}\leq V_0\cdot\dist(g_1,g_2;t).
\]
\end{enumerate}

Going back to the Definition \ref{eq:defMbarL-Lminus1} for $\Mbar^{(L)}(x,\soln_t,a^{(L)})$, we see at once that we can define
\[
\Mbar^{(L)}(x,\soln_t,a^{(L)}) =: \Xi_1\parn{t,x,a^{(L)},\Gsoln}\mbox{ a.e.} 
\]
where for $(t,x,a^{(L)},g)\in [0,T]\times \R^{d_X}\times \R^{D_L}\times \sU$:
\[
\Xi_1\parn{t,x,a^{(L)},g}:= \Mbar^{(L+1)}(x,\soln_t)\cdot D_z\fsigma{L}{\Xi_0\parn{t,x,a^{(L)},g},a^{(L)}}.
\]
$\Xi_1$ satisfies similar properties to (1')-(5') with a constant $V_1$ replacing $V_0$, because $\Mbar^{(L+1)}$ is bounded by $C$ and $D_z\sigma^{(L)}$ is $C$-Lipschitz.

The same is true (with a different constant $V_2$ replacing $V_1$) for the map
\[
\Xi_2\parn{t,x,a,g}:=\Xi_1\parn{t,x,a^{(L)},g}\cdot D_\theta\fsigma{L-1}{\zbar^{(L-1)}(x,\soln_t),a^{(L-1)}},
\]
defined for $(t,x,a,g)\in [0,T]\times \R^{d_X}\times(\R^{D_{L-1}}\times \R^{D_L})\times \sU$; the only change from the previous case is that we replace $a^{(L)}$ with $a$ where needed. Notice also that $\Xi_2$ satisfies:
\[
\Xi_2\parn{t,x,a,\Gsoln} = \Gamma({\bf a},x,\soln_t)\,
\]
for $\mu_0^{(L)}$-almost every choice of $a^{(L)}$ and all choices of ${\bf a}=(a^{(\ell)})_{\ell=0}^{L}$ extending $a^{(L)}$.

The conclusion is that the function
\[
\Xi(t,a,g):=\alpha(t)\,\Exp{(X,Y)\sim P}{(Y-\ybar(x,\soln_t))\,\Xi_2(t,X,a,g)},
\]
defined for $(t,a,g)\in [0,T]\times(\R^{D_{L-1}}\times \R^{D_L})\times \sU$ satisfies properties (1) - (5) in the statement of Proposition \ref{prop:Picardcomplicated}. Moreover
\[
\gbar^{(L-1)}({\bf a},\soln_t)= \Xi(t,a,\Gsoln)
\]
for $\mu_0^{(L)}$-a.e. choice of $a^{(L)}$ and all choices of ${\bf a}=(a^{(\ell)})_{\ell=0}^{L}$ extending $a^{(L)}$. By Lemma \ref{lem:constrainstructure}, we know that, if $\Thetabar\sim \psi(\soln_{[0,T]})=\soln_{[0,T]}$, then a.s. for all $t\geq 0$,
\[
\Thetabar^{(L-1)}(t) = \Thetabar^{(L-1)}(0) + \int_0^t\,\Xi\parn{s,\Thetabar^{(L-1)}(s),\Thetabar^{(L)}(0),\Gsoln}\,ds.
\]
This proves Claim \ref{claim:Xiexists}.
\end{proof}

%% file: measureNN_aux/8-existence.tex
In this section, we prove Theorem \ref{thm:independencestructure}, which asserts existence and uniqueness of the solution of the McKean-Vlasov problem in Definition \ref{def:McKeanVlasov}. 

The groundwork for this was built in sections \ref{sec:apriori_simple} and \ref{sec:apriori_space}. A quick recap: we defined a map $\psi$ (cf. Corollary \ref{cor:mappingwelldefined}) whose fixed points are the solutions of the McKean-Vlasov problem. Then, we defined $\Probspecial_R$ the set of the so called $R$-special measures with the following properties:
\begin{itemize}
\item it is topologically closed (Theorem \ref{thm:Probspecialisclosed});
\item if $R$ is large enough, $\psi(\Probspecial_R)\subset \Probspecial_R$ (Lemma \ref{lem:specialclosedpsi});
\item if $R$ is large enough, $\Probspecial_R$ contains all solutions (Lemma \ref{lem:specialcontainssoln}).
\end{itemize}
The only missing step is to show that {\em for large enough $R>0$, $\Probspecial_R$ contains exactly one fixed point of $\psi$, which is the limit of $\left\{\psi^{j}(\nu_{[0,T]});\, j\in\N\right\}$ for any $\nu_{[0,T]}\in\Probspecial_R$}. For this we will employ a fixed-point argument which follows Rachev and R\"{u}schendorf \cite{Rachev2006}. The main idea is to show that for a large enough $m$, the operator $\psi^m$ is a strict contraction over the complete metric space $\Probspecial_R$ (with the Wasserstein $L^1$  metric).

The first step is to construct, for any fixed $\mu_{[0,T]},\nu_{[0,T]}\in\Probspecial_R$, a reasonable coupling between $\psi(\mu_{[0,T]})$ and $\psi(\nu_{[0,T]})$. By looking at Definition \ref{sub:towardsfixedpoint} for $\psi$, we see that this map comes with an associated process $\Thetabar$.

We will denote by $\Thetabar_\mu$ the process whose law is $\psi(\mu_{[0,T]})$ and $\Thetabar_\nu$ the one with law $\psi(\nu_{[0,T]})$. The way to couple these processes is to notice that, since $\mu_{[0,T]},\nu_{[0,T]}\in\Probspecial_R\subset \Probspace_R$, then according to Definition \ref{def:prelimset} we have $\mu_0 = \nu_0$. This means we can set $\Thetabar_\mu(0) = \Thetabar_\nu(0)$. Since the randomness inherent to these processes comes from their initialization, the coupling of the starting points couples the entire processes.

This coupling gives an upper bound for the Wasserstein distance between the measures, as we see below:
\begin{equation}
\label{eq:couplingupperbound}
W\parn{\psi(\mu_{[0,T]}), \psi(\nu_{[0,T]})}\leq\mu_0\parn{\sup_{t\leq T}\abs{\Thetabar_\mu(t) -\Thetabar_\nu(t)}}.
\end{equation}
So all we need to do is to find an upper bound for the RHS of this inequality.

To do so, we must work coordinate by coordinate. Define
\[
\Delta^{(\ell)}(t) = \sup_{u\leq t}\abs{\Thetabar^{(\ell)}_\mu(u) - \Thetabar^{(\ell)}_\nu(u)}
\quad\mbox{ and }\quad
\Delta(t) = \sup_{{u\leq t}}\abs{\Thetabar_\mu(u) - \Thetabar_\nu(u)}.
\]

Notice that there exist constants $v_{\ref{eq:equiv-norms}},V_{\ref{eq:equiv-norms}} > 0$ depending only on $L$ and \\
$\{D_\ell,\,0\leq \ell\leq L\}$ so that
\begin{equation}
\label{eq:equiv-norms}
v_{\ref{eq:equiv-norms}}\cdot \Delta(t) \leq \sum_{\ell=0}^L\Delta^{(\ell)}(t) \leq V_{\ref{eq:equiv-norms}} \cdot \Delta(t)
\end{equation}

Using this notation, we establish the following lemma:

\begin{lemma}\label{lem:deltaellupperbound}
There is a constant $C_{\ref{lem:deltaellupperbound}}$ such that the following inequalities hold:
\begin{enumerate}
\item $\Delta^{(0)}(t) = \Delta^{(L)}(t) = 0$
\item For $\ell \in [1,L-2]$:
\[\Delta^{(\ell)}(t) \leq C_{\ref{lem:deltaellupperbound}}\cdot\int_0^t \left\{W(\mu_{[0,s]}, \nu_{[0,s]}) + \Delta^{(\ell)}(s)\right\}\,ds\]
\item For $\ell = L-1$, we have:
\begin{align*}
&\Delta^{(L-1)}(t) \leq C_{\ref{lem:deltaellupperbound}}\cdot\int_0^t \left\{W(\mu_{[0,s]},\nu_{[0,s]}) + \Delta^{(L-1)}(s)\right. 
\\&\left.\qquad\qquad\qquad\qquad+\Ex{\abs{\zbar^{(L-1)}(X,\mu_s,\Thetabar_\mu(0)) - \zbar^{(L-1)}(X,\nu_s,\Thetabar_\nu(0))}} \right\}ds
\end{align*}
\end{enumerate}
\end{lemma}

\begin{proof}
The fact that $\Delta^{(0)}(t) = \Delta^{(L)}(t) = 0$ comes directly from the definition of the processes $\Thetabar_\mu,\Thetabar_\nu$. By the coupling we made, we see that $\Delta(0) = 0$, and both $\Delta ^{(0)},\Delta^{(L)}$ are constant over time.

For $\ell \in [1:L-2]$, we can check the following chain of inequalities:
\begin{align*}
\abs{\Thetabar^{(\ell)}_\mu(t) - \Thetabar^{(\ell)}_\nu(t)} &\leq \int_0^t \abs{\alpha(s)}\Big|\Exp{(X,Y)\sim P}{(Y - \ybar(X,\mu_s))^\dag\,\gbar^{(\ell)}(\Thetabar^{(\ell)}_\mu(s),X,\mu_s)} \\
                &\qquad\,\quad\quad\quad -\Exp{(X,Y)\sim P}{(Y - \ybar(X,\nu_s))^\dag\,\gbar^{(\ell)}(\Thetabar^{(\ell)}_\nu(s),X,\nu_s)}\Big| ds\\
                &\leq C_{\ref{lem:unif-lip-gen-l}} \int_0^t \left\{W(\mu_s,\nu_s) + \abs{\Thetabar^{(\ell)}_\mu(s) - \Thetabar^{(\ell)}_\nu(s)}\right\} ds\\
                &\leq C_{\ref{lem:unif-lip-gen-l}} \int_0^t \left\{W(\mu_{[0,s]},\nu_{[0,s]}) + \Delta^{(\ell)}(s)\right\}ds.
\end{align*}
The first inequality comes from the fact that $\mu,a^{(\ell)}\mapsto (Y - \ybar(X,\mu))^\dag\,\gbar^{(\ell)}(a^{(\ell)},X,\mu)$ is a uniformly lipschitz over $(X,Y)$, as we are told by Lemma \ref{lem:unif-lip-gen-l}, while the second comes from the fact that $W(\mu_s,\nu_s)\leq W(\mu_{[0,s]},\nu_{[0,s]})$, as we are told by Lemma \ref{lem:timemeasureupperb} and that a function is always smaller than its supremum.

Notice that, for $u<t$, we have that:
\begin{align*}
\abs{\Thetabar^{(\ell)}_\mu(u) - \Thetabar^{(\ell)}_\nu(u)} &\leq C_{\ref{lem:unif-lip-gen-l}} \int_0^u W(\mu_{[0,s]},\nu_{[0,s]}) + \Delta^{(\ell)}(s)ds
\\&\leq C_{\ref{lem:unif-lip-gen-l}} \int_0^t \left\{W(\mu_{[0,s]},\nu_{[0,s]}) + \Delta^{(\ell)}(s)\right\} ds
\end{align*}
This means that we can make:
\begin{equation}
\label{eq:supargument}
\sup_{u\leq t}\abs{\Thetabar^{(\ell)}_\mu(u) - \Thetabar^{(\ell)}_\nu(u)} = \Delta^{(\ell)}(t) \leq C_{\ref{lem:unif-lip-gen-l}} \int_0^t \left\{W(\mu_{[0,s]},\nu_{[0,s]}) + \Delta^{(\ell)}(s)\right\} ds.
\end{equation}
For the case $\ell = L-1$, we will make use of the Lemma \ref{lem:unif-lip-L-1} to get:
\begin{align*}
\abs{\Thetabar^{(L-1)}_\mu(t) - \Thetabar^{(L-1)}_\nu(t)}&\leq C_{\ref{lem:unif-lip-L-1}}\cdot\int_0^t \left\{W(\mu_s,\nu_s) + \abs{\Thetabar^{(L-1)}_\mu(s) - \Thetabar^{(L-1)}_\nu(s)} +\right.\\
&\left.\quad\quad+\Ex{\abs{\zbar^{(L)}(X,\mu_s,\Thetabar^{(L)}_\mu(s)) - \zbar^{(L)}(X,\nu_s,\Thetabar^{(L)}_\nu(s))}}\right\}ds\\
&\leq C_{\ref{lem:unif-lip-L-1}}\int_0^t\left\{W(\mu_{[0,s]},\nu_{[0,s]}) + \Delta^{(L-1)}(s) +\right. \\ &
\left.\quad\quad+\Ex{\abs{\zbar^{(L)}(X,\mu_s,\Thetabar^{(L)}_\mu(0)) - \zbar^{(L)}(X,\nu_s,\Thetabar^{(L)}_\nu(0))}}\right\} ds,
\end{align*}
justified again by Lemma \ref{lem:timemeasureupperb} and for the fact that a function is always smaller than its supremum. Repeating the same argument as in \ref{eq:supargument}, we can finally conclude that:
\begin{align*}
\Delta^{(L-1)}(t) &\leq C_{\ref{lem:unif-lip-L-1}}\int_0^t W(\mu_{[0,s]},\nu_{[0,s]}) + \Delta^{(L-1)}(s) + \\ &+\Ex{\abs{\zbar^{(L-1)}(X,\mu_s,\Thetabar_\mu(0)) - \zbar^{(L-1)}(X,\nu_s,\Thetabar_\nu(0))}} ds,
\end{align*}
Choose $C_{\ref{lem:deltaellupperbound}} = \max\{C_{\ref{lem:unif-lip-L-1}}, C_{\ref{lem:unif-lip-gen-l}}\}$ to finish Lemma.
\end{proof}

Now we can use these estimates to prove the following upper bound for $\psi$:

\begin{lemma}\label{lem:deltaupperbound}
There is a constant $C_{\ref{lem:deltaupperbound}}$ depending only on $L,R$ and $C$ such that, for all $t\in [0,T]$,
\[W\parn{\psi(\mu_{[0,T]}),\psi(\nu_{[0,T]})}\leq C_{\ref{lem:deltaupperbound}}\cdot e^{C_{\ref{lem:deltaupperbound}} T}\int_0^T W(\mu_{[0,s]},\nu_{[0,s]})ds.	\]
\end{lemma}

\begin{proof}
We start by adding all the inequalities of Lemma \ref{lem:deltaellupperbound} and using the equivalence of norms \ref{eq:equiv-norms} to get
\begin{align}
\label{eq:prepareGronwalldelta}
\nonumber
\Delta(t) &\leq V_{\ref{eq:equiv-norms}}\cdot \sum_{\ell=1}^L \Delta^{(\ell)}(t)\\
\nonumber
 &\leq V_{\ref{eq:equiv-norms}}\cdot \int_0^t C_{\ref{lem:deltaellupperbound}}\cdot\left(\sum_{\ell=1}^{L-1} \parn{W(\mu_{[0,s]},\nu_{[0,s]}) + \Delta^{(\ell)}(s)} +\right.\\
&\quad\quad\quad \left.\vphantom{\sum_{\ell=1}^L \Delta^{(\ell)}(t)}+\Ex{\abs{\zbar^{(L-1)}(X,\mu_s,\Thetabar_\mu(0)) - \zbar^{(L-1)}(X,\nu_s,\Thetabar_\nu(0))}} \Big) \right) ds \\
\nonumber
&\leq V_{\ref{eq:prepareGronwalldelta}}\int_0^t \left\{W(\mu_{[0,s]},\nu_{[0,s]}) + \Delta(s)\right.\\ 
\nonumber
&\left.\quad\quad\quad+\Ex{\abs{\zbar^{(L-1)}(X,\mu_s,\Thetabar_\mu(0)) - \zbar^{(L-1)}(X,\nu_s,\Thetabar_\nu(0))}} \right\}ds,
\end{align}
where  $V_{\ref{eq:prepareGronwalldelta}} =  V_{\ref{eq:equiv-norms}}\cdot C_{\ref{lem:deltaellupperbound}}\cdot\max\{1,L-2,v_{\ref{eq:equiv-norms}}\}$.

Now, we integrate these estimates over $\mu_0$ to find:
\begin{align*}
\mu_0(\Delta(t))& \leq V_{\ref{eq:prepareGronwalldelta}}\int_0^t \left\{W(\mu_{[0,s]},\nu_{[0,s]}) + \mu_0(\Delta(s)) +\right.\\
&\left.\quad\quad+\Ex{\mu_0\parn{\abs{\zbar^{(L-1)}(X,\mu_s,\Thetabar_\mu(0)) - \zbar^{(L-1)}(X,\nu_s,\Thetabar_\nu(0))}}}\right\} ds
\end{align*}

Now we use Lemma \ref{lem:estimatezbarL} to estimate the integral of the difference of $\zbar^{(L)}$ with respect to $\mu_0$ to get:
\begin{align*}
\mu_0(\Delta(t)) &\leq V_{\ref{eq:prepareGronwalldelta}}\int_0^t \left\{W(\mu_{[0,s]},\nu_{[0,s]}) + \mu_0(\Delta(s)) + C(e^{RT} +1) W(\mu_{[0,s]},\nu_{[0,s]})\right\} ds\\
&\leq C_{\ref{lem:deltaupperbound}} \int_0^t \left\{W(\mu_{[0,s]},\nu_{[0,s]}) + \mu_0(\Delta(s))\right\}ds,
\end{align*}
where $C_{\ref{lem:deltaupperbound}} = V_{\ref{eq:prepareGronwalldelta}}\cdot \max\{1,C(e^{RT} +1)\}$. 

Finally, we apply Gronwall's Lemma and the Observation \ref{eq:couplingupperbound} to get:
\[
W\parn{\psi(\mu_{[0,t]}),\psi(\nu_{[0,t]})}\leq\mu_0(\Delta(t))\leq C_{\ref{lem:deltaupperbound}}\cdot e^{C_{\ref{lem:deltaupperbound}}t}\int_0^t W(\mu_{[0,s]},\nu_{[0,s]})ds.
\]
\end{proof}

The above inequality may be combined with a Gronwall-style iteration to obtain: 
\[
W\parn{\psi^m(\mu_{[0,T]}),\psi^m(\nu_{[0,T]})}\leq \frac{\parn{C_{\ref{lem:deltaupperbound}} T\cdot e^{C_{\ref{lem:deltaupperbound}} T}}^m}{(m-1)!}\cdot W(\mu_{[0,T]},\nu_{[0,T]})
\]
for all $m\in\N$. In particular, for large enough $m$, the operator $\psi^m$ is a contraction over $\Probspecial_R$. This implies the existence and uniqueness of a measure $\soln_{[0,T]}\in\Probspecial_R$ such that $\psi(\soln_{[0,T]}) = \soln_{[0,T]}$ and $\psi^m(\mu_{[0,T]})\to \soln_{[0,T]}$ when $m\to +\infty$.

To complete the proof of Theorem \ref{thm:independencestructure}, we need  to argue that $\soln_{[0,T]}$ has the independence structure described there. As it turns out, this is a mere consequence of Lemma \ref{lem:constrainstructure}. Since $\soln_{[0,T]} = \psi(\soln_{[0,T]})$, we have that $\Thetabar_{[0,T]}\sim \soln_{[0,T]}$ must satisfy, for all $t\in[0,T]$
\begin{alignat*}{2}
\Thetabar^{(0)}(t) =& \Thetabar^{(0)}(0) \quad\quad\mbox{ (constant)}\\
\Thetabar^{(1)}(t) =& G^{(1)}_{\soln}\parn{\Thetabar^{(0)}(0),\Thetabar^{(1)}(0)}(t) &&=: F^{(1)}\parn{\Thetabar^{(0)}(0),\Thetabar^{(1)}(0)}(t)\\
\Thetabar^{(\ell)}(t) =& G^{(\ell)}_{\soln}\parn{\Thetabar^{(\ell)}(0)}(t) &&=: F^{(\ell)}\parn{\Thetabar^{(\ell)}(0)}(t)\quad\quad \ell\in[2,L-2]\\
\Thetabar^{(L-1)}(t) =& G^{(L-1)}_{\soln}\parn{\Thetabar^{(L-1)}(0),\Thetabar^{(L)}(0)}(t) &&=: F^{(L-1)}\parn{\Thetabar^{(L-1)}(0),\Thetabar^{(L)}(0)}(t)\\
\Thetabar^{(L)}(t) =& \Thetabar^{(L)}(0) \quad\quad\mbox{ (constant)}.
\end{alignat*}
This implies that $\soln_{[0,T]}$ admits the decomposition:
\[\soln_{[0,T]} = {\soln}^{(0,1)}_{[0,T]} \times\left(\prod_{\ell=2}^{L-2}{\soln}^{(\ell)}_{[0,T]}\right)\times{\soln}^{(L-1,L)}_{[0,T]},\]
thereby completing the proof.

%% file: measureNN_aux/9-SGD_to_CTGD.tex
In this section we show that the Stochastic Gradient Descent (SGD) process used for training our neural network is close to a continuous time version. This is the first step in our connection between SGD and the McKean-Vlasov process.

We start by defining a norm in $\R^{p_n}$ the space of parameters. Given $\btheta_N \in \R^{p_N}$, we define:
\begin{equation}
\label{eq:normParam}
\Lnorm{\btheta_N} = \sup_{\ell \in [0:L]}\left\{\frac{1}{N_\ell\cdot N_{\ell+1}}\sum_{i_{\ell+1} = 1}^{N_{\ell+1}}\sum_{i_\ell = 1}^{N_\ell}\abs{\theta^{(\ell)}_{i_\ell,i_{\ell+1}}}\right\}.
\end{equation}
This norm computes maximum layer average of the parameters.

Now, recall the definition for the SGD. Given an i.i.d sample $(X_k,Y_k)_{k\in\N}\iid P$, and some initial $\btheta_N(0)$, the SGD process is defined iteratively, through
\begin{equation}\label{def:recapSGD}
  \theta_{i_\ell,i_{\ell+1}}^{(\ell)}(k) = \theta_{i_\ell,i_{\ell+1}}^{(\ell)}(k-1) - \eps\,\alpha(k\eps)\, \cdot\Gradh_{i_\ell,i_{\ell+1}}^{(\ell)}(X_k,Y_k,\btheta_N(k-1)).
\end{equation}

When $\eps$ is small, we will show that this process is close to the usual Continuous Time Gradient Descent (CTGD) obtained from the problem of minimizing the parametric Mean-Squared Loss
\[
	\L_N(\btheta_N)=\frac{1}{2}\Exp{(X,Y)\sim P}{\abs{Y-\yhat(X,\btheta_N)}^2},
\]
over the parameters in the inner layers, i.e. with superscripts $\ell\in[1:L-1]$. This process $\bthetatil_N \in C([0,T], \R^{p_N})$ is described by the following system of differential equations:\\
\begin{equation}\label{eq:deriv-thetatil}
  \frac{\partial \thetatil^{(\ell)}_{i_\ell,i_{\ell+1}}}{\partial t}\,(t) =
  \begin{cases}
  -\alpha(t)\,N^2\,\frac{\partial \L_N}{\partial \thetatil^{(\ell)}_{i_\ell,i_{\ell+1}}}\,(\bthetatil_N(t)),& \substack{\ell\in[1:L-1],\\(i_\ell,i_{\ell+1})\in[1:N_\ell]\times [1:N_{\ell+1}]}\\
	0& \ell\in\{0,L\}
\end{cases}.
\end{equation}

By recalling Remark \ref{rem:average-grad}, we see that this process can be written in the integral form as
\begin{equation}
\label{def:ctgd-process}
  \thetatil^{(\ell)}_{i_\ell, i_{\ell+1}}(t) = \thetatil^{(\ell)}_{i_\ell, i_{\ell+1}}(0) - \int_0^t\,\alpha(s)\,\Exp{(X,Y)\sim P}{\Gradh^{(\ell)}_{i_\ell, i_{\ell+1}}(X,Y,\bthetatil_N(s))}\,ds.
\end{equation}

For the SGD, we can unroll the iteration and obtain
\begin{equation}
  \theta^{(\ell)}_{i_\ell, i_{\ell+1}}(k) = \theta^{(\ell)}_{i_\ell, i_{\ell+1}}(0) - \sum_{r=1}^k\,\eps\,\alpha(r\eps)\,\Gradh^{(\ell)}_{i_\ell, i_{\ell+1}}(X_r,Y_r,\btheta_N(r-1)).
\end{equation}

Finally, we observe that the two processes are coupled by setting the same initial configurations, i.e. taking $\bthetatil_N(0) = \btheta_N(0)$. Now, fixed these two processes we prove the following result.

\begin{proposition}[SGD is close to CTGD]
\label{prop:SGDtoCTGD}
  Under the Assumptions of Section \ref{sub:assumptions}, the SGD process $\btheta_N$ and the CTGD process $\bthetatil_N$ are close when $\eps\rightarrow 0$ with high probability.
  More specifically, define:
  \[
  \Delta\thetatil^{(\ell)}_{i_\ell,i_{\ell+1}}(k) = \theta^{(\ell)}_{i_\ell,i_{\ell+1}}(k) - \thetatil^{(\ell)}_{i_\ell,i_{\ell+1}}(k\eps)\qquad\text{for }\substack{\ell\in[0:L],\\(i_\ell,i_{\ell+1})\in[1:N_\ell]\times [1:N_{\ell+1}]}
  \]
  and
  \[
  \Delta\thetatil_N(k) = \parn{\Delta\thetatil^{(\ell)}_{i_\ell,i_{\ell+1}}(k),\quad \substack{\ell\in[0:L],\\(i_\ell,i_{\ell+1})\in[1:N_\ell]\times [1:N_{\ell+1}]}}.
  \]

  Then there exist a constant $C_{\ref{prop:SGDtoCTGD}} = C_{\ref{prop:SGDtoCTGD}}(C,L,T)$ (where $d=\max\{d_\ell\,:\, \ell\in [1:L-1]\}$ is the largest dimension of an inner layer), such that for any $k\in[0:\lceil T/\eps\rceil]$
  \begin{equation}\label{bound:SGD-CTGD}
    \Lnorm{\Delta\thetatil_N(k)} \leq C_{\ref{prop:SGDtoCTGD}} \parn{\eps +\sqrt{\eps\,d} +  \sqrt{\eps}\,u}
  \end{equation}
  holds with probability higher than $1 - e^{-u^2}$.
\end{proposition}

\begin{proof}
We want to show that the SGD works as a discretization of the continuous process. This is accomplished by comparing the processes inside intervals of size $\eps$. Fix $k\in [1:\lceil T/\eps\rceil]$, $\ell\in[0:L]$, and $(i_\ell, i_{\ell+1})\in [N_\ell]\times[N_{\ell+1}]$, and observe that
  \begin{align}
  \label{eq:sgd-ctgd-estimate}
    |\Delta\thetatil^{(\ell)}_{i_\ell,i_{\ell+1}}(k)| =& \left| \sum_{r=1}^{k}
    \,\int_{(r-1)\eps}^{r\eps}\,\left( \alpha(r\eps)
    \,\Gradh^{(\ell)}_{i_\ell, i_{\ell+1}}(X_r,Y_r,\btheta_N(r-1))
    \vphantom{\Exp{(X,Y)\sim P}{\Gradh^{(\ell)}_{i_\ell, i_{\ell+1}}(X,Y,\bthetatil_N(s))}}\right.\right.\\
    \nonumber&\qquad\qquad\qquad \left.\vphantom{\sum_{r=1}^{k}\,\int_{(r-1)\eps}^{r\eps}}\left.- \alpha(s)\, \Exp{(X,Y)\sim P}{\Gradh^{(\ell)}_{i_\ell, i_{\ell+1}}(X,Y,\bthetatil_N(s))}\right) ds\,\right|.
  \end{align}

  To make the estimation we first define an auxiliar process, consider for each choice of indices
  \begin{align*}
    \M^{(\ell)}_{i_\ell, i_{\ell+1}}(k) = \sum_{r=1}^k\, \alpha(r\eps)\,&\left( \Gradh^{(\ell)}_{i_\ell, i_{\ell+1}}(X_r,Y_r,\btheta_N(r-1))
    \right. \\ &\qquad\qquad \left.-
    \Exp{(X,Y)\sim P}{\Gradh^{(\ell)}_{i_\ell, i_{\ell+1}}(X,Y,\btheta_N(r-1))} \right).
  \end{align*}
Notice that, according to Remark \ref{rem:average-grad}, the above process is a martingale with respect to the filtration $\{\sF_k,\,k\in\N\}$, where
\[\sF_k = \sigma\left(\vphantom{\int}\btheta_N(0), (X_1,Y_1),\dots,(X_k,Y_k)\right)\]
When adding and subtracting the crossed term $\alpha(r\eps)\,\Exp{(X,Y)\sim P}{\Gradh^{(\ell)}_{i_\ell, i_{\ell+1}}(X,Y,\btheta_N(r-1)}$ inside each integral of (\ref{eq:sgd-ctgd-estimate}), the estimate can be broken into two separate pieces
  \begin{align}
    \nonumber\abs{\Delta\thetatil^{(\ell)}_{i_\ell,i_{\ell+1}}(k)} \leq \eps\,|\M^{(\ell)}_{i_\ell, i_{\ell+1}}(k)|
        + \sum_{r=1}^k\,\int_{(r-1)\eps}^{r\eps}\,\mathbb{E}_{(X,Y)\sim P}\left[\left|
        \alpha(r\eps) \Gradh^{(\ell)}_{i_\ell, i_{\ell+1}}(X,Y,\btheta_N(r-1))\right.\right.\\
        - \left.\left.\vphantom{\mathbb{E}_{(X,Y)\sim P}}\alpha(s)\,\Gradh^{(\ell)}_{i_\ell, i_{\ell+1}}(X,Y,\bthetatil_N(s))\right|\right] ds.
  \end{align}

  Next, we apply the $\Lnorm{\cdot}$ to obtain the following:
  \begin{align}
    \label{eq:sgd-ctgd-estimate-2}
    \nonumber\Lnorm{\Delta\bthetatil_N(k)} \leq \overset{(I)}{\overbrace{\eps\Lnorm{\M_N(r)}}} + \overset{(II)}{\overbrace{\sum_{r=1}^k\,\int_{(r-1)\eps}^{r\eps}\,\mathbb{E}_{(X,Y)\sim P}\left[\left\|
    \alpha(r\eps) \Gradh_N(X,Y,\btheta_N(r-1))\right.\right.}}\\
    - \left.\left.\vphantom{\mathbb{E}_{(X,Y)\sim P}}\alpha(s)\,\Gradh_N(X,Y,\bthetatil_N(s))\right\|_{(L)}\right] ds.
  \end{align}
where $(I)$ will be controlled using martingale techniques and $(II)$ will be controlled using Lipschitz arguments. We start with $(II)$.

By invoking Lemma \ref{lem:grad-lnorm}, we see that the function $\alpha\cdot \Gradh_N$ is Lipschitz with respect to the $\Lnorm{\cdot}$. This means that
\[
    (II)\leq \sum_{r=1}^k\,\int_{(r-1)\eps}^{r\eps}\, C_{\ref{lem:grad-lnorm}}\,\left( |r\eps - s| +
    \|\btheta_N(r-1) - \bthetatil_N(s)\|_{(L)} \right)\, ds,
\]

  Next, we use the fact that the CTGD process $t\,\longrightarrow\,\bthetatil_N(t)$ is Lipschitz in time. To see why this holds, notice that, fixed $s < t$ in $[0,T]$ we have
  \begin{align}
  \label{eq:btheta-lip}
    \nonumber\abs{\bthetatil^{(\ell)}_{i_\ell,i_{\ell+1}}(t) - \bthetatil^{(\ell)}_{i_\ell,i_{\ell+1}}(s)} =& \abs{ \int_s^t \alpha(r)\,\Exp{(X,Y)\sim P}{\Gradh^{(\ell)}_{i_\ell, i_{\ell+1}}(X,Y,\bthetatil_N(r))}\,dr } \\
    \leq& C_{\ref{eq:btheta-lip}}\cdot|t-s|,
  \end{align}
  where $C_{\ref{eq:btheta-lip}} = C\cdot\supnorm{\Gradh^{(\ell)}_{i_\ell,i_{\ell+1}}}$. Using this last fact we obtain
  \begin{align}
  \label{eq:sgd-ctgd-estimate-3}
        \nonumber (II)\leq& \sum_{r=1}^k\,\int_{(r-1)\eps}^{r\eps	}\,C_{\ref{lem:grad-lnorm}}\,\parn{	\eps + \eps\cdot C_{\ref{eq:btheta-lip}} + \Lnorm{\Delta\bthetatil_N(r-1)}}\, ds\\
        \leq&  k\,\eps^2\,C_{\ref{lem:grad-lnorm}}\,(1+C_{\ref{eq:btheta-lip}}) + C_{\ref{lem:grad-lnorm}}\,\eps\,\sum_{r=1}^{k-1}\Lnorm{\Delta\bthetatil_N(r)}\\
        \nonumber\leq& T\,\eps\,C_{\ref{eq:sgd-ctgd-estimate-3}} +
        C_{\ref{lem:grad-lnorm}}\,\eps\,\sum_{r=0}^{k-1}\Lnorm{\Delta\bthetatil_N(r)},
  \end{align}
  where $C_{\ref{eq:sgd-ctgd-estimate-3}} = C_{\ref{lem:grad-lnorm}}\,(1+C_{\ref{eq:btheta-lip}})$.

Now we can plug this result back into \ref{eq:sgd-ctgd-estimate-2} to obtain
\[
\Lnorm{\Delta\bthetatil_N(k)} \leq \eps\, \sup_{r\leq k} \left\{\Lnorm{\M_N(r)} + \right\} +  T\,\eps\,C_{\ref{eq:sgd-ctgd-estimate-3}} + C_{\ref{lem:grad-lnorm}}\,\eps\,\sum_{r=0}^{k-1}\Lnorm{\Delta\bthetatil_N(r)},
\]
where $\M_N(r) = \parn{\M^{(\ell)}_{i_\ell, i_{\ell+1}}(r),\,\substack{\ell\in[0:L],\\(i_\ell,i_{\ell+1})\in[1:N_\ell]\times[1:N_{\ell+1}]}}$ is written with the same notation used for parameters. By applying a discrete Gronwall inequality, we obtain
  \begin{equation}
  \label{eq:after-gronwall}
  \Lnorm{\Delta\bthetatil_N(k)} \leq \eps\cdot e^{T\,C_{\ref{lem:grad-lnorm}}}\parn{\sup_{r\leq k}\left\{\Lnorm{\M_N(r)}\right\} + T\,C_{\ref{eq:sgd-ctgd-estimate-3}}}
  \end{equation}

Now we will employ a martingale argument to control $\sup_{r\leq k}\left\{\Lnorm{\M_N(r)}\right\}$. The first thing to notice is that, since every $\M^{(\ell)}_{i_\ell,i_{\ell+1}}$ is martingale, then $\Lnorm{\M_N}$ is a submartingale. For any given $\xi>0$, $e^{\xi\Lnorm{\M_N}}$ is also a submartingale. We want to upper bound
\[
(a) = \Pr{\sup_{r\leq k}\Lnorm{\M_N(r)} > u}
\]
Define a stopping time $\tau = \inf\left\{r\leq k,\, \Lnorm{\M_N(r)} > u\right\} \wedge k$. Then notice that
\[\left[\sup_{r\leq k}\Lnorm{\M_N(r)} > u\right] = \left[\Lnorm{\M_N(\tau) > u}\right].\]

By employing a Cramér-Chernoff argument, we get:
\[(a)\leq\inf_{\xi\in\R_+}\left\{ e^{-\xi\cdot u}\Ex{e^{\xi \Lnorm{\M_N(\tau)}}}\right\}\leq \inf_{\xi\in\R_+}\left\{e^{-\xi\cdot u}\Ex{e^{\xi \Lnorm{\M_N(k)}}}\right\},\]
where the second inequality is the Optional Stopping theorem for submartingales.

Computing the Expectation in the RHS above, we find
\begin{align}
  \label{eq:martlMGF}
    \nonumber
    \Ex{\exp\parn{\xi \Lnorm{\M_N(k)}}} = &
    \Ex{\exp\parn{\xi \sup_{\ell \in [0:L]}\left\{\frac{1}{N_\ell\cdot N_{\ell+1}}\sum_{i_{\ell+1} = 1}^{N_{\ell+1}}\sum_{i_\ell = 1}^{N_\ell}\abs{\M^{(\ell)}_{i_\ell,i_{\ell+1}}(k)}\right\}}}\\
	\leq & \sum_{\ell \in [0:L]}\left( \frac{1}{N_\ell\cdot N_{\ell+1}}\sum_{i_{\ell+1} = 1}^{N_{\ell+1}}\sum_{i_\ell = 1}^{N_\ell}\Ex{\exp\parn{\xi \abs{\M^{(\ell)}_{i_\ell,i_{\ell+1}}(r)}}}\right),
  \end{align}
where we applied Jensen's Inequality.
Now, we observe that each of the martingales $\M^{(\ell)}_{i_{\ell},i_{\ell+1}}$ has bounded increments. To see why that holds, just recall that according to Lemma \ref{lem:grad-lnorm}, $\alpha\cdot\Gradh^{(\ell)}_{i_\ell,i_{\ell+1}}$ is $C_{\ref{lem:grad-lnorm}}$-bounded, meaning:
  \begin{align*}
   |\M^{(\ell)}_{i_\ell, i_{\ell+1}}(r) - \M^{(\ell)}_{i_\ell, i_{\ell+1}}(r-1)| \leq&
        \left|\alpha(r\eps)\cdot\left( \Gradh^{(\ell)}_{i_\ell, i_{\ell+1}}(X_r,Y_r,\btheta_N(r-1))\qquad\qquad\qquad
        \vphantom{\Exp{(X,Y)\sim P}{\Gradh^{(\ell)}_{i_\ell}}}\right.\right.\\
        &\left.\left.-\Exp{(X,Y)\sim P}{\Gradh^{(\ell)}_{i_\ell, i_{\ell+1}}(X,Y,\bthetatil_N(r-1))} \right)\right|\\
        \leq& 2\cdot C_{\ref{lem:grad-lnorm}}.
  \end{align*}

This means that we can use Lemma \ref{lem:azuma} to get
\[
\Ex{\exp\parn{\xi \abs{\M^{(\ell)}_{i_\ell,i_{\ell+1}}(k)}}} \leq 5^d \cdot e^{k\cdot 4C_{\ref{lem:grad-lnorm}}^2 \xi^2}.
\]
Plug that back into \ref{eq:martlMGF} to get
\[
\Ex{\exp\parn{\xi \Lnorm{\M_N(k)}}} \leq L\cdot 5^d e^{k\,4C_{\ref{lem:grad-lnorm}}^2 \xi^2}.
\]
We complete the Cramér-Chernoff argument to find:
\[
\Pr{\sup_{r\leq k}\Lnorm{\M_N(r)} > u}\leq L\cdot 5^d\cdot \exp\parn{\frac{-u^2}{8C_{\ref{lem:grad-lnorm}}^2 \cdot k}}
\]
This statement can be rearranged to show that
\[
\sup_{r\leq k}\Lnorm{\M_N(r)} \leq C_{\ref{lem:grad-lnorm}} \sqrt{8\cdot \frac{T}{\epsilon}} \parn{\sqrt{\log(L) + d\log(5)} + u}
\]
holds with probability higher than $1-e^{-u^2}$.

By plugging this back into (\ref{eq:after-gronwall}), we see that there is a constant $C_{\ref{prop:SGDtoCTGD}} = C_{\ref{prop:SGDtoCTGD}}(C,T,L)$ such that
\[
\Lnorm{\Delta\bthetatil_N(k)} \leq C_{\ref{prop:SGDtoCTGD}} \parn{\eps + \sqrt{\eps\,d} + u\cdot\sqrt{\eps}}
\]
holds with probability higher than $1-e^{-u^2}$, which proves the lemma.
\end{proof}

%% file: measureNN_aux/10-ideal_particles.tex
In this section we introduce the {\em ideal particles} to which we will compare the weights  $\btheta_N(\cdot)$ in our network. It will be crucial to use the structure of the McKean-Vlasov measure, as described in Theorem \ref{thm:independencestructure}.

\begin{definition}[Ideal particles]
\label{def:idealparticles} Given the initial positions $\btheta_N(0)$ of the particles in the DNN (sampled as in Assumption \ref{assump:initialization}), we define
\[
\bthetabar_N(t):= \parn{\thetabar^{(\ell)}_{i_\ell,i_{\ell+1}}(t)\,:\, \substack{\ell\in [0:L],\\ (i_\ell,i_{\ell+1})\in [1:N_\ell]\times [1:N_{\ell+1}]}}\quad t\in [0,T]
\]
as follows. Recalling the maps $\left\{F^{(\ell)},\,\ell\in[1:L-1]\right\}$ in Theorem \ref{thm:independencestructure}, we set for $0\leq t\leq T$:
\begin{align*}
\thetabar^{(0)}_{1,i_1}(t) :=& \theta^{(0)}_{1,i_1}(0),&& i_1\in [1:N];\\
\thetabar^{(1)}_{i_1,i_2}(t) :=& F^{(1)}\parn{\theta_{1,i_1}^{(0)}(0),\theta_{i_1,i_2}^{(1)}(0)}(t),&& (i_1,i_2)\in [1:N]^2;\\
\thetabar^{(\ell)}_{i_\ell,i_{\ell+1}}(t) :=& F^{(\ell)}\parn{\theta_{i_{\ell},i_{\ell+1}}^{(\ell)}(0)}(t),&& (i_\ell,i_{\ell+1})\in [1:N]^2,\,\ell\in [2:L-2];\\
\thetabar^{(L-1)}_{i_{L-1},i_{L}}(t) :=& F^{(L-1)}\parn{\theta^{(L-1)}_{i_{L-1},i_{L}}(0),\theta^{(L)}_{i_{L},1}(0)}(t),&& (i_{L-1},i_{L})\in [1:N]^2;\\
\thetabar^{(L)}_{i_L,1}(t) :=& \theta_{i_Ll,1}^{(L)}(0),&& i_L\in [1:N].
\end{align*}\end{definition}

Note the following immediate proposition related to the McKean-Vlasov process (see our discussion in Section \ref{sec:McKeanVlasov}). Recall that $\soln_{[0,T]}$ is the law of the unique solution to our McKean-Vlasov problem.

\begin{proposition}[McKean-Vlasov measure along paths] For all choices of \[(i_1,i_2,\dots,i_L)\in [1:N]^L,\]
\[\parn{\thetabar^{(0)}_{1,i_{1}}(\cdot),\thetabar^{(1)}_{i_1,i_{2}}(\cdot),\dots,\thetabar^{(L-1)}_{i_{L-1},i_{L}}(\cdot),\thetabar^{(L)}_{i_L,1}(\cdot)}\sim \soln_{[0,T]}.\]
\end{proposition}
\begin{proof}[Sketch of proof] This comes from two observations. The first one is the structure of $\soln_{[0,T]}$ described in Theorem \ref{thm:independencestructure}. The second is that  the initial law of the McKean Vlasov process $\soln$ is equal to the starting measure for the $\btheta_N$ along a path.\end{proof}

Our main goal in this section will be to show that {\em the time evolution of the ideal weights $\bthetabar_N(t)$ nearly follows a continuous-time gradient flow with respect to $L_N(\bthetabar_N(t))$.} More precisely, we will prove the following result.

\begin{lemma}[Ideal weight evolution tracks gradient flow]
\label{lem:nearlygrad}
For each $1\leq \ell\leq L-1$, each time $t\in [0,T]$ and each pair $(i_\ell,i_{\ell+1})\in [1:N]^2$, the quantity:
\[\Delta_{\partial}\thetabar^{(\ell)}_{i_\ell,i_{\ell+1}}(t) := \frac{d}{dt}\thetabar^{(\ell)}_{i_\ell,i_{\ell+1}}(t) + N^2\alpha(t)\,\nabla L_N(\bthetabar_N(t))\in \R^{D_{\ell}}\]
satisfies, for all $\xi>0$ and $0\leq t\leq T$,
\[\log\Ex{\exp\left[\xi|\Delta_\partial\thetabar^{(\ell)}_{i_\ell,i_{\ell+1}}(t)|\right|}\leq Ld\log 5+ \frac{C_{\ref{lem:nearlygrad}}\,\xi^2}{2N},\]
where $d=\max\{d_i\,:\, i\in [1:L]\}$ is the largest dimension of an inner layer and $C_{\ref{lem:nearlygrad}}=C_{\ref{lem:nearlygrad}}(C,L)$.
This implies that the inequality
\[\abs{\delta\thetabar^{(\ell)}_{i_\ell,i_{\ell+1}}} \leq \sqrt{\frac{C_{\ref{lem:nearlygrad}}}{N}}\parn{u + \sqrt{Ld\log(5)}}\]
holds with probability greater than $1-e^{-u^2}$
\end{lemma}

\subsection{Comparing ideal and real $z$'s} To prove Lemma \ref{lem:nearlygrad}, we will need to first control the quantities $z^{(\ell)}(x,\bthetabar_N(t))$ that appear in the formulae for the gradients.

\begin{lemma}[The $z^{(\ell)}$ are close to mean-field values]
\label{lem:zmeanfield} 
Given $x\in\R^{d_X}$, $0\leq t\leq T$, $\ell\in [2:L+1]\backslash\{L\}$ and  $i_\ell\in [1:N]$, let $\soln_{[0,T]}$ denote the unique solution of the McKean-Vlasov problem and write:
\[
\Delta z^{(\ell)}_{i_\ell}(x,\bthetabar_N(t),\soln_t):=z^{(\ell)}_{i_\ell}(x,\bthetabar_N(t)) - \zbar^{(\ell)}_{i_\ell}(x,\soln_t)\in\R^{d_\ell}.
\]
Then for all $\xi\in \R_+$,
\[
\log \Ex{\exp\left\{\xi\,\abs{\Delta z^{(\ell)}_{i_\ell}(x,\bthetabar_N(t),\soln_t)}\right\}}\leq (\ell-1)\,d\log 5 + \frac{K_\ell\,\xi^2}{2N},
\]
with $d:=\max\{d_i\,:\, i\in [1:L]\}$ is the largest dimension of an inner layer and $K_\ell:= \sum_{i=2}^{\ell} (4C^{2})^{i}$, with $C$ that comes from Assumption \ref{assump:activations}). The same bound holds for the MGF of:
\[
\Delta z^{(L)}_{i_L}(x,\bthetabar_N(t),\soln_t):=z^{(L)}_{i_L}(x,\bthetabar_N(t)) - \zbar^{(L)}(x,\soln_t, \btheta^{(L)}_{i_{L}, 1}).
\]
\end{lemma}

\begin{proof}We prove this result by considering four cases. The first one is $\ell=2$. The second is $2< \ell\leq L-1$, which we handle via induction from $\ell=2$. The statements for $\ell=L,L+1$ have similar proofs, which we sketch at the end.\\

\underline{Case $\ell=2$.} The formulae for $z^{(1)}$ and $z^{(2)}$ imply:
\[z^{(2)}_{i_2}(x,\bthetabar_N(t)) = \frac{1}{N}\sum_{i_1=1}^N\fsigma{1}{\fsigma{0}{x,\thetabar_{1,i_1}^{(0)}(t)},\thetabar_{i_1,i_2}^{(1)}(t)}.\]

The definition of the ideal particles implies that, for a fixed $i_2$, the pairs
\[(\thetabar_{1,i_1}^{(0)}(t),\thetabar_{i_1,i_2}^{(1)}(t))\]
are i.i.d. with common law $(\soln)_t^{(0,1)}$. If we recall the formula for $\zbar^{(2)}$ in (\ref{eq:defzbar2}), we obtain:
\[\Ex{z^{(\ell)}_{i_\ell}(x,\bthetabar_N(t))} = \int\limits_{\R^{D_0}\times \R^{D_1}}\,\fsigma{1}{\fsigma{0}{x,a^{(0)}},a^{(1)}}\,d(\soln)_t^{(0,1)} = \zbar^{(2)}(x,\soln_t).\]
We conclude that $\Delta z^{(\ell)}_{i_\ell}(x,\bthetabar_N(t),\soln_t)$
is a centered i.i.d. sum of $N$ terms bounded by $C$. Lemma \ref{lem:MGF} finishes the proof of the case $\ell=2$.

\underline{Inductive step: $2<\ell\leq L-1$.} From \eqnref{defzell} we know that:
\begin{equation}
\label{eq:redefzell}
z^{(\ell)}_{i_\ell}(x,\bthetabar_N(t)) = \frac{1}{N}\sum_{i_{\ell-1}=1}^{N}\fsigma{\ell-1}{z_{i_{\ell-1}}^{(\ell-1)}\parn{x,\bthetabar_N(t)},\thetabar^{(\ell-1)}_{i_{\ell-1},i_\ell}(t)},
\end{equation}
whereas from \eqnref{defzbarell}
\[\zbar^{(\ell)}(x,\soln_t) = \int_{\R^{D_{\ell-1}}}\fsigma{\ell-1}{\zbar^{(\ell-1)}\parn{x,\soln_t},a^{(\ell-1)}}\,d\mu^{(\ell-1)}_t(a^{(\ell-1)}).\]
In order to compare these quantities, we introduce an intermediate term.
\begin{equation}
\label{eq:defuell}
u^{(\ell)}_{i_\ell}\parn{x,\bthetabar_N(t),\soln_t} = \frac{1}{N}\sum_{i_{\ell-1}=1}^{N}\fsigma{\ell-1}{\zbar^{(\ell-1)}\parn{x,\soln_t},\theta^{(\ell-1)}_{i_{\ell-1},i_\ell}(t)}.
\end{equation}
We define:
\begin{align}
\label{def:Delta_1}
\Delta_1:=& z^{(\ell)}_{i_\ell}\parn{x,\bthetabar_N(t)} - u^{(\ell)}\parn{x,\bthetabar_N(t),\soln_t};\\
\label{eq:defDelta2}
\Delta_2:= & u^{(\ell)}\parn{x,\bthetabar_N(t),\soln_t} - \zbar^{(\ell)}\parn{x,\soln_t},
\end{align}
we have
\[
\abs{\Delta z^{(\ell)}_{i_\ell}\parn{x,\bthetabar_N(t),\soln_t}}=\abs{\Delta_1 + \Delta_2}\leq \abs{\Delta_1} + \abs{\Delta_2}.
\]
Consider the $\sigma$-field 
\[\sG_{\ell-2} = \sigma\parn{\thetabar_{i_k,i_{k+1}}^{(k)}(t),\quad t\in[0,T],\quad 0\leq k\leq \ell-2}.\] 
Because $2\leq \ell-1\leq L-2$, the definition of the ideal particles implies that, conditionally on $\sG_{\ell-2}$, the $\thetabar^{(\ell-1)}_{i_{\ell-1},i_\ell}(t)$ are i.i.d. with law ${\soln_t}^{(\ell-1)}$, and:
\[\Ex{z^{(\ell)}_{i_\ell}(x,\bthetabar_N(t))\mid\sG_{\ell-2}}=u^{(\ell)}_{i_\ell}(x,\bthetabar_N(t),\soln_t).\]
Since the $z$'s are bounded by $C$ in norm, the upshot is that we may apply Lemma \ref{lem:MGF} conditionally on $\sG_{\ell-2}$ and obtain:
\[
\log \Ex{e^{\xi|\Delta_1|}\mid\sG_{\ell-2}}\leq d\log 5 + \frac{2C^2\xi^2}{N}\qquad\mbox{ almost surely.}
\]
In particular, since $\Delta_2$ is $\sG_{\ell-1}$-measurable,
\begin{align*}
\log \Ex{e^{\xi\,\abs{\Delta z^{(\ell)}_{i_\ell}(x,\bthetabar_N(t),\soln_t)}}} \leq &  \log \Ex{\Ex{e^{\xi|\Delta_1|}\mid\sG_{\ell-2}}\,e^{\xi\,|\Delta_2|}}]\\ 
\leq & d\log 5 + \frac{2C^2\xi^2}{N} + \log \Ex{e^{\xi\,\Delta_2}}.
\end{align*}
We now consider $\Delta_2$ Equations (\ref{eq:redefzell}) and (\ref{eq:defuell}) and the fact that $\sigma^{(\ell)}$ is $C$-Lipschitz (by Assumption \ref{assump:activations}) imply:
\[
\Delta_2\leq \frac{C}{N}\,\sum_{i_{\ell-1}=1}^N\,\abs{\Delta z^{(\ell-1)}_{i_{\ell-1}}\parn{x,\bthetabar_N(t),\soln_t}}.
\]
As a consequence,

\begin{align*}
\log \Ex{\exp\left\{\xi\abs{\Delta_2}\right\}} \,\,\leq &\,\, \log \Ex{\exp\left\{\xi\,\frac{C}{N}\,\sum_{i_{\ell-1}=1}^N\,\abs{\Delta z^{(\ell-1)}_{i_{\ell-1}}(x,\bthetabar_N(t),\soln_t)}\right\}}\\ 
\mbox{(by convexity)} \quad\leq &\,\, \frac{1}{N}\sum_{i_{\ell-1}=1}^N\log \Ex{\exp\left\{C\xi\,\abs{\Delta z^{(\ell-1)}_{i_{\ell-1}}(x,\bthetabar_N(t),\soln_t)}\right\}}\\
\mbox{(induction hyp.)} \quad\leq &\,\, (\ell-2)\,d\log 5 + \frac{4C^2K_{\ell-1}\,\xi^2}{2N}.
\end{align*}

We conclude that:
\[
\log \Ex{e^{\xi\,\abs{\Delta z^{(\ell)}_{i_\ell}(x,\bthetabar_N(t),\soln_t)}}}\leq (\ell-1)\,d\log 5 + \frac{\xi^2}{2N}\,(4C^2 + 4C^2K_{\ell-1}).
\]
This is the desired inequality because $K_{\ell} = 4C^2 + K_{\ell-1}\,4C^2$.

\underline{Cases $\ell=L,L+1$.} The proof of the case $\ell=L$ is similar to the above, so we simply outline the differences. To bound
\[\Delta z^{(L)}_{i_L}(x,\bthetabar_N(t),\soln_t):=z^{(L)}_{i_L}(x,\bthetabar_N(t)) - \zbar^{(L)}(x,\soln_t,\thetabar^{(L)}_{i_L}(t)),\]
we again introduce an intermediate term
\begin{equation}
\label{eq:defuL}
u^{(L)}_{i_L}\parn{x,\bthetabar_N(t),\soln_t} = \frac{1}{N}\sum_{i_{L-1}=1}^{N}\fsigma{L-1}{\zbar^{(L-1)}\parn{x,\soln_t},\thetabar^{(L-1)}_{i_{L-1},i_L}(t)}.
\end{equation}
and define $\Delta_1$ and $\Delta_2$ defined in analogy with (\ref{def:Delta_1}) and (\ref{eq:defDelta2}) respectively.

The MGF of $\Delta_2$ can be analyzed as above. The main difference is in the analysis of $\Delta_1$. Before, the terms in $u_{i_\ell}^{(\ell)}$ had a law $\mu^{(\ell)}$ independent of $i_L$. Now their law of each term depends on the value of $\thetabar_{i_L}(0)$. From Theorem \ref{thm:independencestructure} and Definition \ref{def:idealparticles}, we know that we may write:
\[
\thetabar^{(L-1)}_{i_{L-1},i_L}(t) = F^{(L-1)}\parn{\thetabar_{i_{L-1},i_L}^{(L-1)}(0),\thetabar^{(L)}_{i_L,1}(0)}(t)
\]
and moreover $\thetabar^{(L)}_{i_L,1}(0) = \thetabar^{(L)}_{i_L,1}(t)$ a.s.. Therefore,
\[
u^{(L)}_{i_L}\parn{x,\bthetabar_N(t),\soln_t} = \frac{1}{N}\sum_{i_{L-1}=1}^{N}h^{(L-1)}\parn{x,\soln_t,\thetabar_{i_{L-1},i_L}^{(L-1)}(0),\thetabar^{(L)}_{i_L,1}(0)}\quad\mbox{ a.s.}
\]
for a certain bounded function $h^{(L)}$, and we may bound $\Delta_1$ with a MGF computation as above.

The case $\ell=L+1$ is similar to $\ell<L$: we are back to a setting where the terms in  $u^{(L+1)}_{1}$ have a common law. We omit the details. \end{proof}

\subsection{Gradients and ideal particles}\label{sub:proofnearlygrad} We now prove Lemma \ref{lem:nearlygrad}.
\begin{proof}[Proof of Lemma \ref{lem:nearlygrad}] We will only do the proof for $\ell< L-1$; the argument for $\ell=L-1$ is only slightly different (see Remark \ref{rem:differentwhere} below).

Fix some $\ell\in [1:L-1]$ and $(i_\ell,i_{\ell+1})\in [1:N]^2$. The term $N^2\alpha(t)\,L_N(\bthetabar_N(t))$ is equal to the average of
\begin{equation}\label{eq:idealaverageof}\alpha(t)\,\left(Y-\yhat(X,\bthetabar_N(t))\right)^\dag\,M^{(\ell+1)}_{(i_{\ell+1},\boldj{\ell+2}{L+1})}(X,\bthetabar_N)\,
		     D_\theta\fsigma{\ell}{z_{i_\ell}^{(\ell)}(X,\bthetabar_N),\thetabar^{(\ell)}_{i_\ell,i_{\ell+1}}(t)}\end{equation}
over choices of $(X,Y)\sim P$ and of multi-indices $\boldj{\ell+2}{L+1}$. By Definition \ref{def:McKeanVlasov} of our McKean-Vlasov process, the time derivative of $\thetabar^{(\ell)}_{i_\ell,i_{\ell+1}}(t)$ equals the average of:
\begin{equation}
\label{eq:idealgradientaverageof}
-\alpha(t)\,\left(Y-\ybar(X,\soln_t)\right)^\dag \Mbar^{(\ell+1)}(X,\soln_t)
D_\theta\fsigma{\ell}{\zbar^{(\ell)}(X,\soln_t),\thetabar^{(\ell)}_{i_\ell,i_{\ell+1}}(t)}
\end{equation}
		     over $(X,Y)\sim P$.

\begin{remark}\label{rem:differentwhere}When $\ell=L-1$, the term $\Mbar^{(\ell+1)}$ in (\ref{eq:idealgradientaverageof}	) would also depend on $\thetabar^{(L)}_{i_L}(0)$. That is the only difference between the cases $\ell<L-1$ and $\ell=L-1$.\end{remark}

The upshot is that $\Delta_\partial\thetabar^{(\ell)}_{i_\ell,i_{\ell+1}}(t)$ is a convex combination of terms of the form (\ref{eq:idealgradientaverageof}) $+$ (\ref{eq:idealaverageof}). Since the moment generating function is convex, it suffices to show that, for {\em any} choice of $(X,Y)\in\R^{d_X}\times \R^{d_Y}$ and {\em any} choice of the multi-index $\boldj{\ell+2}{L+1}$, it holds that
\begin{equation}\label{eq:tediousinductiontrilinear}\mbox{\bf Goal: } \log \Ex{\exp(\xi\,|\mbox{(\ref{eq:idealgradientaverageof}) $+$ (\ref{eq:idealaverageof})}|)}\leq d\log 5\,L + \kappa(C,L)\,\frac{\xi^2}{2N}.\end{equation}

From now on, we fix $(X,Y)$ and the multi-index and analyze the LHS of (\ref{eq:tediousinductiontrilinear}). Unfortunately, this requires some tedious calculations. For the sake of brevity, we will be somewhat cavalier in our notation. We will introduce new random elements $A_1,\overline{A_1},\dots$ and a trilinear mapping $\sQ$ below, without specifying their domains and ranges explicitly. All of these ranges are vector spaces will also use $\|\cdot\|$ to denote the norms over any of these spaces. We will also use $V=V(C,L)$ to denote a constant that depends only on $C$ and $L$, which may change from line to line.

Define:
\begin{eqnarray}\label{eq:defAprooftrack}A_1&:=& Y-\yhat(X,\bthetabar_N(t));\\
\label{eq:defAbarprooftrack}\overline{A_1}&:=&Y-\ybar(X,\soln_t);\\
A_2&:=& D_\theta\fsigma{\ell}{z_{i_\ell}^{(\ell)}(X,\btheta_N),\thetabar^{(\ell)}_{i_\ell,i_{\ell+1}}(t)};\\
\overline{A_2}&:=& D_\theta\fsigma{\ell}{\zbar^{(\ell)}(X,\soln_t),\thetabar^{(\ell)}_{i_\ell,i_{\ell+1}}(t)};\\
A_3&:=& M^{(\ell+1)}_{(i_{\ell+1},\boldj{\ell+2}{L+1})}(X,\bthetabar_N(t));\\
\label{eq:defCbarprooftrack}\overline{A_3}&:=& \Mbar^{(\ell+1)}(X,\soln_t). \end{eqnarray}
One can see by inspection that
\[\mbox{(\ref{eq:idealgradientaverageof}) $+$ (\ref{eq:idealaverageof})} = \alpha(t)\,(\sQ(A_1,A_2,A_3) - \sQ(\overline{A_1},\overline{A_2},\overline{A_3}))\]
where $\sQ$ is a trilinear mapping satisfying:
\[\|\sQ(a,b,c)\|\leq \|a\|\,\|b\|\,\|c\|\]
for all $a,b,c$ in the appropriate spaces (and with the appropriate norms). The function $\alpha$ and the random elements in equations (\ref{eq:defAprooftrack}) through (\ref{eq:defCbarprooftrack}) are all uniformly bounded by constants depending only on $C$ and $L$, so we obtain:
\begin{align}
\nonumber 
\abs{\mbox{(\ref{eq:idealgradientaverageof}) $+$ (\ref{eq:idealaverageof})}} \leq & \supnorm{\alpha}\,\abs{Q(A_1-\overline{A_1},A_2,A_3)}\\ 
\nonumber 
& + \supnorm{\alpha}\abs{Q(\overline{A_1},A_2-\overline{A_2},A_3)}  + \supnorm{\alpha}\abs{Q(\overline{A_1},\overline{A_2},A_3-\overline{A_3})}\\ \label{eq:threetermsprooftrack}
\leq & V_{\ref{eq:threetermsprooftrack}}\,\parn{\norm{A_1-\overline{A_1}} + \norm{A_2-\overline{A_2}} + \norm{A_3-\overline{A_3}}},
\end{align}
with $V_{\ref{eq:threetermsprooftrack}}=V_{\ref{eq:threetermsprooftrack}}(C,L)>0$.
Let us now consider each of the terms of the RHS of (\ref{eq:threetermsprooftrack}). Note that
\begin{align*}
A_1 - \overline{A_1} =&\ybar(X,\soln_t) -\yhat(X,\bthetabar_N(t))\\ 
=& \fsigma{L+1}{\zbar^{(L+1)}(x,\soln_t)} - \fsigma{L+1}{z_{1}^{(L+1)}(x,\bthetabar_N(t))}.
\end{align*}
Since $\fsigma{L+1}{\cdot}$ is $C$-Lipschitz, we obtain:
\[
\norm{A_1 - \overline{A_1}} \leq C\,\abs{\Delta z^{(L+1)}_{1}(X,\bthetabar_N(t),\soln_t)},
\]
with $\Delta z^{(L+1)}$ defined in Lemma \ref{lem:zmeanfield}.
Similarly,
\[
A_2 - \overline{A_2} = D_\theta\fsigma{\ell}{z_{i_\ell}^{(\ell)}(X,\btheta_N),\thetabar^{(\ell)}_{i_\ell,i_{\ell+1}}(t)} -  
D_\theta\fsigma{\ell}{\zbar^{(\ell)}(X,\soln_t),\thetabar^{(\ell)}_{i_\ell,i_{\ell+1}}(t)},
\]
and $D_\theta\fsigma{\ell}{\cdot}$ is $C$-Lipschitz, so:
\[
\norm{A_2 - \overline{A_2}} \leq C\,\abs{\Delta z^{(\ell)}_{i_\ell}(X,\bthetabar_N(t),\soln_t)}.
\]
To compare $A_3$ and $\overline{A_3}$, we use the formulas (\ref{eq:defM}) and (\ref{eq:defMbarell}) This shows that the two terms $A_3$ and $\overline{A_3}$ are products of $L-\ell+1\leq L+1$ terms of the form:
\[
D_z\fsigma{\ell}{z_{j_k}^{k}(x,\btheta_N),\thetabar^{(\ell)}_{j_k,j_{k+1}}(t)}\mbox{ vs. }D_z\fsigma{\ell}{\zbar^{(k)}(x,\soln_t),\thetabar^{(k)}_{j_k,j_{k+1}}(t)},
\]
respectively. Each of these terms is $C$-Lipschitz and  $C$-Uniformly Bounded. By comparing the two products as we did for $\sQ$ -- that is, changing terms in the product one by one --, we obtain:
\[
\norm{A_3 - \overline{A_3}} \leq  C\,\abs{\Delta z^{(\ell+1)}_{i_{\ell+1}}(X,\bthetabar_N(t),\soln_t)} + C\,\sum_{k=\ell+2}^{L+1}\abs{\Delta z^{(k)}_{j_k}(X,\bthetabar_N(t),\soln_t)}.
\]
We deduce from (\ref{eq:threetermsprooftrack}) that:
\begin{align*}
\abs{\mbox{(\ref{eq:idealgradientaverageof}) $+$ (\ref{eq:idealaverageof})}} \leq & C\,\abs{\Delta z^{(\ell)}_{i_\ell}(X,\bthetabar_N(t),\soln_t)} + C\,\abs{\Delta z^{(\ell+1)}_{i_{\ell+1}}(X,\bthetabar_N(t),\soln_t)}\\ 
& + C\,\sum_{k=\ell+2}^{L+1}\abs{\Delta z^{(k)}_{j_k}(X,\bthetabar_N(t),\soln_t)}.
\end{align*}
Set $V':=V(L-\ell+1)$. By H\"{o}lder's inequality and Lemma \ref{lem:zmeanfield}, for all $\xi>0$:
\begin{align*}
&\log\Ex{\exp(\xi\,\abs{\mbox{(\ref{eq:idealgradientaverageof}) $+$ (\ref{eq:idealaverageof})}})}\leq \frac{1}{L-\ell+1}\log \Ex{\exp(V'\xi\,\abs{\Delta z^{(\ell)}_{i_\ell}(X,\bthetabar_N(t))})}\\ 
& \quad\quad+ \frac{1}{L-\ell+1}\log \Ex{\exp(V'\xi\,\abs{\Delta z^{(\ell+1)}_{i_{\ell+1}}(X,\bthetabar_N(t))})}\\
& \quad\quad+ \frac{1}{L-\ell+1}\sum_{k=\ell+2}^{L+1}\log \Ex{\exp(V'\xi\,\abs{\Delta z^{(k)}_{j_k}(X,\bthetabar_N(t))})} \\ 
&\hphantom{\log\Ex{\exp(\xi\,\abs{\mbox{(\ref{eq:idealgradientaverageof}) $+$ (\ref{eq:idealaverageof})}})}}\leq Ld\log 5 + C_{\ref{lem:nearlygrad}}\,\frac{\xi^2}{2N},
\end{align*}
where $C_{\ref{lem:nearlygrad}} = C_{\ref{lem:nearlygrad}}(C,L) > 0$, completing the proof.
\end{proof}

%% file: measureNN_aux/11-coupling.tex
This section combines the elements of the previous two sections to prove Theorems \ref{thm:error-approx} (on approximation of the loss function) and \ref{thm:microscopic} (on approximating weights by ideal particles). We will focus on the first proof most of the time, and then argue that Theorem \ref{thm:microscopic} follows from the same calculations.

\subsection{Loss function approximation} To prove Theorem \ref{thm:error-approx}, we will control the successive approximations of the process developed in the previous sections. 
\begin{align}\label{eq:master-inequality}
&\Exp{\stackrel{(X_k,Y_k) \stackrel{\scalebox{.5}{\mbox{\text{iid}}}}{\sim} P}{\btheta_N(0)\sim \mu_0}}{\abs{\L_N(\btheta_N(k)) - \Lbar(\soln_{k\eps})}}\leq
\overset{{( I )}}{\overbrace{\Exp{(X_k,Y_k) \stackrel{\tiny{iid}}{\sim} P}{\abs{\L_N(\btheta_N(k)) - \L_N(\thetatil_N(\eps k))}}}}\\
&\quad\quad+\overset{{( II )}}{\overbrace{\Exp{\thetatil_N(0)\sim \mu_0}{\abs{\L_N(\thetatil_N(\eps\,k)) - \L_N(\thetabar_N(\eps k)}}}}
+\overset{{( III )}}{\overbrace{\Exp{\thetabar_N(0)\sim \mu_0}{\abs{\L_N(\thetabar_N(\eps k)) - \Lbar(\soln_{k\eps})}}}},\nonumber
\end{align}
where
\begin{itemize}
\item $\btheta_N(\cdot)$ are the weights in the original DNN, which evolve in discrete time;
 \item  $\thetatil_N(\cdot)$ is the continuous time gradient descent process from Section \ref{sec:SGDtoCTGD}; 
 \item $\thetabar_N(\cdot)$ is the ideal particle process introduced in Section \ref{sec:ideal}; and
 \item $\soln$ is the solution to the McKean-Vlasov problem in Definition \ref{def:McKeanVlasov}.
 \end{itemize} 
 
 We deal with term $(I)$ in the following lemma:

\begin{lemma}\label{lem:error-sgd-ctgd}
Let $\btheta_N$ be the process from Definition \ref{def:sgd-process}, and $\thetatil_N$ the process appearing in Definition \ref{def:ctgd-process}, both sharing the same initial contidion $\btheta_N(0) = \thetatil_N(0)$. Then, no matter which initial condition is picked, it holds for all $k\in\left[0:\left\lceil\frac{T}{\epsilon}\right\rceil\right]$:
\begin{equation}
\label{eq:error-bound-sgd-ctgd}
\mathbb{E}_{(X_r,Y_r)\iid P}\Big[\abs{\L_N\parn{\btheta_N(k)} - \L_N\parn{\thetatil_N(\eps\cdot k)}}\Big] \leq C_{\ref{eq:error-bound-sgd-ctgd}}\cdot \parn{\eps+ \sqrt{\eps\,d}} 
\end{equation}
where $d:=\max\{d_i,1\leq i\leq L-1\}$ and $C_{\ref{eq:error-bound-sgd-ctgd}}$ depends only on $L$ and $C$.
\end{lemma}

\begin{proof}
Notice that according to Lemma \ref{lem:L_N-lipschitz} in the Appendix, we have that the function $\L_N$ is Lipschitz. Therefore, it holds that:
\begin{equation}
\mathbb{E}_{(X_r,Y_r)\iid P}\Big[\abs{\L_N\parn{\btheta_N(k)} - \L_N\parn{\thetatil_N(\eps\cdot k)}}\Big] \leq
C_{\ref{lem:L_N-lipschitz}} \Exp{(X_k,Y_k)\iid P}{\Lnorm{\btheta_N(k) - \thetatil_N(\eps\cdot k)}}
\end{equation}

Now, we will use the result in proposition \ref{prop:SGDtoCTGD} together with the elementary lemma \ref{lem:expected-upper-bound} to get:

\[\Exp{(X_k,Y_k)\iid P}{\Lnorm{\btheta_N(k) - \thetatil_N(\eps k)}} \leq
 \cdot C_{\ref{prop:SGDtoCTGD}} \parn{\eps +\sqrt{\eps\,d} + \frac{\sqrt{\eps \pi}}{2}}.\]\end{proof}

The next step is to find an upper bound for $(II)$. That task is accomplished in the next lemma:
\begin{lemma}
\label{lem:error-bound-ctgd-ideal}
Let $\thetatil_N$ be the process described in Definition \ref{def:ctgd-process} and $\thetabar_N$ the process appearing in Definition \ref{def:idealparticles}, both sharing an initial condition
$\thetatil_N(0) = \thetabar_N(0)\sim \mu_0$ as in Assumption \ref{assump:initialization}. Then it holds for all $t\in [0,T]$:
\begin{equation}
\label{eq:error-bound-ctgd-ideal}
\mathbb{E}_{\mu_0}\Big[\abs{\L_N\parn{\thetatil_N(t)} - \L_N\parn{\thetabar_N(t)}}\Big] \leq \frac{C_{\ref{eq:error-bound-ctgd-ideal}}\sqrt{d}}{\sqrt{N}}
\end{equation}
for some constant $C_{\ref{eq:error-bound-ctgd-ideal}}$ which depends only on $T, L$ and $C$. Here $d:=\max\{d_i,1\leq i\leq L-1\}$ as in the previous Lemma.\end{lemma}

\begin{proof}
This Lemma is a direct consequence of Lemma \ref{lem:nearlygrad}. Recall from (\ref{eq:deriv-thetatil}) that for $\ell \in [1:L-1]$
\[\frac{d\thetatil^{(\ell)}_{i_\ell,i_{\ell+1}}(t)}{dt} = - N^2\alpha(t)\frac{\partial \L_N(\thetatil_N(s))}{\partial\thetatil^{(\ell)}_{i_\ell,i_{\ell+1}}}.\]
We use this to derive the following inequality:
\begin{align}
\label{eq:a-b-inequality}
\abs{\thetatil^{(\ell)}_{i_\ell,i_{\ell+1}}(t) - \thetabar^{(\ell)}_{i_\ell,i_{\ell+1}}(t)} &\leq \int_0^t\abs{\frac{d\thetatil^{(\ell)}_{i_\ell,i_{\ell+1}}(s)}{dt} - \frac{d\thetabar^{(\ell)}_{i_\ell,i_{\ell+1}}(s)}{dt}}ds.\nonumber\\
&\leq \overset{(a)}{\overbrace{\int_0^t \alpha(s)\abs{N^2\frac{\partial \L_N(\thetatil_N(s))}{\partial\thetatil^{(\ell)}_{i_\ell,i_{\ell+1}}} - N^2\frac{\partial \L_N(\thetabar_N(s))}{\partial\thetabar^{(\ell)}_{i_\ell,i_{\ell+1}}}}ds} }\\
&\quad\quad\quad +\overset{(b)}{\overbrace{\int_0^t \abs{N^2\alpha(s)\frac{\partial \L_N(\thetabar_N(s))}{\partial\thetabar^{(\ell)}_{i_\ell,i_{\ell+1}}} + \frac{d\thetabar^{(\ell)}_{i_\ell,i_{\ell+1}}(s)}{dt}} ds}}.\nonumber
\end{align}
We now work on finding good upper bounds for $(a)$ and $(b)$. For $(a)$, we need the fact that, according to remark \ref{rem:lipschitz-y_N}, for any choice of $\ell\in[1:L-1]$ and $(i_\ell,i_{\ell+1})\in[1:N]^2$, the functions $N^2\frac{\partial \L_N(\,\cdot\,)}{\partial\thetatil^{(\ell)}_{i_\ell,i_{\ell+1}}}$ are Lipschitz with a common constant $C_{\ref{rem:lipschitz-y_N}}$. This means that
\[
(a)\leq C\cdot C_{\ref{rem:lipschitz-y_N}} \cdot\int_0^t\Lnorm{\thetatil_N(s) - \thetabar_N(s)}ds
\]
Now, according to Lemma \ref{lem:nearlygrad}, with probability greater than $1-e^{-u^2}$ it holds that
\[\abs{N^2\alpha(t)\frac{\partial \L_N(\thetabar_N(t))}{\partial\thetabar^{(\ell)}_{i_\ell,i_{\ell+1}}} + \frac{d\thetabar^{(\ell)}_{i_\ell,i_{\ell+1}}(t)}{dt} }\leq \sqrt{\frac{C_{\ref{lem:nearlygrad}}}{N}}\parn{u + \sqrt{Ld\log(5)}} .
\]
This means we can make:
\[(b)\leq t\cdot\sqrt{\frac{C_{\ref{lem:nearlygrad}}}{N}}\parn{u + \sqrt{Ld\log(5)}}\leq T\cdot\sqrt{\frac{C_{\ref{lem:nearlygrad}}}{N}}\parn{u + \sqrt{Ld\log(5)}}.\]

Notice that both these bounds are independent of $\ell\in[1:L-1]$ and of $(i_\ell,i_{\ell+1})\in[1:N]^2$. Since for $\ell \in\{0,L\}$ the difference is always zero, we can apply the supremum on the LHS of (\ref{eq:a-b-inequality}) together with the previous bounds to get:
\[
\Lnorm{\thetatil_N(s) - \thetabar_N(s)} \leq  C\cdot C_{\ref{rem:lipschitz-y_N}} \cdot\int_0^t\Lnorm{\thetatil_N(s) - \thetabar_N(s)}ds + T\cdot\sqrt{\frac{C_{\ref{lem:nearlygrad}}}{N}}\parn{u + \sqrt{Ld\log(5)}}.
\]

Follow that up with a Gronwall Inequality to get:
\[
\Lnorm{\thetatil_N(s) - \thetabar_N(s)} \leq  T^2\cdot\sqrt{\frac{C_{\ref{lem:nearlygrad}}}{N}}\parn{u + \sqrt{Ld\log(5)}} e^{C\cdot C_{\ref{rem:lipschitz-y_N}} \cdot\ T}.
\]

Recall that this inequality holds in a set of probability higher than $1-e^{-u^2}$. To close things off, we recall that according to \ref{lem:L_N-lipschitz}, $\L_N$ is a
Lipschitz function. Then, we invoke Lemma \ref{lem:expected-upper-bound} to find:
\[
\Exp{\mu_0}{\abs{\L_N(\thetatil_N(s)) - \L_N(\thetabar_N(s)})}\leq C_{\ref{lem:L_N-lipschitz}}\cdot
\Exp{\mu_0}{\Lnorm{\thetatil_N(s) - \thetabar_N(s)}} \leq  \frac{C_{\ref{eq:error-bound-ctgd-ideal}}\sqrt{d}}{\sqrt{N}},
\]
where $C_{\ref{eq:error-bound-ctgd-ideal}} = C_{\ref{lem:L_N-lipschitz}}\cdot\sqrt{C_{\ref{lem:nearlygrad}}}\parn{\frac{\sqrt{\pi}}{2} + \sqrt{L\log(5)}} e^{C\cdot C_{\ref{rem:lipschitz-y_N}} \cdot\ T}\cdot T^2$.
\end{proof}

Finally, we need to find a proper bound for $(III)$, which is achieved by the following lemma:

\begin{lemma}
\label{lem:error-bound-ideal-vlasov}
Let $\thetabar_N$ be the ideal particle process defined in section \ref{sec:ideal} and $\soln_{[0,T]}$ be the unique solution to the McKean-Vlasov problem 
 under assumptions \ref{assump:initialization} and \ref{assump:activations}. Then, for $t\in[0:T]$ the following inequality holds:
\begin{equation}
\label{eq:error-bound-ideal-vlasov}
\Exp{\mu_0}{\abs{\L_N(\thetabar_N(t)) - \Lbar(\soln_t)}} \leq \frac{C_{\ref{eq:error-bound-ideal-vlasov}}\sqrt{d}}{\sqrt{N}}
\end{equation}
\end{lemma}

\begin{proof}
To start the proof of this lemma, we rewrite the Loss function difference as:
\begin{align*}
\L_N(\thetabar_N(t)) - \Lbar(\soln_t)&=\frac{1}{2} \mathbb{E}_{(X,Y)\sim P}\left[\abs{\yhat(X,\thetabar_N(t))}^2 - \abs{\ybar(X,\soln_t)}^2\right.\\
&\quad\quad\quad\quad\quad\quad\quad\quad\quad\left.+ 2Y^\dag(\ybar(X,\soln_t) - \yhat(X,\thetabar_N(t)))\right].
\end{align*}
Then, using the fact that $Y,\yhat, \ybar$ are $C $-bounded we get:
\[
\abs{\L_N(\thetabar_N(t)) - \Lbar(\soln_t)}\leq 4\cdot C  \cdot\Exp{(X,Y)\sim P}{\abs{\ybar(X,\soln_t)-\yhat(X,\thetabar_N(t))}},
\]
Now, by recalling the definition of $\yhat$ in (\ref{eq:outputNN}), the definition of $\ybar$ in (\ref{eq:defzbarL+1}), and using the fact that $\sigma^{(L+1)}$ is $C$-Lipschitz according to
Assumption \ref{assump:activations}, we get:
\[
\abs{\L_N(\thetabar_N(t)) - \Lbar(\soln_t)} \leq 4\cdot C^2\cdot\Exp{(X,Y)\sim P}{\abs{z^{(L+1)}_{1}(x,\bthetabar_N(t)) - \zbar^{(L+1)}(x,\soln_t)}}
\]

Finally, we invoke lemma \ref{lem:zmeanfield} together with lemma \ref{lem:expected-upper-bound} to conclude:
\[
\Exp{\mu_0}{\abs{\L_N(\thetabar_N(t)) - \Lbar(\soln_t)}} \leq \frac{C_{\ref{eq:error-bound-ideal-vlasov}}\sqrt{d}}{\sqrt{N}},
\]
where $C_{\ref{eq:error-bound-ideal-vlasov}} = 4\cdot C^2\sqrt{2 K_{L+1}}\parn{\frac{\sqrt{\pi}}{2} + \sqrt{L\cdot \log(5)}}$
\end{proof}

\subsection{Approximation by ideal particles}\label{sub:proof:microscopic} To prove Theorem \ref{thm:microscopic}, we first note that, by symmetry,
\[\Exp{\stackrel{(X_k,Y_k) \stackrel{\tiny{iid}}{\sim} P}{\btheta_N(0)\sim \mu_0}}{\sum_{\ell=0}^L|\theta^{(\ell)}_{i_\ell,i_{\ell+1}}(k) - \thetabar^{(\ell)}_{i_\ell,i_{\ell+1}}(k\eps)|}\]
does not depend on the specific choice of input-to-output path. Therefore, we only need to bound:
\[\Exp{\stackrel{(X_k,Y_k) \stackrel{\tiny{iid}}{\sim} P}{\btheta_N(0)\sim \mu_0}}{\sum_{\ell=0}^L 
\parn{\sum_{i_\ell=1}^N\sum_{i_{\ell+1}=1}^N
\abs{\theta^{(\ell)}_{i_\ell,i_{\ell+1}}(k) - \thetabar^{(\ell)}_{i_\ell,i_{\ell+1}}(k\eps)}}}.\]

This was done implictly in the proofs of the three lemmas. We omit the details.

%% file: measureNN_aux/12-conclusion.tex
From our point of view, the key contribution of our paper is to deal with some of the intricacies of mean-field limits for deep neural networks. Chief among those is the fact that DNNs lead to McKean-Vlasov problems involving conditional densities. Our results are obtained at the cost of considering that weights close to the input and output are "random features" that are not learned. However, as noted in Remark \ref{rem:timescales}, our results should extend to the setting where weights close to input and output have a higher learning rate. We also believe that one could add noise to our analysis as in \cite{MeiArxiv2018}. This is the approach of \cite{Sirignano2019,Nguyen2019}. Another challenge would be to remove the boundedness and smoothness conditions on the nolinearities (cf. Assumption \ref{assump:activations}), and also to ameliorate the dimensional dependence of our bounds (cf. \cite{Mei2019}). 

To our mind, however, there are other, more pressing problems. One is to find a good way to model other DNN architectures.   An obvious first goal, which is also related to supervised learning, is to study convolutional layers in DNNs \cite{Goodfellow2016,lecun2015}. These are fundamentally different even at the heuristic level. Another would be to develop methods to study the long-term behavior of DNNs, following the lead of Mei et al \cite{Mei2018,MeiArxiv2018}. The main challenge here is that, as far as we can see, the PDEs for our densities have not been studied before. 

A different direction is to develop a theory of nonparametric learning and optimization inspired by our result. Theorem \ref{thm:error-approx} implies that, after $O(1/\epsilon)$ steps of training, our network computes (in the limit of large $N$) a function $\ybar$ which is a {\em composition} of functions of the form
\[z^{(\ell)}\mapsto \int_{\R^{D_\ell}}\,h^{(\ell)}\left(z^{(\ell)},\theta\right)\,d\nu^{(\ell)}(\theta),\]
where the $h^{(\ell)}$ are defined by the activation functions and the $\nu^{(\ell)}$ are given by marginals of $\soln_t$ at a specific time $t$. It would be interesting to study the {\em general problem} of learning  such functions, that is, of choosing the measures $\nu^{(\ell)}$ (with the $h^{(\ell)}$ fixed) so as to approximate $Y$ well -- in terms of algorithmic and sample complexity. From this perspective, the DNNs are simply one method to approximate this problem. \\

\noindent{\bf Acknowledgements.} During the late stages of the preparation of this paper, RIO was a participant of~"The rough high dimensional landscape problem" program in the Kavli Institute for Theoretical Physics (KITP) in Santa Barbara, CA. We thank program organizers Giulio Biroli, Chiara Cammarota, Patrick Charbonneau and Andrea Montanari for the invitation. We also thank other participants for their questions regarding this work, and the KITP and its staff for their hospitality.

%% file: measureNN_aux/A1-technicalities.tex
\subsection{Proofs of technical lemmas on ODE-type problems} We prove here some technical results on Cauchy ODE-type problems needed in Section \ref{sec:apriori_space}.

\subsubsection{Proof of Proposition \ref{prop:Picardsimple}}\label{sub:proofPicardsimple} This result was stated and used in the proof of Lemma \ref{lem:specialclosedpsi}, which showed that the set of $R$-special measures (Definition \ref{def:Rspecialmeasure}) are closed under the map $\psi$ from Corollary \ref{cor:mappingwelldefined}.
\begin{proof}Recall that we assume $V>0$, $R\geq 2V$ and \[H:[0,T]\times \R^{D_{L-1}}\times \R^{D_{L}}\to \R^{D_{L-1}}\]
satisfy the following properties:
\begin{enumerate}
\item $H$ is uniformly bounded by $V>0$ in norm;
\item for each fixed $t\in [0,T]$ and $a^{(L-1)}\in \R^{D_{L-1}}$, the map
\[a^{(L)}\in \R^{D_{L}}\mapsto H(t,a^{(L-1)},a^{(L)})\] is $V\,e^{Rt}$-Lipschitz;
\item for each fixed $t\in [0,T]$ and $a^{(L)}\in \R^{D_{L}}$,
\[a^{(L-1)}\in \R^{D_{L}}\mapsto H(t,a^{(L-1)},a^{(L)})\] is $V$-Lipschitz.
\end{enumerate}

Under these conditions, the Picard-Lindel\"{o}f Theorem implies that, for each pair $a=(a^{(L-1)},a^{(L)})\in \R^{D_{L-1}}\times \R^{D_{L}}$, there exists a unique solution $G(a)\in C([0,T],\R^{D_{L}})$ to the problem:
\[\left\{\begin{array}{llll}x(0) &=& a^{(L-1)}\\
\frac{dx}{dt}(t) &=& H(t,x(t),a^{(L)}) & (t\in [0,T]).\end{array}\right.\]
In particular, $G:a\mapsto G(a)$ defines a function from $\R^{D_{L-1}}\times \R^{D_L}$ to $C([0,T],\R^{D_{L}})$.

We claim that $G$ is $R$-special, as per Definition \ref{def:Rspecialfunction} for any $R\geq 2V$. To see this, notice first that $G(a)(0)=a^{(L-1)}$ and $G(a)$ is $V$-Lipschitz because $H$ is bounded by $V$ in norm.

What is left is to show that $a\mapsto G(a)(t)$ is $e^{Rt}$-Lipschitz for each $t\in [0,T]$. To this end, consider $\epsilon>0$, $t\in [0,T]$, and define:
\[\Delta_{\epsilon}(t):= \sup\left\{\frac{|G(a)(t) - G(b)(t)|}{\epsilon}\,:\,a,b\in \R^{D_{L-1}}\times \R^{D_L}, |a-b|\leq \epsilon.\right\}.\]
We will show that \[\Delta_{\epsilon}(t)\leq e^{Rt} \mbox{ for all }t\in [0,T]\mbox{ and }\epsilon>0,\]  which implies the desired Lipschitz property.

We will bound $\Delta_{\epsilon}(t)$ via a Gronwall-style argument. First note that $\Delta_{\epsilon}(0)\leq 1$. Also, since, $t\mapsto G(a)(t)$ is $V$-Lipschitz for all $a\in\R^{D_{L-1}}\times \R^{D_L}$, one can show that $\Delta_{\epsilon}(t)\leq 1+2Rt/\epsilon<+\infty$ for all $t$, and moreover $\Delta_{\epsilon}(t)$ is Lipschitz continuous. For $s\in [0,T]$ and $a,b\in\R^{D_{L-1}}\times \R^{D_L}$ with $|a-b|\leq \epsilon$, our assumptions on $H$ imply:
\begin{eqnarray*} \frac{|H(s,G(a)(s),a^{(L)}) - H(s,G(b)(s),b^{(L)})|}{\epsilon}& \leq &\frac{V\,|G(a)(s) - G(b)(s)|}{\epsilon}\\ & & +  \frac{V\, e^{Rs}|a^{(L)}-b^{(L)}|}{\epsilon}\\  & \leq & V\,(\Delta_{\epsilon}(s)+e^{Rs}).\end{eqnarray*}
Therefore, for all $t\in [0,T]$ and $a,b$ as above:
\begin{eqnarray*}\frac{|G(a)(t) - G(b)(t)|}{\epsilon} &\leq & \frac{|a^{(L-1)} - b^{(L-1)}|}{\epsilon} \\ & & + \int_{0}^{t} \frac{|H(s,G(a)(s),a^{(L)}) - H(s,G(b)(s),b^{(L)})|}{\epsilon}\,ds \\ &\leq & 1 +  \int_{0}^{t}V\,(\Delta_{a,b}(s) +e^{Rs})\,ds,\end{eqnarray*}
which implies:
\[\forall t\in [0,T]\,:\,\Delta_{\epsilon}(t)\leq  1 +  \int_{0}^{t}V\,(\Delta_{\epsilon}(s) +e^{Rs})\,ds.\]
Since $t\mapsto \Delta_{\epsilon}(t)$ is continuous, one can use Picard iteration with the above expression and deduce that \[\forall t\in [0,T]\,:\,\Delta_{\epsilon}(t)\leq \xi(t),\] where $\xi(\cdot)$ is the unique solution of the Cauchy Problem:
\[\xi(t) = 1 + \int_0^t\, V\,(\xi(s)+e^{Rs})\,ds\,\,(t\geq 0).\]
It is easily seen that, if $R>V$, then
\[(e^{-Vt}\xi(t))' = V\,e^{(R-V)t}\Rightarrow \xi(t) = \left(\frac{V}{R-V}\right)\,e^{Rt} + \left(1-\frac{V}{R-V}\right)\,e^{Vt}.\]
In particular, since $R\geq 2V$, we have that $\Delta_{\epsilon}(t)\leq \xi(t)\leq e^{Rt}$, as desired. \end{proof}

\subsubsection{Proof of Proposition \ref{prop:Picardcomplicated}}\label{sub:proofPicardcomplicated} This result was stated and used in the proof of Lemma \ref{lem:specialcontainssoln}, where we showed that all solutions to the McKean-Vlasov problem in Definition \ref{def:McKeanVlasov} are $R$-special measures (cf. Definition \ref{def:Rspecialmeasure}).
\begin{proof} \underline{Step 1: preliminaries.} Recall the definition of $\sU$ in (\ref{eq:defsU}) and of $\dist(\cdot,\cdot\cdot;t)$ in (\ref{eq:distU}) (also recalled in (\ref{eq:distU2}) below). The following simple fact will be convenient.
\begin{fact}[Proof omitted]\label{fact:Ucomplete} For any $g_1,g_2\in\sU$ and $t\in [0,T]$,
\begin{equation}\label{eq:distU2} \dist(g_1,g_2;t):=\sup\limits_{a\in \R^{D_{L-1}}\times \R^{D_L}}|g_1(a)(t) - g_2(a)(t)|\leq 2Rt<+\infty\end{equation}
because $g_1(a)(0)=g_2(a)(0)$ and both $g_1,g_2$ are $R$-Lipschitz. For the same reason, $t\mapsto  \dist(g_1,g_2;t)$ is continuous for any fixed $g_1,g_2\in\sU$. Finally,
\[ (g_1,g_2)\in\sU^2\mapsto \dist(g_1,g_2):=\max_{t\in[0,T]}\dist(g_1,g_2;t)\]
is a metric over $\sU$, and $(\sU,\dist)$ is a complete metric space.\end{fact}

Also recall our assumptions: $2R\geq V>0$ are constants and
\[\Xi:[0,T]\times \R^{D_{L-1}}\times \R^{D_L}\times \sU\to \R^{D_{L-1}}\] satisfies the following properties:
\begin{enumerate}
\item $\Xi$ is uniformly bounded by $V$.
\item If two functions $g_1,g_2\in\sU$ are such that \[g_1(a^{(L-1)},a^{(L)})=g_2(a^{(L-1)},a^{(L)})\] for $\mu_0^{(L-1,L)}$-a.e. $(a^{(L-1)},a^{(L)})$, then:
\[\Xi(t,a^{(L-1)},a^{(L)},g_1)=\Xi(t,a^{(L-1)},a^{(L)},g_2)\]
for $\mu_0^{(L-1,L)}$-a.e. $(a^{(L-1)},a^{(L)})$;
\item for any fixed $g\in \sU$, the map
\[(t,a^{(L-1)},a^{(L)})\in  [0,T]\times \R^{D_{L-1}}\times \R^{D_L}\mapsto \Xi(t,a^{(L-1)},a^{(L)},g)\in \R^{D_{L-1}}\]
is measurable; if $a^{(L)}$ is also fixed, then $(t,a^{(L-1)})\mapsto \Xi(t,a^{(L-1)},a^{(L)},g)$ is continuous;
\item for any fixed $t\in [0,T]$, $g\in \sU$ and $a,b\in \R^{D_{L-1}}\times \R^{D_L}$ with $|a-b|\leq \epsilon$
\begin{eqnarray*}& |\Xi(t,a^{(L-1)},a^{(L)},g) - \Xi(t,b^{(L-1)},b^{(L)},g)|\\ \leq & V\,\sup\limits_{a',b'\in\R^{D_{L-1}}\times \R^{D_L}\,:\, |a'-b'|\leq \epsilon}|g(a')(t)-g(b')(t)|;\end{eqnarray*}
\item for any fixed $t\in [0,T]$, $g_1,g_2\in \sU$ and $a\in \R^{D_{L-1}}\times \R^{D_L}$:
\[|\Xi(t,a^{(L-1)},a^{(L)},g_1) - \Xi(t,a^{(L-1)},a^{(L)},g_2)|\leq V\,\dist(g_1,g_2;t).\]
\end{enumerate}

For $g\in\sU$, we define a new function $\sT(g)\in \sU$, so that, if $a = (a^{(L-1)},a^{(L)})\in \R^{D_{L-1}}\times \R^{D_L}$, $\sT(g)(a)\in \sU$
is a function satisfying:
\begin{equation}\label{eq:defTsoln}\forall t\in [0,T]\,:\,\sT(g)(a)(t) = a^{(L-1)} + \int_0^t\,\Xi(s,g(a)(t),a^{(L)},g)\,ds.\end{equation}
Notice that this integral is a Riemann integral, since $\Xi$ is continuous in the first two variables.
Also observe that that $\sT(g)$ is measurable, $\sT(g)(a)(0)=a^{(L-1)}$, and $t\mapsto \sT(g)(a)(t)$ is $V$-Lipschitz (and also $R$-Lipschitz) ecause $\Xi$ is bounded by $V$ in norm. We conclude that $\sT(g)\in\sU$ for all $g\in\sU$, so $\sT:\sU\to\sU$ is well defined. \\

\noindent \underline{Step 2: existence, uniqueness and convergence to fixed point.} We will show that:
\begin{equation}\label{eq:goal2specialcontainssoln}\mbox{\bf Goal (step 2): } \exists G_o\in\sU\;\forall g\in \sU\,:\,\lim_{j\to +\infty}\dist(\sT^{j}(g),G_o)=0,\end{equation}
which implies in particular that $G_o$ is the {\em unique} fixed point of $\sT$. 

We achieve this goal this via a Picard-type iteration. Note that for all $g_1,g_2\in\sU$, item (5) in the assumptions and equation (\ref{eq:defTsoln}) imply:
\[\dist(\sT(g_1),\sT(g_2);t)\leq  V\,\int_0^t\,\dist(g_1,g_2;s)\,ds.\]
Since $\dist(\cdot,\cdot\cdot;t)$ is a continuous function of $t$ that is upper bounded by $\dist$, we may iterate this inequality to deduce that
\[\dist(\sT^j(g_1),\sT^j(g_2);t)\leq V^j\,\idotsint\limits_{0\leq t_j\leq t_{j-1}\leq \dots \leq t_1\leq t}\dist(g_1,g_2;t_j)\,dt_j\,dt_{j-1}\dots dt_1,\]
from which we obtain
\[\dist(\sT^j(g_1),\sT^j(g_2))\leq \frac{(VT)^j}{j!}\,\dist(g_1,g_2).\]
Recall from Fact \ref{fact:Ucomplete} that $(\sU,\dist)$ is complete. Therefore, the Banach fixed point theorem shows that $\sT$ has a unique fixed point $G$ and that $\lim_j\dist(\sT^j(g),G_o)= 0$ for all $g\in\sU$. In particular, this $G_o$ must satisfy the following identity for all $(t,a)$ in the appropriate set.
\begin{equation}\label{eq:defTsoln}G_o(a)(t) = a^{(L-1)} + \int_0^t\,\Xi(s,G_o(a)(t),a^{(L)},G)\,ds.\end{equation}

\noindent \underline{Step 3: any $\tilde{G}$ with $\tilde{G}=\sT(\tilde{G})$ is a.e. equal to $G_o$}. This follows directly from (\ref{eq:goal2specialcontainssoln}) and assumption (2), which implies that $\sT^j(\tilde{G})=\tilde{G}$ a.s. for all $j\in\N$. \\

\noindent \underline{Step 4: $G_o$ is $R$-special.} Since $G_o\in\sU$, it suffices to show that
\[\Delta_\epsilon(t):= \frac{1}{\epsilon}\sup\{|G_o(a)(t)-G_o(b)(t)|\,:\, a,b\in \R^{D_{L-1}}\times\R^{D_L},\,|a-b|\leq \epsilon\}.\]
satisfies $\Delta_\epsilon(t)\leq e^{2Vt}$ (recall the assumption that $R\geq 2V$.

To estimate this quantity, we note first that $\Delta_\epsilon(0)\leq 1$ and $\Delta_{\epsilon}(t)\leq 1 + 2Rt/\epsilon$ because $G_o(a)\in C([0,T],\R^{D_{L-1}})$ is $R$-Lipschitz for each $a$. For the same reason, $\Delta_{\epsilon}$ is $2R$-Lipschitz. Using assumption (4), and given $a,b$ with $|a-b|\leq \epsilon$, we estimate:
\begin{multline*}\frac{|\Xi(t,G_o(a)(t),a^{(L)},G_o) - \Xi(t,G_o(b)(t),b^{(L)},G_o)|}{\epsilon}\\  \leq V\,\frac{|G_o(a)(t) - G_o(b)(t)| + |a^{(L)} - b^{(L)}|}{\epsilon}\leq V\,(\Delta_\epsilon(t) + 1).\end{multline*}
We conclude
\begin{multline*}\frac{|G_o(a)(t) - G_o(b)(t)|}{\epsilon} \leq  \frac{|a^{(L-1)} - b^{(L-1)}|}{\epsilon} \\  + \int_0^t\frac{|\Xi(s,G_o(a)(s),a^{(L)},G_o) - \Xi(s,G_o(b)(s),b^{(L)},G_o)|}{\epsilon}\,ds\\ \leq  1 + V\int_0^t\,(\Delta_\epsilon(s)+1)\,ds.\end{multline*}
Therefore,
\[\Delta_\epsilon(t)\leq 1 + V\int_0^t\,(\Delta_\epsilon(s)+1)\,ds.\]
A simple Gronwall argument reveals:
\[\Delta_\epsilon(t)\leq 2e^{Vt}-1\leq e^{2Vt}.\]\end{proof}

\subsection{Lemmata on moment-generating function}

\begin{lemma}
\label{lem:MGF}
Suppose $X_1,\dots,X_m$ are independent random variables that are uniformly bounded by $c>0$. Then for any $\xi\in \R$:
\[
	\log \Ex{\exp\left[\xi \frac{1}{m}\sum_{i=1}^m\,(X_i - \Ex{X_i})\right]}\leq  \frac{c^2\xi^2}{2m}.
\]
\end{lemma}

\begin{proof}
Hoeffding's Lemma implies that, since
\[
	-\frac{2c}{m}\leq \frac{X_i - \Ex{X_i}}{m} \leq \frac{2c}{m},
\]
we have
\[
	\log \Ex{\exp\left[\xi \,\frac{(X_i - \Ex{X_i})}{m}\right]}\leq \frac{c^2}{2m^2}.
\]
Applying this to each term in
\[
	\log \Ex{\exp\left[\xi \left(\frac{1}{m}\sum_{i=1}^m\,(X_i - \Ex{X_i})\right)\right]} = \sum_{i=1}^m\log \Ex{\exp\left[\xi \,\frac{(X_i - \Ex{X_i})}{m}\right]}
\]
finishes the proof.
\end{proof}

\begin{lemma}[MGF for vectors] Suppose we have a sum of $m$ $d$-dimensional random vectors $Y_1,\dots,Y_m$. Assume $c>0$ is such that $|Y_i|\leq c$ a.s. for all $1\leq i\leq m$. Then
\[
	\log \Ex{\exp\left[\xi \left|\frac{1}{m}\sum_{i=1}^m\,(Y_i - \Ex{Y_i})\right|\right]}\leq d\log 5 + \frac{2c^2\xi^2}{m}.
\]
\end{lemma}
\begin{proof}Let $\sN$ be an $1/2$-net of the unit sphere in $\R^d$ of cardinality $\leq 5^d$ (such nets exist by a simple volumetric argument). It is well-known that for any vector $v\in\R^d$, $|v|\leq 2\max_{w\in\sN}\langle w,v\rangle$. Therefore,
\[\Ex{e^{\xi \left|\frac{1}{m}\sum_{i=1}^m\,(Y_i - \Ex{Y_i})\right|}}\leq \sum_{w \in \sN}\Ex{e^{\xi \frac{1}{m}\sum_{i=1}^m\,2\langle w,Y_i - \Ex{Y_i}\rangle}}
\]
Each MGF in the RHS corresponds to taking $X_i=2\langle w,Y_i\rangle$ in the previous Lemma, with $2c$ replacing $c$. Applying the bound from that result and the fact that $|\sN|\leq 5^d$ finishes the proof.\end{proof}

\begin{lemma}[MGF for martingales]
\label{lem:azuma}
	Let $(Y_k)_{k\geq0}$ be a $d$-dimensional martingale with respect to the filtration $(\sF_k)_{k\geq0}$ with $Y_0 = 0$. Assume $c>0$ is such that $|Y_k - Y_{k-1}|\leq c$ a.s. for all $k>0$. Then,
	\[
		\log \Ex{e^{\xi |Y_m|}}\leq d\log 5 + 2\,m\,c^2\xi^2.
	\]
	We also have,
	\[
		\Pr{\max_{k\leq m} |Y_k| \geq 2\,c\,\sqrt{2\,m\,(u^2 + d\,\log{5})}} \leq e^{-u^2}
	\]
\end{lemma}
\begin{proof}
	To work in $\R^d$, we set $\sN$ to be an $1/2$-net just as in the proof of the previous lemma. As before, we have
	\[
		\Ex{e^{\xi |Y_m|}}\leq \sum_{w \in \sN}\Ex{e^{\xi \,2\langle w,Y_m\rangle}}.
	\]
	A simple modification of Hoeffding's Lemma implies that for $k> 0$,
	\[
		\Ex{e^{\xi \,2\langle w,Y_k - Y_{k-1}\rangle} \mid \sF_{k-1}} \leq e^{2\,(\xi\,c)^2}.
	\]
	Then, applying this bound recursively gives
	\begin{eqnarray*}
		\Ex{e^{\xi \,2\langle w,Y_m\rangle}} &=& \Ex{\Ex{e^{\xi \,2\langle w,(Y_m - Y_{m-1})\rangle} \mid \sF_{m-1}}\, e^{\xi \,2\langle w,Y_{m-1}\rangle}}\\
		&\leq& e^{2\,(\xi\,c)^2}\,\Ex{e^{\xi \,2\langle w,Y_{m-1}\rangle}} \leq e^{2\,m\,(\xi\,c)^2}.
	\end{eqnarray*}
	Using the fact that $|\sN|\leq 5^d$ finishes the proof.
\end{proof}

To use the upper bounds in probability to get upper bound in expected value, we use the following elementary lemma:

\begin{lemma}
\label{lem:expected-upper-bound}
Let $Z$ be a positive random variable for which we know that, for some $a,b>0$,
$\quad Z \leq a + b\cdot t\quad$
holds with probability greater than $1 - e^{-t^2}$. 

Then $\Ex{Z}\leq a + b\cdot \frac{\sqrt{\pi}}{2}$
\end{lemma}

%% file: measureNN_aux/A2-lipshitz.tex
In this section, we present a collection of Lipschitz estimates for several functions used throughout many of our results. These estimates are specially usefull to work through the many nonlinearities in our model. We start by setting a framework that make computations for our results cleaner.

Since many of our maps act on measure spaces, for instance the ones in Section \ref{sub:McKeanVlasov} (the mean-field versions of the model components), is necessary to set a metric for these spaces. We always use the Wasserstein $L^1$ metric. Let us briefly review the definition of this metric:

\begin{definition}[Wasserstein $L^1$ metric]
Given a Polish space $(M,\rho)$, and two probability measures $\mathbb{P}, \Q \in \Probspace(M)$. We define
\begin{equation}\label{def:W1}
  W(\mathbb{P},\Q) = \inf
  \left\{\int_{M\times M} \rho(\theta, \gamma) d\nu(\theta, \gamma);
  \quad \nu\in\Coup(\mathbb{P},\Q),
  \right\}
\end{equation}
where $\Coup(\mathbb{P},\Q)$ denotes the set of couplings of $\mathbb{P}$ and $\Q$. Recall that $\nu\in\Coup(\mathbb{P},\mathbb{Q})$ whenever $\nu$ is a probability measure over $M\times M$ satisfying
\begin{equation}
	\nu(A\times M) = \mathbb{P}(A)\quad \nu(M\times A) = \Q(A)
	\quad\quad \text{ for all }A\in\sB(M).
\end{equation}
\end{definition}

In our setting, we will need to consider two different Polish spaces $M$: $\R^D$ and $C([0,T],\R^D)$

\begin{remark}
It is a standard result in Optimal Transport that, when our probability measures are defined over a Polish space, then there actually exists a coupling that achieves the infimum in the definition. In our setting, this means two things. Given $\mu,\nu$ measures in $\R^D$ there exists $\eta\in\Coup(\mu,\nu)$ so that
\[
W(\mu,\nu) = \int_{\R^D\times\R^D} \|a - b\|d\eta(a,b).
\]
Also, given probability measures $\mu_{[0,T]}$ and $\nu_{[0,T]}$ over $C([0,T],\R^D)$, there exists $\eta_{[0,T]}\in\Coup(\mu_{[0,T]},\nu_{[0,T]})$ such that:
\[
W(\mu_{[0,T]},\nu_{[0,T]}) = \int_{C([0,T],\R^D)\times C([0,T],\R^D)} \left(\sup_{u\in[0,T]}\abs{\theta(u) - \gamma(u)}\right)\,d\eta_{[0,T]}(\theta,\gamma).
\]
\end{remark}

\begin{remark}
  For more readable computations, we also denote the integral of a function $f$ over some measure $\mu$ as
  \[\mu\left[ f(\theta) \right] = \int f(\theta) d\mu(\theta).\]
\end{remark}

\begin{remark}
\label{lem:timemeasureupperb}
Let $\mu_{[0,T]},\nu_{[0,T]}$ be measures over $C([0,T],\R^D)$ and fix $t<T$. We consider the following marginals: the marginal  'up to time t', i.e., $\mu_{[0,t]},\nu_{[0,t]}$ measures over $C([0,t],\R^D)$, and the marginal 'at time t', i.e., $\mu_t,\nu_t$ measures over $\R^D$. Then the Wasserstein distance between these objects satisfy the following inequality:
\[W\parn{\mu_{t},\nu_{t}}\leq W\parn{\mu_{[0,t]},\nu_{[0,t]}}\leq W\parn{\mu_{[0,T]},\nu_{[0,T]}}\]
\end{remark}

\subsection{General upper bounds on Lipschitz semi-norm}
Here we present some auxiliary results using the concept of Lipschitz semi-norm. These summarize typical computations, making our subsequent results more concise.

\begin{definition}[Lipschitz semi-norm]
Let $(M_{in},d_{in}), (M_{out},d_{out})$ be two metric spaces, and let $f:M_{in}\to M_{out}$ be a Lipschitz function. We define
\begin{equation}
  \lipnorm{f} = \sup
  \left\{\frac{d_{out}(f(u), f(v))}{d_{in}(u,v)}
  ;\quad u,v\in M_{in},\,\,u\neq v
  \right\}.
\end{equation}
Notice that this quantity is finite iff $f$ is Lipschitz. Also, if $V\geq\lipnorm{f}$ then
\[
 d_{out}(f(u),f(v)) \leq V \cdot d_{in}(u,v)
\]
\end{definition}

The following two lemmas are well known results regarding the Lipschitz semi-norm; we omit their proofs.

\begin{lemma}\label{lem:sum-prod-lipschitz}
  Let $f,g: M \rightarrow \R^{n_{out}}$ bounded Lipschitz functions, $M$ metric space. Then,
  \begin{enumerate}
    \item $\lipnorm{f \cdot g} \leq \supnorm{f}\cdot\lipnorm{g} + \supnorm{g}\cdot \lipnorm{f}$,
    \item $\lipnorm{f+g} \leq \lipnorm{f} + \lipnorm{g}$.
  \end{enumerate}
\end{lemma}

\begin{lemma}\label{lem:compos-lipschitz}
  For Lipschitz functions $f: M_1 \rightarrow M_2$, $g: M_2 \rightarrow M_3$, we have
  \begin{equation}
    \lipnorm{f\circ g} \leq \lipnorm{f}\cdot\lipnorm{g}
  \end{equation}
\end{lemma}

This next lemma is key in our estimates.

\begin{lemma}\label{lem:mean-lipschitz}
Consider the metric space $(\R^{n_{in}} \times \sM \times \R^n, \rho)$, where $\sM\subset\Probspace(\R^n)$, and $\rho$ is defined as:
  \[
  	\rho\Big((x,\mu,a), (y,\nu,b)\Big) = \norm{x-y} + W(\mu,\nu) + \norm{a-b}
  \]
Now, let $f:\R^{n_{in}} \times \sM \times \R^n \rightarrow \R^{n_{out}}$ be a Lipschitz function, such that for any fixed $\mu\in\sM$, we have:
\[
F(x,\mu):= \int_{\R^n} f(x, \mu, a) d\mu(a) < \infty.
\]
Then, by considering the metric space $(\R^{n_{in}}\times\sM,\tilde\rho)$ with metric
\[
\tilde\rho\Big((x,\mu),(y,\nu)\Big) = \norm{x-y} + W(\mu,\nu)
\]
we have that the function $F:\R^{n_{in}}\times\sM\to\R^{n_{out}}$ defined above is also Lipschitz, with $\lipnorm{F}\leq 2\lipnorm{f}$.
\end{lemma}

\begin{proof}
Let $\eta$ be the optimal coupling between $(\mu,\nu)$. Then:
  \begin{align}
    \lipnorm{F} & = \sup_{(x,\mu)\neq (y, \nu)}
    \left\{\frac{\norm{F(x,\mu) - F(y, \nu)}}{\norm{x - y} + W(\mu, \nu)}
    \right\}\\
    & \leq \sup_{(x,\mu)\neq (y, \nu)}
    \left\{\eta\left[\frac{\norm{f(x,\mu,a) - f(y,\nu,b)}}{\norm{x - \tilde{x}} + W(\mu, \nu)}\right]\right\}\label{step:triangular}\\
    & \leq \sup_{(x,\mu)\neq (y, \nu)}
    \left\{\lipnorm{f}\left( 1 + \eta\left[\frac{\norm{a-b}}{\norm{x - y} + W(\mu, \nu)}\right] \right)\right\}\label{step:lipschitz}\\
    & \leq \sup_{(x,\mu)\neq (y, \nu)} \left\{\lipnorm{f}\left( 1 + \frac{W(\mu,\nu)}{\norm{x - y} + W(\mu, \nu)} \right)\right\} \leq 2\,\lipnorm{f},\label{step:upperbound}
\end{align}
 where step \ref{step:triangular} is simply the triangular inequality for integrals, \ref{step:lipschitz} uses the fact that $f$ is Lipschitz, and \ref{step:upperbound} is an obvious upper bound.
\end{proof}

\subsection{Lipschitz bounds for the mean-field maps}
We now proceed to the Lipschitz estimates for the maps described in section \ref{sub:McKeanVlasov}. We assume that conditions \ref{assump:initialization} and \ref{assump:activations} hold.

\begin{lemma}\label{lem:zbar-lipschitz}
  The maps $\{\zbar^{(\ell)},\ell\in[2:L+1]\}$ are bounded by C, and for $\ell \in [2:L+1]\setminus\{L\}$ we have
  \begin{equation}\label{eq:zbar-order}
    \lipnorm{\zbar^{\ell}} = C_{\ref{eq:zbar-order}}(\ell),
  \end{equation}
  where $C_{\ref{eq:zbar-order}}(\ell)=\bigoh{C^\ell}$
\end{lemma}
\begin{proof}
  The boundness of $\zbar^{(\ell)}$ follow from the assumption of boundedness of $\sigma^{(\ell-1)}$ in \ref{assump:activations}.

 To obtain the Lipschitz upper bound, we prove (\ref{eq:zbar-order}) by induction up until $L-1$, and then work the case $\zbar^{(L+1)}$ separately.

  For $\ell=2$, we apply lemmas \ref{lem:mean-lipschitz} and \ref{lem:compos-lipschitz} and obtain
  $$ \lipnorm{\zbar^{(2)}} \leq 2 \cdot \lipnorm{\sigma^{(1)}\circ \sigma^{(0)}} \leq 2 \cdot\lipnorm{\sigma^{(1)}}\lipnorm{\sigma^{(0)}} = 2\,C^2=:C_{\ref{eq:zbar-order}}(2) $$

  When (\ref{eq:zbar-order}) holds for $\ell$, we observe that
  $$ \lipnorm{\zbar^{(\ell+1)}} \leq 2 \,\lipnorm{\sigma^{(\ell)}} \lipnorm{\zbar^{\ell}} \leq 2 C\cdot C_{\ref{eq:zbar-order}}(\ell) =: C_{\ref{eq:zbar-order}}(\ell+1),$$
  where clearly $ C_{\ref{eq:zbar-order}}(\ell+1) = \bigoh{C^{\ell+1}}.$

This argument works up until $L-1$. It fails for $\zbar^{(L)}$ for it is defined differently. Nevertheless, we can work through it and prove that the result holds for $\zbar^{(L+1)}$ as well. Let $\eta$ be the optimal coupling between $(\mu,\nu)$ and notice that:
\begin{align}\label{eq:mean-zbar-L}
    \nonumber
    \eta\left[\norm{\zbar^{(L)}(x, \mu, a^{(L)}) - \zbar^{(L)}(y, \nu, b^{(L)})}\right]
    \leq&
    C\eta\left[\eta\left[\norm{\zbar^{(L-1)}(x,\mu) -
     \zbar^{(L-1)}(y,\nu)}\right.\right.\\
    \,&
    + \left.\left.\norm{a^{(L-1)} - b^{(L-1)}} \right|\left. a^{(L)}, b^{(L)}\right] \right]\\
      \nonumber \leq& C\lipnorm{\zbar^{(L-1)}} (\|x - y\| + W(\mu, \nu))
  \end{align}

  Then, we can obtain the last estimative
  \begin{align*}
    \lipnorm{\zbar^{(L+1)}} &\leq \sup_{(x,\mu)\neq (y, \nu)}\left\{ C\, \eta\left[ \frac{\norm{\zbar^{(L)}(x, \mu, a^{(L)}) -
     \zbar^{(L)}(\tilde{x}, \nu, b^{(L)})} + \norm{a^{(L)} - b^{(L)}}}{\|x - y\| + W(\mu, \nu)} \right]\right\}\\
     &\leq C (1 + C\lipnorm{\zbar^{(L-1)}}) =:  C_{\ref{eq:zbar-order}}(L+1,)
  \end{align*}
  showing $C_{\ref{eq:zbar-order}}(L+1)=\bigoh{C^{L+1}}$.
\end{proof}

\begin{lemma}\label{lem:Mbar-lipschitz}
  The collection of maps $\{\Mbar^{(\ell)}, \ell\in[2:L+1]\}$ and $\{\gbar^{(\ell)}, \ell\in[1:L-1]\}$ are bounded and Lipschitz as described below
  \begin{enumerate}
    \item For $\ell\in[2:{L+1}]\backslash\{L\}$
    $$\supnorm{\Mbar^{(\ell)}} \leq C^{L+2-\ell}, \quad\quad  \lipnorm{\Mbar^{(\ell)}} = \bigoh{C^{2(L+1)-\ell+1}}.$$
    \item For $\ell\in[1:{L-1}]$
    $$\supnorm{\gbar^{(\ell)}} \leq C^{L+2-\ell},$$
    and for $\ell\in[1:{L-2}]$
    $$\lipnorm{\gbar^{(\ell)}} = \bigoh{C^{2(L+1)-\ell+1}}.$$
  \end{enumerate}
\end{lemma}

\begin{proof}
  The boundness property of item (1) also comes directly from the assumption \ref{assump:activations}, since the maps $\Mbar$ and $\gbar$ are products of derivative terms
  \begin{equation*}
      \supnorm{\Mbar^{(\ell)}} \leq \prod_{k=\ell}^{L+1} \supnorm{D_z \sigma^{(k)}} \leq C^{L+2-\ell},
  \end{equation*}
  \begin{equation*}
    \supnorm{\gbar^{(\ell)}} \leq \prod_{k=\ell+1}^{L+1} \supnorm{D_z \sigma^{(k)}} \supnorm{D_\theta\sigma^{(\ell)}} \leq C^{L+2-\ell}.
  \end{equation*}

  For the Lipschitz estimate we do induction in the inverse order.

  For $\ell = L+1$, we have
    $$\lipnorm{\Mbar^{(L+1)}} \leq \lipnorm{D\sigma^{(L+1)} \circ \zbar^{(L+1)}} \leq C\, \lipnorm{\zbar^{(L+1)}} = \bigoh{C^{L+2}}.$$
  When item (1) holds for $\ell+1$ we obtain
  \begin{align*}
    \lipnorm{\Mbar^{(\ell)}} &\leq \supnorm{\Mbar^{(\ell+1)}} \, 2 \lipnorm{D_z\sigma^{(\ell)} \circ \zbar^{(\ell)}} +
                                  \lipnorm{\Mbar^{(\ell+1)}} \supnorm{D_z\sigma^{(\ell)}}\\
                            &\leq C^{L+1-\ell}\, 2 C\, \lipnorm{\zbar^{(\ell)}} + C\, \lipnorm{\Mbar^{(\ell+1)}}\\
                            &\leq \bigoh{C^{L+2}} + C \, \bigoh{C^{2(L+1)-\ell}} = \bigoh{C^{2(L+1)-\ell+1}}.
  \end{align*}
  For the special case of $\ell = L$, the term $\mu\left[\Mbar(x,\mu,\theta^{L})\right]$ can be shown to be Lipschitz in $(x,\mu)$ just as we did in (\ref{eq:mean-zbar-L}).

  For the maps $\gbar$ we have
  $$\lipnorm{\gbar^{(1)}} \leq \lipnorm{\Mbar^{(2)} \cdot (D_\theta \sigma^{(1)} \circ \sigma^{(0)})} \leq C^L\, C^2 + C\,\lipnorm{\Mbar^{(2)}} = \bigoh{C^{2(L+1)}},$$
  and for $\ell\in[2:L-2]$
  \begin{align*}
    \lipnorm{\gbar^{(\ell)}} &\leq \lipnorm{\Mbar^{(\ell+1)} \cdot (D_\theta \sigma^{(\ell)} \circ \zbar^{(\ell)})}\\
                               &\leq C^{L+1-\ell}\, C \,\lipnorm{\zbar^{(\ell)}} + C\,\lipnorm{\Mbar^{(\ell+1)}}\\
                               & = C^{L+2-\ell} \,\bigoh{C^\ell} + C\,\bigoh{C^{2(L+1)-\ell}} = \bigoh{C^{2(L+1)-\ell+1}}
  \end{align*}
\end{proof}

\subsection{Lipschitz bounds for the existence of McKean-Vlasov process}

In this subsection, we isolate the Lipschitz computations which will be needed when proving the existence of the McKean-Vlasov process. In order to better present these computations, we define the following function:
\[
\Grad^{(\ell)}(\mu,a):=(Y-\ybar(X,\mu))\gbar^{(\ell)}(X,\mu,a)
\]

\begin{lemma}\label{lem:unif-lip-gen-l}
Fix $(X,Y)\in \R^{d_{in}}\times\R^{d_{out}}$. Then, for $\ell\in[2:L-2]$ the maps $\Grad^{(\ell)}$ are Lipschitz, uniformly in $(X,Y)$.
\end{lemma}

\begin{proof}
Follows directly from the fact that $\ybar$ and $\{\gbar^{(\ell)},\ell\in[2:L-2]\}$ are Bounded and Lipschitz  (Lemma \ref{lem:Mbar-lipschitz}). Then result follows from Lemma \ref{lem:sum-prod-lipschitz}.
\end{proof}

The only gradient missing in previous Lemma is $\gbar^{(L-1)}$. It needs a different analysis for it has a different structure than the others. Although it may not necessarily be Lipschitz, we can prove the following upper bound, which turns out to be good enough for our purposes:

\begin{lemma}\label{lem:unif-lip-L-1}
Fix $(X,Y)\in\R^{d_{in}}\times\R^{d_{out}}$, where $Y$ is assumed to be Bounded. Then, there exists a constant $V_{\ref{lem:unif-lip-L-1}}>0$ such that:
\begin{align*}
\abs{\Grad^{(L-1)}(\mu,a) - \Grad^{(L-1)}(\nu,b)} \leq V_{\ref{lem:unif-lip-L-1}} \Big(W(\mu,\nu)  + |a^{(L-1)} - b^{(L-1)}| +\\
\abs{\zbar^{(L)}(X,\mu,a^{(L)}) - \zbar^{L}(X,\nu,b^{(L)})}\Big)
\end{align*}
\end{lemma}

\begin{proof}
We start by making the following computations:
\begin{align}
\abs{\Grad^{(L-1)}(\mu,a) - \Grad^{(L-1)}(\nu,b)}&\leq \supnorm{\gbar^{(L-1)}}\abs{\ybar(X,\mu) - \ybar(X,\nu)} +\nonumber\\
&\supnorm{Y - \ybar(X,\mu)}\abs{\gbar^{(L-1)}(X,\mu,a) - \gbar^{(L-1)}(X,\nu,b)}\\
&\leq V_1\cdot \Big(W(\mu,\nu) + \abs{\gbar^{(L-1)}(X,\mu,a) - \gbar^{(L-1)}(X,\nu,b)}\Big),\nonumber
\end{align}
where $V_1 \geq \max\left\{\supnorm{\gbar^{(L-1)}}\cdot\lipnorm{\ybar}\,,\,\supnorm{Y}+\supnorm{\ybar}\right\}$.

Now, it suffices to show that, for some constant $V_2$, it holds that:
\begin{align*}
\abs{\gbar^{(L-1)}(X,\mu,a) - \gbar^{(L-1)}(X,\nu,b)}\leq V_2 \Big(W(\mu,\nu) + |a^{(L-1)} - b^{(L-1)}| + \\
\abs{\zbar^{(L)}(X,\mu,a^{(L)}) - \zbar^{L}(X,\nu,b^{(L)})}\Big)
\end{align*}

By recalling definition (\ref{eq:defgbarL-1}), using the upper bounds from previous section and some seriously tedious calculations, one may show that:
\begin{align*}
&\abs{\gbar^{(L-1)}(X,\mu,a) - \gbar^{(L-1)}(X,\nu,b)}\leq \\
&\quad\quad\quad\supnorm{\Mbar^{(L)}}\lipnorm{D_\theta\sigma^{(L-1)}}\Big(\lipnorm{\zbar^{(L-1)}}\cdot W(\mu,\nu) + \abs{a^{(L-1)} - b^{(L-1)}}\Big) +\\
&\quad\quad\quad + \supnorm{D_\theta\sigma^{(L-1)}}\Big(\supnorm{\Mbar^{(L+1)}}\cdot\lipnorm{D_z\sigma^{(L)}}\abs{\zbar^{(L)}(X,\mu,a^{(L)}) - \zbar^{(L)}(X,\nu,b^{(L)})}+\quad\quad\\
&\quad\quad\quad + \supnorm{D_z\sigma^{(L)}}\lipnorm{\Mbar^{(L+1)}}\cdot W(\mu,\nu)\Big)\quad\quad\quad\quad\\
&\leq V_2 \cdot \parn{W(\mu,\nu) + \abs{a^{(L-1)}-b^{(L-1)}} + \abs{\zbar^{(L)}(X,\mu,a^{(L)}) - \zbar^{(L)}(X,\nu,b^{(L)})}},
\end{align*}
where
\begin{equation}
V_2 \geq \max\left\{
\begin{array}{c}
\supnorm{D_\theta\sigma^{(L-1)}}\supnorm{\Mbar^{(L+1)}}\lipnorm{D_z\sigma^{(L)}},\\ \supnorm{D_\theta\sigma^{(L-1)}}\supnorm{D_z\sigma^{(L)}}\lipnorm{\Mbar^{(L+1)}},\\ \supnorm{\Mbar^{(L)}}\lipnorm{D_\theta\sigma^{(L-1)}} \lipnorm{\zbar^{(L-1)}},\\ \supnorm{\Mbar^{(L)}}\lipnorm{D_\theta\sigma^{(L-1)}}
\end{array}
\right\}.
\end{equation}
\end{proof}

\subsection{Lipschitz bounds for DNN maps}\label{sec:lipschitz-DNN}
We also present Lipschitz estimates for the maps defined in section \ref{sec:setup}. The estimates and the proofs are similar to the ones obtained for their mean-field versions.

\begin{lemma}For $\ell \in [1:L+1]$ and any choice of $i_\ell \in [1:N_\ell]$
    $$\supnorm{z_{i_\ell}^{(\ell)}} \leq C, \quad\quad  \lipnorm{z_{i_\ell}^{(\ell)}} = \bigoh{C^{\ell}}.$$
\end{lemma}
\begin{proof}Boundness again comes up from $\supnorm{\sigma^{(\ell)}}\leq C$.

  For the Lipschitz estimate, we apply induction on $\ell$. The case $\ell=1$ is clear. If the result holds for $\ell$ and any $i_\ell \in [1:N_\ell]$. Then, for $i_{\ell+1} \in [1:N_{\ell+1}]$ we can show that
  $$\lipnorm{z_{i_{\ell+1}}^{(\ell+1)}} \leq \frac{1}{N_\ell}\sum_{i_\ell=1}^{N_\ell} \lipnorm{\sigma^{(\ell)}}\,\parn{\lipnorm{z_{i_\ell}^{(\ell)}}+1} \leq \bigoh{C^{\ell+1}}.$$
\end{proof}

\begin{remark}
\label{rem:bounded-yhat}
  The previous lemma also implies that $\yhat$ is bounded and Lipschitz,
  $$\supnorm{\yhat}\leq C ,\quad\quad \lipnorm{\yhat} \leq \bigoh{C^{L+2}}.$$
\end{remark}

\begin{lemma}For $\ell \in [1:L+1]$ and indices $\boldj{\ell}{L+1} \in [1:\boldN{\ell}{L+1}]$
    $$\supnorm{M_{\boldj{\ell}{L+1}}^{(\ell)}} \leq C^{L+2-\ell}, \quad\quad  \lipnorm{M_{\boldj{\ell}{L+1}}^{(\ell)}} = \bigoh{C^{2(L+1)-\ell+1}}.$$
\end{lemma}
\begin{proof}
  For boundness we just notice that
  \begin{equation*}
    \supnorm{M_{\boldj{\ell}{L+1}}^{(\ell)}} \leq \prod_{k=\ell}^{L+1} \supnorm{D_z \sigma^{(k)}} \leq C^{L+2-\ell}.
  \end{equation*}
  And for the Lipschitz property we estimate using induction,
  \[\lipnorm{M_1^{(L+1)}} \leq \lipnorm{D\sigma^{(L+1)}}\,\lipnorm{z^{(L+1)}} = C\,\bigoh{C^{L+1}} = \bigoh{C^{L+2}}.  \]
  If the estimate holds for $\ell+1$ and any sequence $\boldj{\ell+1}{L+1} \in [1:\boldN{\ell+1}{L+1}]$.
  Then,
  \begin{eqnarray*}
    \lipnorm{M_{\boldj{\ell}{L+1}}^{(\ell)}} &\leq& \lipnorm{M_{\boldj{\ell+1}{L+1}}^{(\ell+1)}}\,
    \supnorm{D_z \sigma^{(\ell)}} + \supnorm{M_{\boldj{\ell+1}{L+1}}^{(\ell+1)}}\,
    \lipnorm{D_z \sigma^{(\ell)}}(\lipnorm{z_{j_\ell}^\ell} + 1)\\
    &\leq& \bigoh{C^{2(L+1)-\ell}}\,C + C^{L+1-\ell}\,C\,(\bigoh{C^\ell} + 1) = \bigoh{C^{2(L+1)-\ell+1}}
  \end{eqnarray*}
\end{proof}

\begin{remark}
\label{rem:lipschitz-y_N}
  Using that the average of Lipschitz functions of same order is still Lipschitz, we can see that the gradients obtained through backpropagation are also Lipschitz, and bounded.
  $$\supnorm{N^2\,\frac{\partial \yhat}{\partial \theta_{i_\ell,i_{\ell+1}}^{(\ell)}}}\leq C^{L+2-\ell}\quad\quad
  \lipnorm{N^2\,\frac{\partial \yhat}{\partial \theta_{i_\ell,i_{\ell+1}}^{(\ell)}}} \leq \bigoh{C^{2(L+1)-\ell+1}}.$$
\end{remark}

%
%

Our last lemma for this section deals with the special norm $\Lnorm{\cdot}$,
\begin{equation}
\label{eq:normParam}
\Lnorm{\btheta_N} = \sup_{\ell \in [0:L]}\left\{\frac{1}{N_\ell\cdot N_{\ell+1}}\sum_{i_{\ell+1} = 1}^{N_{\ell+1}}\sum_{i_\ell = 1}^{N_\ell}\abs{\theta^{(\ell)}_{i_\ell,i_{\ell+1}}}\right\},
\end{equation}
which is used in section \ref{sec:SGDtoCTGD}. Unfortunately, the Lipschitz property for this case do not quite follows the results obtained for the $L^1$ norm. We need to exploit the averaging structure of our functions, as the norm itself does it.  In the next lemma, we omit the dependency on the variables $X$ and $Y$. We remark that the estimates are also uniform in these entries, which can be proved using the same strategies used during this section.

\begin{lemma}[$\Gradh_N$ is Lipschitz with $\Lnorm{\cdot}$]\label{lem:grad-lnorm}
  For $a_N, b_N \in \R^D$ weight vectors, there exists a constant $K$ depending only on C and L, such that
  \begin{equation}
    \Lnorm{\Gradh_N(a_N) - \Gradh_N(b_N)} \leq K\;\Lnorm{a_N - b_N}.
  \end{equation}
\end{lemma}
\begin{proof}
  First, we state a few inequalities for the $z^{(\ell)}_{i_\ell}$ that can be obtained recursively:
  \begin{equation}
    \abs{z^{(1)}_{i_1}(a_N) - z^{(1)}_{i_1}(b_N)} \leq C\;\abs{a^{(1)}_{1,i_1} - b^{(1)}_{1,i_1}};
  \end{equation}
  for $\ell\in[2:L+1]$ and $i_\ell\in [1:N_\ell]$ we have
  \begin{align}
    \nonumber \abs{z^{(\ell)}_{i_\ell}(a_N) - z^{(\ell)}_{i_\ell}(b_N)} \leq& \frac{1}{N_{\ell-1}} \sum_{i_{\ell-1} = 1}^{N_{\ell-1}} C\parn{\abs{z^{(\ell-1)}_{i_{\ell-1}}(a_N) - z^{(\ell-1)}_{i_{\ell-1}}(b_N)} + \abs{a^{(\ell-1)}_{i_{\ell-1},i_\ell} - b^{(\ell-1)}_{i_{\ell-1},i_\ell}}}\\
    \leq& \sum_{k=0}^{\ell-2} \frac{C^{\ell-k}}{N_k N_{k+1}} \sum_{i_k=1}^{N_k} \sum_{i_{k+1}=1}^{N_{k+1}} \abs{a^{(k)}_{i_{k},i_{k+1}} - b^{(k)}_{i_{k},i_{k+1}}} + \frac{C}{N_{\ell-1}} \sum_{i_{\ell-1} = 1}^{N_{\ell-1}} \abs{a^{(\ell-1)}_{i_{\ell-1},i_\ell} - b^{(\ell-1)}_{i_{\ell-1},i_\ell}}\\
    \leq& \ell\;C^\ell \Lnorm{a_N - b_N} + \frac{C}{N_{\ell-1}} \sum_{i_{\ell-1} = 1}^{N_{\ell-1}} \abs{a^{(\ell-1)}_{i_{\ell-1},i_\ell} - b^{(\ell-1)}_{i_{\ell-1},i_\ell}}.
  \end{align}

  These also implies
  \begin{equation}\label{eq:lnorm-yhat}
    \abs{\yhat(a_N) - \yhat(b_N)} \leq L\;C^{L+2}\Lnorm{a_N - b_N}.
  \end{equation}

  Next, we observer the behavior of the gradient terms:
  \begin{align}\label{eq:compare-grad}
    \nonumber\abs{N^2\;\frac{\partial \yhat}{\partial \theta_{i_\ell,i_{\ell+1}}^{(\ell)}}(a_N) - N^2\;\frac{\partial \yhat}{\partial \theta_{i_\ell,i_{\ell+1}}^{(\ell)}}(b_N)} \leq& \frac{C}{\prod_{k={\ell+2}}^{L+1}N_k} \sum\limits_{\boldj{\ell+2}{L+1}\in[1:\boldN{\ell+2}{L+1}]} \abs{ M^{(\ell+1)}_{(i_{\ell+1},\boldj{\ell+2}{L+1})}(a_N) -
     M^{(\ell+1)}_{(i_{\ell+1},\boldj{\ell+2}{L+1})}(b_N)} +\\
    & + C^{L+1-\ell} \parn{ \abs{z^{(\ell)}_{i_\ell}(a_N) - z^{(\ell)}_{i_\ell}(b_N)} + \abs{a^{(\ell)}_{i_{\ell},i_{\ell+1}} - b^{(\ell)}_{i_{\ell},i_{\ell+1}} }}.
  \end{align}
  We first take care of the first part of the RHS:
  \begin{align}
    \nonumber\abs{ M^{(\ell)}_{\boldj{\ell}{L+1}}(a_N) - M^{(\ell)}_{\boldj{\ell}{L+1}}(b_N)} \leq& C\abs{ M^{(\ell+1)}_{\boldj{\ell+1}{L+1}}(a_N) - M^{(\ell+1)}_{\boldj{\ell+1}{L+1}}(b_N)} +\\
    &+\; C^{L+2-\ell} \parn{ \abs{z^{(\ell)}_{j_\ell}(a_N) - z^{(\ell)}_{j_\ell}(b_N)} + \abs{a^{(\ell)}_{j_{\ell},j_{\ell+1}} - b^{(\ell)}_{j_{\ell},j_{\ell+1}} }}\\
    \leq& C^{L+2-\ell} \parn{ \sum_{k=\ell}^{L+1} \abs{z^{(k)}_{j_k}(a_N) - z^{(k)}_{j_k}(b_N)} +  \sum_{k=\ell}^{L} \abs{a^{(k)}_{j_{k},j_{k+1}} - b^{(k)}_{j_{k},j_{k+1}} } }\\
    \leq& \tilde{C} \Lnorm{a_N - b_N} + \sum_{k=\ell}^{L+1} \frac{1}{N_{k-1}} \sum_{i_{k-1}=1}^{N_{k-1}} \abs{a^{(k-1)}_{i_{k-1},j_{k}} - b^{(k-1)}_{i_{k-1},j_{k}}}
  \end{align}
  Now, replacing the above inequality and the ones for $z^{(\ell)}_{i_\ell}$ in (\ref{eq:compare-grad}) we obtain
  \begin{align}
    \nonumber\abs{N^2\;\frac{\partial \yhat}{\partial \theta_{i_\ell,i_{\ell+1}}^{(\ell)}}(a_N) - N^2\;\frac{\partial \yhat}{\partial \theta_{i_\ell,i_{\ell+1}}^{(\ell)}}(b_N)} \leq&
    K_1\Lnorm{a_N - b_N} + \frac{K_2}{N_{\ell-1}} \sum_{j_{\ell-1}=1}^{N_{\ell-1}} \abs{a^{(\ell-1)}_{j_{\ell-1},i_{\ell}} - b^{(\ell-1)}_{j_{\ell-1},i_{\ell}}} +\\
    &+\; K_3 \abs{a^{(\ell)}_{i_{\ell},i_{\ell+1}} - b^{(\ell)}_{i_{\ell},i_{\ell+1}}} + \frac{K_4}{N_{\ell}} \sum_{j_{\ell}=1}^{N_{\ell}} \abs{a^{(\ell)}_{j_{\ell},i_{\ell+1}} -
    b^{(\ell)}_{j_{\ell},i_{\ell+1}}}
  \end{align}
  If we combine this estimate with the one in (\ref{eq:lnorm-yhat}), then we can update the constants $K_1$ to $K_4$ obtaining
  \begin{align}
    \nonumber \abs{\Gradh^{(\ell)}_{i_\ell,i_{\ell+1}}(a_N) - \Gradh^{(\ell)}_{i_\ell,i_{\ell+1}}(b_N)} \leq&
    K_1\Lnorm{a_N - b_N} + \frac{K_2}{N_{\ell-1}} \sum_{j_{\ell-1}=1}^{N_{\ell-1}} \abs{a^{(\ell-1)}_{j_{\ell-1},i_{\ell}} - b^{(\ell-1)}_{j_{\ell-1},i_{\ell}}} +\\
    &+\; K_3 \abs{a^{(\ell)}_{i_{\ell},i_{\ell+1}} - b^{(\ell)}_{i_{\ell},i_{\ell+1}}} + \frac{K_4}{N_{\ell}} \sum_{j_{\ell}=1}^{N_{\ell}} \abs{a^{(\ell)}_{j_{\ell},i_{\ell+1}} -
    b^{(\ell)}_{j_{\ell},i_{\ell+1}}}.
  \end{align}

  Is not hard to see that if we average the terms on the RHS, as in the $\Lnorm{\cdot}$, then each term can be bounded by $\Lnorm{a_N - b_N}$. This concludes the proof.
\end{proof}

We'll also need the following result:
\begin{lemma}[$L_N$ is Lipschitz with $\Lnorm{\cdot}$]
\label{lem:L_N-lipschitz}
There is a constant $C_{\ref{lem:L_N-lipschitz}}$ such that
\[
\sup_{\btheta_N\in\R^{p_N}}\abs{ L_N(\btheta_N)}\leq C_{\ref{lem:L_N-lipschitz}}\quad
\text{ and }\quad
\abs{L_N(\btheta_N) - L_N(\thetatil_N)}\leq C_{\ref{lem:L_N-lipschitz}}\Lnorm{\btheta_N - \thetatil_N}
\]
\end{lemma}
\begin{proof}
Recall that $L_N(\btheta_N) = \Exp{(X,Y)\sim P}{\abs{Y-\yhat(X,\btheta_N)}^2}$. According to Remark \ref{rem:bounded-yhat} and the Assumptions in \ref{assump:activations}, we see that $\yhat$ and $Y$ are $C$-bounded, implying
\[\abs{L_N(\btheta_N)}\leq 2 C.\]
To show the other inequality, consider:

\begin{align*}
\abs{L_N(\btheta_N) - L_N(\thetatil_N)}&\leq C_{\ref{lem:L_N-lipschitz}}\\
&=\abs{\Exp{(X,Y)\sim P}{2\dotprod{Y, \yhat(X,\btheta_N)-\yhat(X,\thetatil_N)} + \abs{\yhat(X,\btheta_N)}^2 - \abs{\yhat(X,\thetatil_N)}^2}}\\
& \leq 4C\cdot \Exp{(X,Y\sim P)}{\abs{\yhat(X,\btheta_N)-\yhat(X,\thetatil_N)}}
\end{align*}
Now we use observation (\ref{eq:lnorm-yhat}) to get:
\[
\abs{L_N(\btheta_N) - L_N(\thetatil_N)}\leq 4C\cdot L\;C^{L+2}\Lnorm{\btheta_N - \thetatil_N}
\]
Pick $C_{\ref{lem:L_N-lipschitz}} = \max\{4L\;C^{L+3}, 2C\}$ to finish lemma.
\end{proof}

%% file: measureNN_aux/A3-specialmeasures.tex
We prove here that the set $\Probspecial_R$ defined in \S \ref{sub:towardsfixedpoint} is closed under weak convergence.

\begin{proof}[Proof of Theorem \ref{thm:Probspecialisclosed}] Let $\Probspecial_R$ denote the set of $R$ special probability measures (cf. Definition \ref{def:Rspecialmeasure}. Consider a sequence $\{\mu_{j,[0,T]}\}_{j\in\N}\subset \Probspecial_R$. Each measure $\mu_j$ in the sequence has a corresponding $R$-special function $F_j:=F_{\mu_{j}}$, with the properties required by Definition \ref{def:Rspecialfunction}. 

Assume that the sequence $\{\mu_{j,[0,T]}\}_{j\in\N}$ converges to some $\mu_{[0,T]}\in \Probspace(C([0,T],\R^D))$. We wish to show that $\mu_{[0,T]}\in\Probspecial_R$. The fact that $\mu_0$ is a product and the $L$-th coordinate is a.s. constant under $\mu_{[0,T]}$ are straightforward to show, and we omit them. What is left, then, is to prove that there exists a $R$-special function $F=F_{\mu}$ (as per Definition \ref{def:Rspecialfunction}) which attests that $\mu$ is a special measure (as per Definition \ref{def:Rspecialmeasure}). 

We will use the following Claim.
\begin{claim}There exist a subsequence $\{j_k\}_{k\in\N}\subset \N$ a function $F:\R^{D_{L-1}}\times \R^{D_L}\to C([0,T],\R^{D_{L-1}})$ such that $F_{j_k}\to F$ uniformly over compact sets when $k\to +\infty$.\end{claim}

In what follows, we first show that the Claim implies that $\mu_{[0,T]}\in\Probspecial_R$, and then prove the Claim.

\underline{Proof that the Claim implies $\mu_{[0,T]}\in\Probspecial_R$.} By passing to a subsequence, we may assume that $F_j\to F$ uniformly over compact sets. We will use this to show that we may take $F_{}:=F$ in Definition \ref{def:Rspecialmeasure} to certify that $\mu_{[0,T]}$ is $R$-special. 

To do this, we first argue that $F$ is $R$-special. This is simple: inspection of Definition \ref{def:Rspecialfunction} reveals that the properties required of a $R$-special function are preserved under pointwise convergence, and thus also under uniform convergence over compact sets.

We now argue that $F$ attests that $\mu$ is $R$-special: that is, we will show that, if $\Theta\sim \mu_{[0,T]}$, then $\Theta_{[0,T]}^{(L-1)}=F(\Theta^{(L-1)}(0),\Theta^{(L)}(0))$ a.s.. 

The Skorohod representation theorem for weak convergence implies that we can find random elements $\Theta_{j,[0,T]}\sim \mu_{j}$ ($j\in \N$) and $\Theta_{[0,T]}\sim \mu$, defined over some common probability space, so that $\|\Theta_j - \Theta\|_\infty\to 0$ almost surely. This implies that:
\[(\Theta_j^{(L-1)}(0),\Theta_j^{(L)}(0),\Theta_{j,[0,T]}^{(L-1)})\to (\Theta^{(L-1)}(0),\Theta^{(L)}(0),\Theta^{(L-1)}_{[0,T]})\mbox { a.s.}.\]
Now, when $j\to +\infty$, the pair $(\Theta_j^{(L-1)(0)},\Theta_j^{(L)}(0))$ is a.s. convergent, and there a.s. is some (random) $R\in [0,+\infty)$ with: 
\[\sup_{j\in\N}|(\Theta_j^{(L-1)}(0) ,\Theta_j^{(L)}(0))|\leq R\mbox{ almost surely.}\]
This guarantees:
\begin{eqnarray*}& \limsup_j\|F_j(\Theta_j^{(L-1)}(0),\Theta_j^{(L)}(0)) -F(\Theta_j^{(L-1)}(0),\Theta_j^{(L)}(0))\|_{\infty}\\ \leq& (\limsup_j\sup_{|(a,b)|\leq R}\|F_j(a,b) - F(a,b)\|_{\infty})\\ =& 0\mbox{ a.s. because $F_j\to F$ uniformly over compacts.}\end{eqnarray*}
For each $j\in\N$, we know that $F_j(\Theta_j^{(L-1)}(0),\Theta_j^{(L)}(0)) = \Theta^{(L-1)}_{j,[0,T]}$ a.s. for each $j\in\N$. This means we may replace $F_j(\Theta_j^{(L-1)}(0),\Theta_j^{(L)}(0))$ with $\Theta^{(L-1)}_{j,[0,T]}$ in the previous display and obtain:  
\[\|\Theta_{j,[0,T]}^{(L-1)} -F(\Theta_j^{(L-1)}(0),\Theta_j^{(L)}(0))\|_{\infty} \to 0\mbox{ a.s..}\]
Using continuity of $F$ and the convergence of $\Theta_j$, we obtain:
\[\|\Theta_{[0,T]}^{(L-1)} -F(\Theta^{(L-1)}(0),\Theta^{(L)}(0))\|_{\infty}=0\mbox{ a.s.},\]
which is what we wished to show.

\underline{Proof of the Claim.} We will need the following form of the Ascoli-Arz\`{e}la Theorem.

\begin{theorem}[Ascoli-Arz\`{e}la Theorem] Suppose $(X,\rho_X)$ and $(Y,\rho_Y)$ are metric spaces and $M>0$. Assume that all closed balls in $X$ are compact. Assume $\sH=\{f_j\}_{j\in\N}$ is a sequence of equicontinuous mappings $f:X\to Y$ such that the sets
\[\sH(x):=\{f_j(x)\,:\,j\in\N\}\;\;(x\in X)\]
are precompact\footnote{Recall that a set in a topological space is precompact if its closure is compact, or equivalently, if it is contained in a compact set.} subsets of $Y$. Then there exists a subsequence $\{f_{j_k}\}_{k\in\N}$ and a continuous function $f:X\to Y$ such that $f_{j_k}\to f$ uniformly over compact sets of $X$ when $k\to +\infty$.\end{theorem}

This follows from the usual formulation of Ascoli-Arz\`{e}la for the case when $X$ is itself compact. We omit the details.

To apply the above Theorem, we take $X=\R^{D_{L-1}}\times \R^{D_L}$, $Y=C([0,T],\R^{D_{\ell-1}})$ and $\sH=\{F_j\}_{j\in\N}$. In order to deduce the Claim, we need to show that $\sH$ is equicontinous and that $\sH(x)$ is precompact for each $x\in X$.

We first show that $\sH$ is equicontinuous. Note that, if $a,b\in \R^{D_{L-1}}\times \R^{D_{L}}$ and $t\in [0,T]$, property 1 in the definition of special measure implies:
\[\|F_j(a) - F_j(b)\|_{\infty} = \sup_{0\leq t\leq T}|F_j(a)(t) - F_j(b)(t)|\leq \sup_{0\leq t\leq T}e^{Rt}|a-b|\leq e^{RT}|a-b|.\]
In particular, all $F_j$ are $e^{RT}$-Lipschitz mappings from the set $\R^{D_{L-1}}\times \R^{D_{L}}$ to $C([0,T],\R^{D_{L-1}})$.

To check precompactness of the images $\sH(x)$, we use Property 2 of Definition \ref{def:Rspecialmeasure}, which ensures that, for each $x=(a^{(L-1)},a^{(L)})\in \R^{D_{L-1}}\times \R^{D_L}$,
\[\sH(x)\subset U(x):=\{f\in C([0,T],\R^{D_{\ell-1}})\,:\, f(0) = a^{(L-1)}\mbox{ and }f\mbox{ is $R$-Lipschitz.}\}.\]
The set $U(x)$ is a closed subset of $C([0,T],\R^{D_{\ell-1}})$ that consists of uniformly bounded and equicontinuous functions. The Ascoli-Arz\`{e}la theorem in its classical form guarantees that $U(x)$ is a compact subset of $Y=C([0,T],\R^{D_{\ell-1}})$. This implies that $\sH(x)$ is precompact, which finishes the proof of the claim.\end{proof}

%% file: measureNN.bbl
\begin{thebibliography}{10}

\bibitem{Ambrosio2005}
L.~Ambrosio, N.~Gigli, and G.~Savare.
\newblock {\em Gradient Flows: In Metric Spaces and in the Space of Probability
  Measures}.
\newblock Lectures in Mathematics. ETH Z{\"u}rich. Birkh{\"a}user Basel, 2005.

\bibitem{Anthony2009}
Martin Anthony and Peter~L. Bartlett.
\newblock {\em Neural Network Learning: Theoretical Foundations}.
\newblock Cambridge University Press, New York, NY, USA, 1st edition, 2009.

\bibitem{belkin2018}
Mikhail Belkin, Daniel Hsu, and Partha Mitra.
\newblock Overfitting or perfect fitting? {R}isk bounds for classification and
  regression rules that interpolate.
\newblock In {\em Advances in Neural Information Processing Systems 31}, 2018.

\bibitem{Chizat2018}
Lenaic {Chizat} and Francis {Bach}.
\newblock {A Note on Lazy Training in Supervised Differentiable Programming}.
\newblock {\em arXiv e-prints}, page arXiv:1812.07956, Dec 2018.

\bibitem{Du2018.2}
Simon~S. {Du}, Jason~D. {Lee}, Haochuan {Li}, Liwei {Wang}, and Xiyu {Zhai}.
\newblock {Gradient Descent Finds Global Minima of Deep Neural Networks}.
\newblock {\em arXiv e-prints}, page arXiv:1811.03804, Nov 2018.

\bibitem{Du2018}
Simon~S. Du, Xiyu Zhai, Barnab{\'{a}}s P{\'{o}}czos, and Aarti Singh.
\newblock Gradient descent provably optimizes over-parameterized neural
  networks.
\newblock {\em CoRR}, abs/1810.02054, 2018.

\bibitem{Engel2001}
A.~Engel and C.~Van~den Broeck.
\newblock {\em Statistical Mechanics of Learning}.
\newblock Cambridge University Press, 2001.

\bibitem{Geiger2019}
Mario {Geiger}, Arthur {Jacot}, Stefano {Spigler}, Franck {Gabriel}, Levent
  {Sagun}, St{\'e}phane {d'Ascoli}, Giulio {Biroli}, Cl{\'e}ment {Hongler}, and
  Matthieu {Wyart}.
\newblock {Scaling description of generalization with number of parameters in
  deep learning}.
\newblock {\em arXiv e-prints}, page arXiv:1901.01608, Jan 2019.

\bibitem{Goodfellow2016}
Ian Goodfellow, Yoshua Bengio, and Aaron Courville.
\newblock {\em Deep Learning}.
\newblock MIT Press, 2016.
\newblock \url{http://www.deeplearningbook.org}.

\bibitem{goodfellow2014}
Ian~J. Goodfellow, Jean Pouget-Abadie, Mehdi Mirza, Bing Xu, David
  Warde-Farley, Sherjil Ozair, Aaron Courville, and Yoshua Bengio.
\newblock Generative adversarial nets.
\newblock In Zoubin Ghahramani, Max Welling, Corinna Cortes, Neil~D. Lawrence,
  and Kilian~Q. Weinberger, editors, {\em Advances in Neural Information
  Processing Systems 27 (NIPS 2014)}, pages 2672--2680. Curran Associates,
  Inc., 2014.

\bibitem{Jacot2018}
Arthur Jacot, Franck Gabriel, and Cl{\'e}ment Hongler.
\newblock Neural tangent kernel: Convergence and generalization in neural
  networks.
\newblock In {\em Proceedings of the 32Nd International Conference on Neural
  Information Processing Systems}, NIPS'18, pages 8580--8589, USA, 2018. Curran
  Associates Inc.

\bibitem{Kac1959}
M.~Kac.
\newblock {\em Probability and Related Topics in Physical Sciences}.
\newblock Lectures in applied mathematics (American Mathematical Society) ;
  1.A. American Mathematical Society, 1959.

\bibitem{Kolokoltsov2010}
Vassili~N. Kolokoltsov.
\newblock {\em Nonlinear Markov Processes and Kinetic Equations}.
\newblock Cambridge Tracts in Mathematics. Cambridge University Press, 2010.

\bibitem{Krizhevsky2017}
Alex Krizhevsky, Ilya Sutskever, and Geoffrey~E. Hinton.
\newblock Imagenet classification with deep convolutional neural networks.
\newblock {\em Commun. ACM}, 60(6):84--90, May 2017.

\bibitem{lecun2015}
Yann LeCun, Yoshua Bengio, and Geoffrey Hinton.
\newblock Deep learning.
\newblock {\em Nature}, 521:436 EP --, 05 2015.

\bibitem{McKean1967}
H.~P. McKean.
\newblock A class of markov processes associated with nonlinear parabolic
  equations.
\newblock {\em Proceedings of the National Academy of Sciences},
  56(6):1907--1911, 1966.

\bibitem{Mei2019}
Song {Mei}, Theodor {Misiakiewicz}, and Andrea {Montanari}.
\newblock {Mean-field theory of two-layers neural networks: dimension-free
  bounds and kernel limit}.
\newblock {\em arXiv e-prints}, page arXiv:1902.06015, Feb 2019.

\bibitem{Mei2018}
Song Mei, Andrea Montanari, and Phan-Minh Nguyen.
\newblock A mean field view of the landscape of two-layer neural networks.
\newblock {\em Proceedings of the National Academy of Sciences},
  115(33):E7665--E7671, 2018.

\bibitem{MeiArxiv2018}
Song {Mei}, Andrea {Montanari}, and Phan-Minh {Nguyen}.
\newblock A mean field view of the landscape of two-layers neural networks.
\newblock {\em arXiv e-prints}, page arXiv:1804.06561, Apr 2018.

\bibitem{Mezard2009}
Marc Mezard and Andrea Montanari.
\newblock {\em Information, Physics, and Computation}.
\newblock Oxford University Press, Inc., New York, NY, USA, 2009.

\bibitem{Nguyen2019}
Phan-Minh {Nguyen}.
\newblock {Mean Field Limit of the Learning Dynamics of Multilayer Neural
  Networks}.
\newblock {\em arXiv e-prints}, page arXiv:1902.02880, Feb 2019.

\bibitem{Rachev2006}
S.T. Rachev and L.~R{\"u}schendorf.
\newblock {\em Mass Transportation Problems: Applications}.
\newblock Probability and Its Applications. Springer New York, 2006.

\bibitem{Rahmi2008}
Ali Rahimi and Benjamin Recht.
\newblock Random features for large-scale kernel machines.
\newblock In J.~C. Platt, D.~Koller, Y.~Singer, and S.~T. Roweis, editors, {\em
  Advances in Neural Information Processing Systems 20}, pages 1177--1184.
  Curran Associates, Inc., 2008.

\bibitem{Rotskoff2018}
Grant~M. {Rotskoff} and Eric {Vanden-Eijnden}.
\newblock {Neural Networks as Interacting Particle Systems: Asymptotic
  Convexity of the Loss Landscape and Universal Scaling of the Approximation
  Error}.
\newblock {\em arXiv e-prints}, page arXiv:1805.00915, May 2018.

\bibitem{Sirignano2018}
Justin {Sirignano} and Konstantinos {Spiliopoulos}.
\newblock {Mean Field Analysis of Neural Networks}.
\newblock {\em arXiv e-prints}, page arXiv:1805.01053, May 2018.

\bibitem{Sirignano2018.2}
Justin {Sirignano} and Konstantinos {Spiliopoulos}.
\newblock {Mean Field Analysis of Neural Networks: A Central Limit Theorem}.
\newblock {\em arXiv e-prints}, page arXiv:1808.09372, Aug 2018.

\bibitem{Sirignano2019}
Justin {Sirignano} and Konstantinos {Spiliopoulos}.
\newblock {Mean Field Analysis of Deep Neural Networks}.
\newblock {\em arXiv e-prints}, page arXiv:1903.04440, Mar 2019.

\bibitem{Spohn2012}
Herbert Spohn.
\newblock {\em Large Scale Dynamics of Interacting Particles}.
\newblock Theoretical and Mathematical Physics. Springer Berlin Heidelberg,
  2012.

\bibitem{Sznitman1991}
Alain-Sol Sznitman.
\newblock Topics in propagation of chaos.
\newblock In Paul-Louis Hennequin, editor, {\em Ecole d'Et{\'e} de
  Probabilit{\'e}s de Saint-Flour XIX --- 1989}, pages 165--251, Berlin,
  Heidelberg, 1991. Springer Berlin Heidelberg.

\bibitem{Venturi2018}
Luca {Venturi}, Afonso~S. {Bandeira}, and Joan {Bruna}.
\newblock {Spurious Valleys in Two-layer Neural Network Optimization
  Landscapes}.
\newblock {\em arXiv e-prints}, page arXiv:1802.06384, Feb 2018.

\bibitem{Zhang2016}
Chiyuan {Zhang}, Samy {Bengio}, Moritz {Hardt}, Benjamin {Recht}, and Oriol
  {Vinyals}.
\newblock {Understanding deep learning requires rethinking generalization}.
\newblock {\em arXiv e-prints}, page arXiv:1611.03530, Nov 2016.

\end{thebibliography}
